\documentclass[12pt]{amsart}
\usepackage[dvips]{color}
\usepackage{amsmath}
\usepackage{amsxtra}
\usepackage{amscd}
\usepackage{amsthm}
\usepackage{amsfonts}
\usepackage{amssymb}
\usepackage{eucal}
\usepackage{epsfig}
\usepackage{graphics}
\usepackage{graphicx}
\usepackage{accents}
\usepackage{dynkin-diagrams}
\usepackage{comment}

\usepackage{mathtools}
\usepackage{tikz-cd}
\usepackage{mathtools}
\usepackage{pgf}
\usetikzlibrary{arrows, matrix}

\usepackage{here} 
\usepackage{simplewick}
\numberwithin{equation}{section}
\newtheorem{thm}{Theorem}[section]
\newtheorem{prop}[thm]{Proposition}
\newtheorem{lem}[thm]{Lemma}
\newtheorem{rem}[thm]{Remark}
\newtheorem{cor}[thm]{Corollary}
\newtheorem{conj}[thm]{Conjecture}

\theoremstyle{definition}
\newtheorem{example}[thm]{Example}
\newtheorem{remark}[thm]{Remark}

\newcommand{\C}{{\mathbb C}}
\newcommand{\Z}{{\mathbb Z}}
\newcommand{\Q}{{\mathbb Q}}

\newcommand{\cA}{{\mathcal A}}

\newcommand{\E}{{\mathcal E}}
\newcommand{\F}{\mathcal F}
\newcommand{\cH}{\mathcal H}
\newcommand{\cK}{\mathcal{K}}

\newcommand{\cW}{\mathcal{W}}

\newcommand{\Rc}{\check{R}}

\newcommand{\bs}{\boldsymbol}

\newcommand{\gl}{\mathfrak{gl}}

\newcommand{\sln}{\mathfrak{sl}}

\newcommand{\bI}{{\mathbf{I}}}
\newcommand{\Eb}{\bs E}

\newcommand{\La}{\Lambda}

\newcommand{\ssa}{\textsf{a}}
\newcommand{\sss}{\textsf{s}}
\newcommand{\ssv}{\textsf{v}}

\newcommand{\ssy}{\textsf{y}}

\newcommand{\ssA}{\textsf{A}}
\newcommand{\ssB}{\textsf{B}}
\newcommand{\ssC}{\textsf{C}}
\newcommand{\ssD}{\textsf{D}}
\newcommand{\ssCD}{\textsf{CD}}

\newcommand{\ssG}{\textsf{G}}
\newcommand{\ssK}{\textsf{K}}

\newcommand{\ssN}{\textsf{N}}
\newcommand{\ssR}{\textsf{R}}
\newcommand{\ssT}{\textsf{T}}

\newcommand{\End}{\mathop{\rm End}}

\newcommand{\id}{{\rm id}}

\newcommand{\FF}{\mathbb{F}}

\newcommand{\mc}{\mathcal}

\newcommand{\U}{U}

\newcommand{\cont}[2]{\contraction[1ex]{}{#1}{}{#2} #1 #2}

\newcommand{\ft}{\tilde{f}}

\renewcommand{\Rc}{\check{\ssR}}

\newcommand{\kfun}{\omega}
 
\begin{document}
\date{\today}

\begin{title}[Deformations of $\cW$ algebras]
{Deformations of $\cW$ algebras \\
via quantum toroidal algebras}
\end{title}
\author{B. Feigin, M. Jimbo, E. Mukhin, and I. Vilkoviskiy}
\address{BF: National Research University Higher School of Economics,  101000, Myasnitskaya ul. 20, Moscow,  Russia, and Landau Institute for Theoretical Physics, 142432, pr. Akademika Semenova 1a, Chernogolovka, Russia
}
\email{bfeigin@gmail.com}
\address{MJ: Department of Mathematics,
Rikkyo University, Toshima-ku, Tokyo 171-8501, Japan}
\email{jimbomm@rikkyo.ac.jp}
\address{EM: Department of Mathematics,
Indiana University Purdue University Indianapolis,
402 N. Blackford St., LD 270, 
Indianapolis, IN 46202, USA}
\email{emukhin@iupui.edu}
\address{IV: Center for Advanced Studies,
Skolkovo Institute of Science and Technology, 
1 Nobel St., 143026, Moscow, Russia, and
National Research University Higher School of Economics, 
Russian Federation,  101000,  
Myasnitskaya ul. 20,  Moscow,  Russia}
\email{reminguk@gmail.com}


\begin{abstract} 
We study the uniform description of deformed $\cW$ algebras of type $\ssA$ including the supersymmetric case 
in terms of the quantum toroidal $\mathfrak{gl}_1$ algebra $\E$. In particular, we recover the deformed 
affine Cartan matrices and the deformed integrals of motion.

We introduce a comodule algebra $\cK$ over $\E$ which gives a uniform construction of basic deformed 
$\cW$ currents and screening operators in types $\ssB,\ssC,\ssD$ including twisted and supersymmetric cases. 
We show that a completion of algebra $\cK$ contains three commutative subalgebras. In particular, it allows 
us to obtain a commutative family of integrals of motion associated with affine Dynkin diagrams of all 
non-exceptional types except $\ssD^{(2)}_{\ell+1}$. We also obtain in a uniform way deformed finite and  
affine Cartan matrices in all classical types together with a number of new examples, and discuss the 
corresponding screening operators.
\end{abstract}

\keywords{Quantum toroidal algebra; $W$ algebra; integrals of motion; $qq$ character.}

\maketitle

\section{Introduction}
Commutative families of operators coming from conformal field theory (CFT), 
known as local integrals of motion (IM), have attracted a lot of attention in the last quarter of a century. 
The interest was boosted by a seminal sequence of papers 
by Bazhanov, Lukyanov, and Zamolodchikov 
\cite{BLZ1}--\cite{BLZ3} where many related structures were revealed and a number of intriguing 
conjectures was put forth. One outcome of that study is the celebrated ODE/IM correspondence relating the spectrum of local IM to remarkable differential operators whose monodromy coefficients satisfy Bethe ansatz equations \cite{DT}, \cite{MRV}, \cite{BLZ4}, \cite{BLZ5}, see also \cite{BL2}. However, many important questions in the area still have no answers and many key statements remain conjectures. The explicit form of the local integrals of motion is unknown even in the simplest cases, see \cite{FF} for the discussion of their existence, and there are no theorems about their spectrum. 

From a general philosophy of quantum groups, it is quite  natural to consider $q$-deformations of local 
IM in search for clarification of the matters. It turns out that after the $q$-deformation, 
the type $\ssA$ local IM become non-local, but can be written down explicitly \cite{FKSW}. 
Moreover, one can describe their spectrum via Bethe ansatz, \cite{FJM}. 
In the present article we address the problem of constructing the $q$-deformation of local IM in other types.

\medskip

The algebraic structure underlying the CFT in question is given by the Virasoro algebra, and more generally, 
$\cW$ algebras. The $q$-deformations of $\cW$ algebras have been provided in \cite{AKOS} for type $\ssA$, and in \cite{FR1} for simple Lie algebras. It has been recognized in recent years that quantum groups provide a natural framework for studying $\cW$ algebras \cite{MO}.

\medskip

For the deformed $\cW$ algebras of type $\ssA$, the relevant quantum group is the quantum toroidal $\mathfrak{gl}_1$ algebra $\E$ (also known as the Ding-Iohara-Miki algebra), see Section \ref{sec E}. Algebra $\E$ has three Fock modules depending on a color $c\in\{1,2,3\}$, see Section \ref{sec Fock}. Consider generating current $e(z)$ of $\E$ acting on a tensor product of $\ell+1$ Fock modules associated to arbitrary colors $c_1,\dots, c_{\ell+1}$. Then $e(z)$ acts as a sum of $\ell+1$ vertex operators, see Section \ref{sec roots A}. When all Fock modules are of the same kind, the current $e(z)$ turns out to be related to the basic current $T(z)$ of the deformed $\cW$ algebra of type $\ssA_\ell$, as $e(z)=T(z)Z(z)$, where $Z(z)$ is a vertex operator written in a single boson $\{h_i\}$ given by the action of the Cartan current of $\E$ and which commutes with $T(z)$. For other choices of colors, see for example \eqref{example fock}, one gets currents which can be viewed as various $q$-deformations of $\cW$ algebras associated to Lie superalgebras of types $\mathfrak{gl}_{m|n}$ with $m+n=\ell+1$. 
In particular, we introduce the currents $A_i(z)$, $i=1,\ldots,\ell$, given by ratios of neighboring terms in this sum of $\ell+1$ vertex operators, see \eqref{A}. The current $A_i(z)$ is bosonic if $c_{i}=c_{i+1}$, and it is fermionic otherwise. Then the table of contractions between the $A_i(z)$'s gives the deformed Cartan matrix of finite type \eqref{A Cartan}. Moreover, the screening operators which are integrals of $q$-primitives of $A_i(z)$, see \eqref{Spm},\eqref{Sf}, commute with $T(z)$.

Algebra $\E$ has a family of commutative algebras depending on a parameter $\mu$ given by the transfer matrices. One can  ``dress" current $e(z)$ multiplying by an appropriate combination of Cartan current depending on $\mu$, see \eqref{dress}. Then
generators of the commutative algebras can be computed explicitly and are given by multiple integrals \eqref{integrals} of the dressed current $\bs e(z)$ with Feigin-Odesskii \cite{FO} kernel, \eqref{kfun}, see \cite{FJM}. These generators acting on tensor products of Fock spaces 
are deformations of the local IM associated to $\cW$ of type $\ssA$.
The spectrum of transfer matrices (or of deformed IM), is computed by Bethe ansatz, see \cite{FJMM1}, \cite{FJMM2} for two different derivations. 
Then one can attempt to obtain the spectrum of the original local IM by taking an appropriate limit, see \cite{FJM}.

The dressed current $\bs e(z)$ has the form $e(z)\tilde Z_\mu(z)$ where $\tilde Z_\mu(z)$ is a vertex operator written in terms of $\{h_i\}$. Then we observe that 
the ratio $A_0(z)$ of the last term in the dressed current $\bs e(z)$ and the first term shifted by $\mu$  produces one more screening operator which commutes with the deformed IM. Moreover, the matrix of contractions between all of $A_i(z)$ is given by the deformed Cartan matrix of affine type as in \cite{KP}.  Following the terminology of \cite{N}, \cite{KP}, we interpret the dressed current $\bs e(z)$ as a $qq$-character of the first fundamental representation of quantum affine $\sln_{\ell+1}$, see Section \ref{q char sec}.

\medskip

Then we develop a similar picture for non-exceptional types other than $\ssA$. 

We introduce an algebra $\cK$ depending on three parameters $q_1,q_2,q_3$ such that $q_1q_2q_3=1$,
which plays a role of quantum toroidal algebra, see \eqref{EE}-\eqref{EEE}.
Unlike $\E$, the algebra $\cK$ is not a Hopf algebra but a left $\E$ comodule, see \eqref{comodule formula}. 
Algebra $\cK$ has 6 representations in one boson denoted  $\F^{\ssB}_c$, see Proposition \ref{prop B boundary}, 
and $\F_c^{\ssC\ssD}$, $c\in\{1,2,3\}$, see Proposition \ref{prop CD boundary}, which we call boundary 
Fock modules of types $\ssB$  and $\ssC\ssD$, respectively. 
The algebra $\cK$ has a central element $C$ which determines the level of the module. 
We have $C^2=q_c^{1/2}$ in type $\ssB$ and  $C^2=q_c^{-1}$ in type $\ssC\ssD$.

The comodule structure allows us to consider a tensor product of a boundary Fock module of color $c_{\ell+1}$ with $\ell$ Fock modules of $\E$ of colors $c_1,\dots, c_\ell$. This $\cK$ module is realized in $\ell+1$ bosons.
The generating current $E(z)$ of $\cK$ acting on such a tensor product is a sum of $2\ell+1$ terms if the the boundary module is of type $\ssB$ and of
$2\ell$ vertex operators if the boundary module is of type $\ssC\ssD$. Similarly to type $A$, if all colors of $\E$ Fock modules are the same, $c_1=\dots=c_\ell$, 
then up to a boson we recover the deformed $\cW$ currents of \cite{FR1}. 
More precisely, if the boundary module is of type $\ssB$, we obtain a 
deformed $\cW$ current of type $\ssB$ when $c_{\ell+1}\neq c_\ell$, and of 
type ${\mathfrak{osp}}(1,2\ell)$ when $c_{\ell+1}=c_\ell$ (the latter is called $\ssA_{2\ell}^{(2)}$ type 
in \cite{FR1}). If the boundary module is of type $\ssC\ssD$, we recover the deformed $\cW$ current of type 
$\ssC$ when $c_{\ell+1}\neq c_\ell$, and of type $\ssD$ when $c_{\ell+1}= c_\ell$. 
If the colors of $\E$ Fock spaces vary, we get various $q$-deformations of $\cW$ currents related to supersymmetric cases. 
We again introduce currents $A_i(z)$ as ratios of neighboring terms (in terms of Dynkin diagram), 
see \eqref{A-i}-\eqref{D-ell}. 
Then we observe that the contractions give a deformed Cartan matrix of corresponding finite type, 
and that the screening operators constructed from $A_i(z)$ commute with our $E(z)$.

We define a dressed current $\Eb(z)$ and consider integrals of products of $\Eb(z)$ with the Feigin-Odesskii kernel similarly to type $\ssA$, see \eqref{bIn}. We show that these integrals commute if $C^2/\mu=q_{c_0}$, $c_0\in\{1,2,3\}$. Thus we have three families of commutative subalgebras in $\cK$ and three commuting families of operators acting on the representation. Naturally, we call them deformed IM. Note that the simplest deformed IM is always given by the constant term of the deformed $\cW$ current $\bs E(z)$.

We use dressed current $\bs E(z)$ to introduce current $A_0(z)$, see \eqref{B-0}-\eqref{D-0}, and observe that the contractions of $A_i(z)$ give deformed affine Cartan matrices. If $c_0\neq c_1$ then the affine root turns out to be of type $\ssC$ and if $c_0=c_1$ then of type $\ssD$. In particular, we cover that way all non-exceptional affine Dynkin diagrams except $\ssD^{(2)}_{\ell+1}$. 
Moreover, the screening operator constructed from $A_0(z)$ commutes with the deformed IM.
 
It is then natural to consider affine roots of type $\ssB$. 
It corresponds to the case when dressing parameter $\mu$ satisfies  $C^2/\mu=q_{c_0}^{-1/2}$. 
In this case we introduce current $A_0(z)$ in a similar way. 
Adding one more vertex operator to $\bs E(z)$ (that is  ``adding a one dimensional representation") 
we obtain a new current $\tilde\Eb(z)$. Then  $\tilde \Eb(z)$ commutes with screening operators constructed from $A_i(z)$, $i=1,\dots,\ell$, and the integral of $\tilde \Eb(z)$  commutes with the screening operator corresponding to $A_0(z)$. In particular, contractions of $A_i(z)$ lead to deformed affine Cartan matrices whose affine node is of type $\ssB$.

We do not know know how to include the integral of $\tilde \Eb(z)$ into a family of commuting operators directly.
However if the boundary module is of type $\ssC\ssD$, we can exchange the affine node with the $\ell$-th node. It turns out that the integral of $\tilde \Eb(z)$ coincides with the integral of a different deformed $\cW$ current for which the boundary module is of type $\ssB$ and the affine root is of types $\ssC$ or  $\ssD$, see Remark \ref{rem B}. Then the integral of the latter current is a part of the family of integrals of motion as before. Such a recipe does not work for $\ssD^{(2)}_{\ell+1}$ for which the end nodes are both of type $\ssB$.

The full list of deformed affine  Cartan matrices produced by our construction includes a number of new examples, 
see Appendix \ref{library}.

\bigskip

There are several questions  arising from our work which we plan to address in the future publications.

\begin{itemize}
\item The deformed integrals of motion related to $\ssD^{(2)}_{\ell+1}$ are not constructed yet. The affine roots of type $\ssB$ need an additional study in general.

\item We introduced the commuting algebras in $\cK$ as explicit integrals. We would like to construct these algebras from some version of transfer matrices and obtain their spectrum by a Bethe ansatz method. 

\item The nature of $\cK$ algebra needs to be clarified. We expect that it should be recognized as a twisted version of the quantum toroidal algebra. Also, it is interesting to understand the relation of $\cK$ to the universal $\cW$ algebras of types $\ssB$ and $\ssC,\ssD$ of \cite{KL}.

\item There are similar comodule algebras for quantum toroidal algebras $\E_n$ associated to $\mathfrak{gl}_n$. It is important to study the currents they produce and the corresponding integrals of motion.

\item The deformations of integrals of motion related to exceptional types seem to be related to other quantum algebras, see Section \ref{sec G2} for a discussion of type $G_2$.
\end{itemize}

\bigskip

The plan of the paper is as follows. 

In Section \ref{section E} we introduce our convention and review generalities 
about the quantum toroidal $\mathfrak{gl}_1$ algebra $\E$. 

Section \ref{section A} is devoted to deformed $\cW$ algebras and deformed integrals of motion of type $\ssA$.  Section \ref{sec roots A} reviews the non-affine case. In Section \ref{affine root A sec} we introduce the zeroth root current and recover deformed affine Cartan matrices. We discuss $qq$-characters in Section \ref{q char sec}, screenings in Section \ref{sec A screenings}, and the deformed IM in Section \ref{int A sec}.

In Section \ref{section BCD} we treat the case of other classical types.
We introduce the algebra $\cK$ (Section \ref{sec K}), along with its left $\E$ comodule structure (Section \ref{sec comodule})
and the boundary  
representations $\F^\ssB_c,\F^\ssCD_c$ (Section \ref{sec boundary}). We construct the deformed $\cW$ currents and obtain deformed finite type Cartan matrices in Section \ref{sec B roots}.  We discuss the affinization in Section \ref{sec B affine}. Section \ref{sec qq B} deals with $qq$-characters, Section \ref{sec B screenings} with screenings, and Section \ref{sec B integrals} with the deformed commuting integrals.

In Section \ref{section suppl} we make some additional remarks. We discuss
integrals of motion of Knizhnik-Zamolodchikov type in Section \ref{sec KZ} and  the situation in type $\mathsf{G}_2$ in Section \ref{sec G2}.

The text is followed by three appendices.
In Appendix \ref{proof app}, we give a proof of the commutativity of integrals of motion. Appendix \ref{sec:Kmat} discusses the existence of operator $\ssK(u)$.
In Appendix \ref{library} we give a library of deformed Cartan matrices obtained from $\cK$.

\section{The quantum toroidal algebra associated to $\gl_1$}\label{section E}
In this section we recall the quantum toroidal algebra associated to $\gl_1$ and its Fock modules.
\subsection{Relations} \label{sec E}
Fix $q_1,q_2,q_3\in\C^{\times}$ such that $q_1q_2q_3=1$. We also fix values of $\ln q_i$ 
and set $q_i^a=\exp(a\ln q_i)$ for all $a\in\C$. 
In this paper we assume that our choice of parameters is generic,  
meaning that $q_1^aq_2^b=1$ for $a,b\in\Z$ if and only if $a=b=0$. 

We use the notation
$$
s_c=q_c^{1/2}, \qquad t_c=s_c-s_c^{-1}, \qquad b_c=\frac{t_c}{t_1t_2t_3}\,
\qquad (c\in\{1,2,3\}).
$$
We also use
$$
g(z,w)=\prod_{i=1}^3(z-q_iw), \quad \bar{g}(z,w)=\prod_{i=1}^3(z-q_i^{-1}w)=-g(w,z)\,,
\qquad \kappa_r=\prod_{i=1}^3(1-q_i^r)\quad  (r\in\Z).
$$
Note that $\kappa_1=-t_1t_2t_3=\sum_{i=1}^3q^{-1}_i-\sum_{i=1}^3q_i$.

The quantum toroidal algebra $\E$ associated to $\gl_1$ is an associative algebra with parameters $q_1,q_2,q_3$
generated by coefficients of the currents
\begin{align*}
&e(z)=\sum_{n\in\Z}e_nz^{-n} \,,\quad f(z)=\sum_{n\in\Z}f_nz^{-n} \,,\quad
\psi^\pm(z)=\exp\bigl(\sum_{r>0}\kappa_r h_{\pm r}z^{\mp r}\bigr)\,,
\end{align*}
and an invertible 
central element\footnote{In the standard definition,  
the quantum toroidal algebra $\E$ associated to $\gl_1$ has two central elements. 
Here we set the second central element to $1$.}  $C$.
The defining relations are as follows.
\begin{align*}
&\psi^\pm(z)\psi^\pm(w) = \psi^\pm(w)\psi^\pm (z), 
\\
&\psi^+(z)\psi^-(w)
=
\psi^-(w)\psi^+ (z) 
\frac{g(z,C w)}{\bar{g}(z,Cw)}
\frac{\bar{g}(Cz,w)}{g(C z,w)}\,,
\\
&g(z,w)\psi^\pm(C^{(-1\mp1)/2}z)e(w)
=\bar{g}(z,w)e(w)\psi^\pm(C^{(-1\mp1) /2}z),
\\
&\bar{g}(z,w)\psi^\pm(C^{(-1\pm1)/2}z)f(w)=g(z,w)f(w)\psi^\pm(C^{(-1\pm1)/2}z)\,,
\\
&[e(z),f(w)]=
\frac{1}{\kappa_1}
(\delta\bigl(\frac{Cw}{z}\bigr)\psi^+(w)
-\delta\bigl(\frac{Cz}{w}\bigr)\psi^-(z)),\\
&g(z,w)e(z)e(w)=\bar{g}(z,w)e(w)e(z)\,, \\
&\bar{g}(z,w)f(z)f(w)=g(z,w)f(w)f(z)\,,\\
&\mathop{\mathrm{Sym}}_{z_1,z_2,z_3}z_2z_3^{-1}
[e(z_1),[e(z_2),e(z_3)]]=0\,,\\
&\mathop{\mathrm{Sym}}_{z_1,z_2,z_3}z_2z_3^{-1}
[f(z_1),[f(z_2),f(z_3)]]=0\,,
\end{align*}
where
\begin{align*}
&\mathop{\mathrm{Sym}}_{z_1,\ldots,z_N}\ F(z_1,\dots,z_N) =\frac{1}{N!}
\sum_{\pi\in\mathfrak{S}_N} F(z_{\pi(1)},\dots,z_{\pi{(N)}})\,.
\end{align*}

The relations for $\psi^\pm(z)$ are equivalent to
\begin{align}\label{hh}
[h_r,h_{s}]=\delta_{r+s,0}\frac{1}{\kappa_r}\frac{C^{r}-C^{-r}}{r} \,.
\end{align}

\medskip 

Let $\mc A$ be a $\Z$ graded algebra with a central element $C$.
The completion of $\mc A$ in the positive direction is the algebra $\tilde{\mc A}$, linearly spanned by products 
of series of the form $\sum_{i=M}^\infty f_ig_i$, where $M\in \Z$, $f_i,g_i\in\mc A$ and $\deg g_i=i$. 

We call an $\mc A$ module $V$ admissible if $V$ is $\Z$ graded with degrees bounded from above, i.e., 
$V=\oplus_{i=-\infty}^{N} V_{i}$, where $V_i=\{v\in V\mid \deg v=i\}$, 
and if $C$ is diagonalizable. The completion $\tilde {\mc A}$ acts on all admissible modules.

Algebra $\E$ has a $\Z$ grading given by 
$$
\deg e_i=\deg f_i=\deg h_i=i\,,\quad \deg C=0\,.
$$
In other words, if we formally set $\deg z=1$, then $e(z)$, $f(z)$, $\psi^\pm(z)$ all have degree zero.

Let $\E\tilde \otimes \E$ be the tensor algebra $\E\otimes \E$ completed in the positive direction. We use the topological coproduct $\Delta: \E\to \E\tilde \otimes \E$ 
as in \cite{FJM}:
\begin{align}
&\Delta e(z)=e(C_2^{-1}z)\otimes \psi^+(C_2^{-1}z)+1\otimes e(z)\,,\label{e coproduct}
\\
&\Delta f(z)=f(z)\otimes 1+ \psi^-(C_1^{-1}z)\otimes f(C_1^{-1}z)\,,\notag
\\
&\Delta\psi^+(z)=\psi^+(z)\otimes \psi^+(C_1z)\,,\label{psi+ coproduct}
\\
&\Delta\psi^-(z)=\psi^-(C_2z)\otimes \psi^-(z)\,,\notag
\end{align}
where $C_1=C\otimes1$, $C_2=1\otimes C$. 

Note that the coproduct can be extended to the map $\Delta: \tilde \E\to\E\tilde \otimes \E$.

The counit map $\epsilon:\E\to\C$ 
is given by $\epsilon\bigl(e(z)\bigr) =\epsilon\bigl(f(z)\bigr) =0$,  $\epsilon\bigl(\psi^\pm(z)\bigr) =1$, 
$\epsilon\bigl(C\bigr)=1$.

We set also
\begin{align*}
\tilde{f}(z)=S\bigl(f(z)\bigr)=- {\psi^-(z)}^{-1}f(C z)
\,
\end{align*}
where $S$ is the antipode. The coproduct reads
\begin{align*}
\Delta\ft(z)=\ft(C_2z)\otimes{\psi^-(z)}^{-1}+1\otimes\ft(z)\,.
\end{align*}

Note that the topological Hopf algebra $\E$ depends on $q_1,q_2,q_3$ symmetrically, in other words it depends on the unordered set $\{q_1,q_2,q_3\}$.

\medskip

\subsection{Vertex operators}
In this paper we often study vertex operators acting on bosonic Fock spaces. Here we give a brief description of these objects.

A Heisenberg algebra $\mc H$ is an associative algebra generated by linearly independent elements 
$h_{r}^{(i)}$, where $r\in\Z$, $i=1,\dots,N$, 
with relations $[h_{r}^{(i)},h_{s}^{(j)}]=\delta_{r+s,0}\,a_{r}^{ij}$, where $a_r^{i,j}\in\C$. 
Algebra $\mc H$ has a $\Z$ grading such that $\deg h_{r}^{(i)}=r$.

A weight $\bs\alpha$ is a linear functional $\bs\alpha:\textrm{span}_\C\{h_0^{(1)},\dots,h_0^{(N)} \}\to \C$. 
For a weight $\bs\alpha$, we define the Fock space $\FF_{\bs\alpha}$
as an $\mc H$ module generated by a vacuum vector $v_{\bs\alpha}$ such that $h^{(i)}_r v_{\bs\alpha}=0$ for $r>0$,  
$h^{(i)}_0v_{\bs\alpha}={\bs\alpha}(h^{(i)}_0)v_{\bs\alpha}$ for $i=1,\dots,N$, and such that 
$\FF_{\bs\alpha}$ is freely generated over $\C[h_r^{(i)}]_{r\in\Z_{<0}}^{i=1,\dots,N}$. 

Given  a weight $\bs\alpha$, let 
$e^{Q_{\bs\alpha}}: \FF_{\bs \mu}\to \FF_{\bs \mu+\bs\alpha}$ 
be a linear operator such that $e^{Q_{\bs\alpha}} v_{\bs\mu}=v_{\bs\mu+\bs\alpha}$ 
and such that $[e^{Q_{\bs\alpha}},h_r^{(i)}]=0$, $r\neq 0$. We have
$$
e^{Q_{\bs\alpha}}z^{\ssv_0}= z^{-{\bs\alpha}(\ssv_0)} z^{\ssv_0}e^{Q_{\bs\alpha}}\,,
\qquad \ssv_0\in \textrm{span}_\C\{h_0^{(1)},\dots,h_0^{(N)} \}\,.
$$

\medskip

In this paper, by a vertex operator we mean a formal series $V(z)$ 
of the form 
\begin{align}\label{VO}
V(z)=b\exp\big(\sum_{r>0}\ssv_{-r}z^r\big)\exp\big(\sum_{r\geq  0}\ssv_rz^{-r}\big), 
\qquad  \ssv_r\in \textrm{span}_\C\{h_r^{(1)},\dots,h_r^{(N) }\}\quad  (r\in\Z),
\end{align}
where $b\in\C^\times$. For any $v\in\FF_{\bs\mu}$, $V(z)v$ is a well defined Laurent series in $z$ with values in  $\FF_{\bs\mu}$.

The product of \eqref{VO} with another  vertex operator 
$V'(w)=b'\exp(\sum_{r>0} \ssv_{-r}'w^r)\exp(\sum_{r\geq 0} \ssv_r' w^{-r})$ has the form
\begin{align*}
&V(z)V'(w)=\varphi_{V,V'}(w/z):V(z)V'(w):\,,
\end{align*}
where $\varphi_{V,V'}(w/z)\in\C[[w/z]]$ is a formal power series called the contraction of $V(z)$ and $V'(w)$, and  the vertex operator
\begin{align*}
&:V(z)V'(w):\,=bb'\exp\big(\sum_{r>0}( \ssv_{-r}z^r+\ssv_{-r}'w^r)\big)
\exp\big(\sum_{r\geq 0} (\ssv_r z^{-r}+\ssv_r'w^{-r})\big)
\end{align*}
is called the normal ordered product of $V(z)$ and $V'(w)$. Obviously  $:V(z)V'(w):\,\,=\,\,:V'(w)V(z):$.
\medskip

Our vertex operators will depend on various parameters, $q_1$, $q_2$, $\mu$, $C$, etc. One can think 
of these parameters as variables or as generic complex numbers.
We call $\alpha$ a formal monomial if $\alpha$ is a product of parameters in rational powers. 
For example $\alpha$ can be of the form $q_1^{a_i}q_2^{b_i}$ with $a_i,b_i\in\Q$.

We often study the case when the contraction has the form 
\begin{align}
\varphi_{V,V'}(w/z)=\prod_i(1-\alpha_iw/z)^{-n_i}
\label{contVV}
\end{align}
where $n_i\in\Z$ and  $\alpha_i$ are formal monomials. Equation \eqref{contVV} is equivalent to saying
$$
[\ssv_r,\ssv_s']=\delta_{r+s,0}\ \frac{ \sum_{i}n_i\alpha_i^r}{r} \qquad (r\in\Z_{>0}).
$$
We use the notation
$$
\mathcal{C}\bigl(V(z),V'(w)\bigr) = \sum_{i}n_i\alpha_i
$$
to represent \eqref{contVV}, and by abusing the language we also call the sum in the right hand side the contraction of $V(z)$ and $V'(w)$. 

Note that $\mathcal{C}\bigl(V(pz),V'(p'w)\bigr)=\mathcal{C}\bigl(V(z),V'(w)\bigr)\cdot p'/p $.

We call a contraction rational if $\sum_in_i\alpha_i$ has a finite number of summands.
We call a contraction elliptic if we can write $\sum_in_i\alpha_i=\sum_jm_j\beta_j/(1-\beta)$ where $m_j\in \Z$, $\beta_i,\beta$ are formal monomials, and the summation over $j$ is finite. 
\medskip

We also often study screening currents  and screening operators.
Screening currents $S(z)$ have the form
\begin{align}\label{screening current}
S(z)=
e^{Q_{\bs\alpha}} z^{\sss_0} \exp\big(\sum_{r>0}\sss_{-r}z^r\big)\exp\big(\sum_{r>0}\sss_rz^{-r}\big), 
\qquad  \sss_r\in \textrm{span}_\C\{h_r^{(1)},\dots,h_r^{(N) }\}\  (r\in\Z),
\end{align}
where 
$\bs\alpha$ is a weight.  In particular, $S(z)=e^{Q_{{\bs\alpha}}} z^{\sss_0}S^{osc}(z)$ 
where $S^{osc}(z)$ is a vertex operator with $\bar \sss_0=0$.

Given a screening current $S(z)=e^{Q_{{\bs\alpha}}} z^{\sss_0}S^{osc}(z)$ and a vertex operator $V(w)$, the normal ordering is given by
$:S(z)V(w):\,\,=\,\,:V(w)S(z):\,\,=e^{Q_{{\bs\alpha}}} z^{\sss_0} :S^{osc}(z) V(w):$.

The screening operator $S$ is the constant term of $zS(z)$:
\begin{align}
S=\frac{1}{2\pi i}\int S(z)\, dz\,.
\label{Sc-op}
\end{align}
Note that $S$ is a well defined operator $\FF_{\bs\mu} \to \FF_{\bs\mu+\bs\alpha}$ if and only if $\bs\mu(s_0)\in\Z$.

\medskip

Let $V(z)$, $V'(z)$ be vertex operators of the form \eqref{VO}, and set 
\begin{align*}
A(z)=\,\,:V(z)V'(z)^{-1}:.
\end{align*}
Assume that there are formal monomials $p_1,p_2,p_3$ such that
\begin{align}\label{condition}
&\mathcal{C}\bigl(A(z),V(w)\bigr)=-(1-p_1^2)(1-p_2^2)\,,
\quad 
\mathcal{C}\bigl(V(w),A(z)\bigr)=-(1-p_1^{-2})(1-p_2^{-2})\,,
\\
&\mathcal{C}\bigl(A(z),V'(w)\bigr)=(1-p_2^{-2})(1-p_3^{-2})\,,
\quad 
\mathcal{C}\bigl(V'(w),A(z)\bigr)=(1-p_2^2)(1-p_3^2)\,.\notag
\end{align}
For these equations to be consistent, we must have $(1-p_2^2)(p_1^2-p_3^2)(1-p_1^{-2}p_2^{-2}p_3^{-2})=0$. 
Excluding the trivial case $p_2^2=1$, we shall assume that either $p_1=p_3$ or $p_1p_2p_3=1$.
We also assume that if $v_0=a v_0'$ for some $a\in\C$ then $a=-\log p_1/\log p_3$.

Then we can construct a screening operator commuting with $V(z)+V'(z)$ as follows.

We define the screening current $S(z)$ of the form \eqref{screening current} by 
\begin{align}\label{screening}
A(z)=\frac{1-p_3^2}{1-p_1^2}p_2^{-2}p_3^{-2}:S(p_2^{-1}z)S(p_2z)^{-1}:.
\end{align}
This amounts to the relations
\begin{align*}
&\ssa_r=\ssv_r-\ssv'_r=(p_2^r-p_2^{-r})\sss_r \qquad (r\neq 0),\\
&
e^{\ssa_0}=e^{\ssv_0-\ssv_0'}
=\frac{b'}{b}\, \frac{1-p_3^2}{1-p_1^2}p_2^{-2}p_3^{-2}p_2^{-2\sss_0}.
\end{align*}
In addition we choose a weight $\bs\alpha$ so that 
\begin{align}
\bs\alpha(\ssv_0)=2\log p_1,  \qquad \bs\alpha(\ssv_0')= -2\log p_3\,.
\label{choose-wt}
\end{align}
Then we have the following well known lemma.
\begin{lem}\label{screening lemma}
When the screening operator $S$ is well defined, we have
$$
[S,V(z)+V'(z)]=0.
$$
\end{lem}
\begin{proof}
We have
\begin{align*}
&S(z)V(w)=\frac{1-p_1^2p_2w/z}{1-p_2w/z} :S(z)V(w):, \qquad 
V(w)S(z)=\frac{1-p_1^{-2}p_2^{-1}z/w}{1-p_2^{-1}z/w}p_1^{2} :S(z)V(w):\,,
\\
&S(z)V'(w)=\frac{1-p_3^{-2}p_2^{-1}w/z}{1-p_2^{-1}w/z} :S(z)V'(w):, \qquad 
V'(w)S(z)=\frac{1-p_3^{2}p_2z/w}{1-p_2z/w}p_3^{-2} :S(z)V'(w):\,.
\end{align*}
It follows that 
$[S(z),V(w)]=(1-p_1^2)\delta(p_2 w/z):S(z)V(w)$, where $\delta(z)=\sum_{i\in\Z}z^i$ is the formal delta function. 
Integrating over $z$, we obtain $[S,V(w)]=(1-p_1^2)p_2w:S(p_2 w)V(w):$.
Similarly, we obtain $[S,V'(w)]=(1-p_3^{-2})p_2^{-1}w:S(p_2^{-1}w)V'(w):$. 

On the other hand, from the definitions we have 
\begin{align*}
:S(p_2w)V(w):\,\,=\frac{1-p_3^2}{1-p_1^2}p_2^{-2}p_3^{-2}:S(p_2^{-1}w)V'(w):.
\end{align*}
Hence the lemma follows.
\end{proof}
We note that in \eqref{condition} one can replace $A(z)$ with $A(pz)$ for any formal monomial $p$. Then the screening current $S(z)$ will be also shifted to $S(pz)$ and $S$ will be replaced by a constant multiple $p^{-1}S$. Therefore Lemma \ref{screening lemma} will still hold.

When $p_1=p_3\neq p_2$, the right hand sides of \eqref{condition} become symmetric in $p_1$ and $p_2$.
Interchanging the roles of $p_1$ and $p_2$, one can construct another screening operator commuting with $V(z)+V'(z)$. 

It is easy to check  that if  one has  $p_1p_2p_3=1$ in \eqref{condition}, 
then the corresponding screening current is an ordinary fermion  \cite{BFM}. 
\medskip

\subsection{The Fock modules}\label{sec Fock}
Algebra $\E$ has three families of Fock representations 
$\F_c(u)$, where $c\in\{1,2,3\}$ and $u\in \C^\times$. 
We call $c$ the color of the Fock module and $u$ the evaluation parameter. 
Quite generally, if the central element $C$ acts on an $\E$ module $M$ by a scalar $k$, then we say that $M$ has level $k$
and often write $C=k$. 
The Fock module $\F_c(u)$ has level $s_c$. 

The Fock modules are irreducible with respect to the Heisenberg subalgebra of $\E$
generated by $\psi^\pm(z)$. 
Thus we have the identification of vector spaces $\F_c(u)=\C[h_{-r}]_{r>0}\,{v}_c$, 
where $v_c$ is the vacuum vector such that $h_rv_c=0$ $(r>0)$ and $Cv_c=s_cv_c$. 
Then the generators $e(z)$ and $\tilde f(z)$ are given by vertex operators
\begin{align*}
&e(z)\mapsto b_c :V_c(z;u):\,,
\quad 
\tilde{f}(z)\mapsto b_c:V_c(z;u)^{-1}:\,,
\end{align*}
where $b_c=-t_c/\kappa_1$ and 
\begin{align}\label{fock va}
&V_c(z;u)=u\exp\Bigl(\sum_{r>0}\frac{\kappa_r h_{-r}}{1-q_c^r}z^{r}\Bigr)
\exp\Bigl(\sum_{r> 0}\frac{\kappa_r h_{r}}{1-q_c^r}q_c^{r/2}z^{-r}\Bigr)\,.
\end{align}
Note that 
\begin{align}
V_c(s_c^{-1}z;u)\psi^+(z)=\psi^-(s_c^{-1}z)V_c(s_cz;u)\,.\label{id-VO}
\end{align}

\medskip 

For vertex operators  \eqref{fock va}, the contractions are rational. 
The non-trivial ones are 
\begin{align}
&\mathcal{C}\bigl(V_c(z;u),V_c(w;v)\bigr)=-\frac{\kappa_1}{1-q_c},\notag
\\
&\mathcal{C}\bigl(\psi^+(q_c^{-1/2}z),V_c(w;u)\bigr)=\mathcal{C}\bigl(V_c(z;u),\psi^-(w)\bigr)=-\kappa_1\,,
\label{fock contr}\\
&\mathcal{C}\bigl(\psi^+(q_c^{-1/2}z),\psi^-(w)\bigr)=-(1-q_c)\kappa_1\,.\notag
\end{align}
\medskip

\section{$\cW$ algebras and Integrals of Motion of type $\ssA$}\label{section A}
The deformed $\cW$ algebras are introduced in \cite{FR1} 
starting from a deformed non-affine Cartan matrix, or in a more general context, a Dynkin quiver in \cite{KP}.
In these papers, supersymmetric cases were not considered.

It turns out that the deformed $\cW$ currents of  type $\ssA$ are essentially the images of the current $e(z)$ of the quantum toroidal $\gl_1$ algebra $\E$ 
acting on tensor products of Fock modules. 
In this section we discuss this connection, and explain how the deformed affine Cartan matrix 
and integrals of motion are recovered from the data of Fock modules.

\subsection{The current $e(z)$ and root currents $A_i(z)$.}\label{sec roots A}

Fix a tensor product of $\ell+1$ Fock modules 
\begin{align}
\F_{c_1}(u_1)\otimes\cdots\otimes \F_{c_{\ell+1}}(u_{\ell+1})
\,,\quad c_1,\ldots,c_{\ell+1}\in\{1,2,3\}\,,
\label{tenFock}
\end{align}
which has level $C=\prod_{j=1}^{\ell+1} s_{c_j}$.
By coproduct formulas \eqref{e coproduct}, \eqref{psi+ coproduct}, the current $e(z)$ acts 
as a sum of vertex operators in $\ell+1$ bosons
\begin{align*}
&e(z)=b_{c_1}\Lambda_1(z)+\cdots+b_{c_{\ell+1}} \Lambda_{\ell+1}(z)\,,\\
&\Lambda_i(z)=1\otimes\cdots\otimes 1
\otimes \overset{\underset{\smile}{i}}{V_{c_i}(a_{i}^\ssA z;u_i)}
\otimes \psi^+(s^{-1}_{c_{i+1}}a_{i+1}^\ssA z)\otimes \cdots \otimes \psi^+(s^{-1}_{c_{\ell+1}}a_{\ell+1}^\ssA z)\,,
\end{align*}
where
\begin{align}
a_i^\ssA=\prod_{j=i+1}^{\ell+1} s_{c_j}^{-1}\,.
\label{sta}
\end{align}

Note  that the current $e(z)$ with evaluation parameters $au_i$, 
where $a\in\C^\times$, is obtained from $e(z)$ with evaluation parameters $u_i$ by scalar multiplication by $a$.

\bigskip

To each neighboring pair of Fock spaces
\begin{align*}
\cdots\otimes
\overset{\underset{\smile}{i}}{\F_{c_i}(u_i)}\otimes\overset{\underset{\smile}{i+1}}{\F_{c_{i+1}}(u_{i+1})}
\otimes\cdots,
\quad  i=1,\ldots,\ell\,,
\end{align*}
we associate a current $A_i(z)$. Namely, for $i=1,\dots,\ell$, let $A_i(z)$ be 
given by a normally ordered ratio:
\begin{align} \label{A}
&\quad A_i(z)=\,\,:\frac{\Lambda_i((a_i^\ssA)^{-1} z)}{\Lambda_{i+1}((a_i^\ssA)^{-1} z)}:\,,
\end{align}
where $a_i^\ssA$ is given in \eqref{sta}. We call $A_i(z)$ a root current.

From \eqref{fock contr} we have the following contractions:
\begin{align*}
    \mathcal{C}\bigl(\La_i(z),\La_j(w)\bigr) = 
    \begin{cases}
\displaystyle{-\kappa_1}  & (i<j), \\ -\kappa_1/(1-q_{c_i}) & (i=j), \\0 & (i>j), 
\end{cases}
\end{align*}
where $i,j=1,\dots,\ell+1$. 

Denote the contractions between root currents by 
\begin{align}\label{Bij}
B_{i,j}=\mathcal{C}\bigl(A_i(z),A_j(w)\bigr)\, \qquad (i,j=1,\dots,\ell).
\end{align} 
The only non-trivial ones are 
\begin{align}\label{A-contr}
&
B_{i-1,i}=B_{i,i-1}=\frac{t_1t_2t_3}{t_{c_i}}, \quad
B_{j,j}=-\frac{t_1t_2t_3}{t_{c_j}t_{c_{j+1}}}(s_{c_j}s_{c_{j+1}}-s^{-1}_{c_j}s^{-1}_{c_{j+1}})\,,
\end{align}    
where $i=2,\dots, \ell$, $j=1,\dots,\ell$.

In other words, only neighboring $A_i(z)$ 
have non-trivial contractions which can be described by 
$2\times 2$ matrices corresponding to the choice of three Fock spaces, see 
type $\ssA$ matrices in Appendix \ref{sec finite}.

Note that all these contractions are rational and do not depend on evaluation parameters $u_i$. 

The matrix of contractions $B=\bigl(B_{i,j}\bigr)_{i,j=1,\ldots,\ell}$ should be compared with
the Cartan matrix of type $\ssA$. 
We say that $A_i(z)$ is a bosonic current of type $c$ if it corresponds to a pair $\mc F_c\otimes \mc F_c$, 
and a fermionic current of type $c$ if it 
corresponds to a pair $\mc F_{c_1}\otimes \mc F_{c_2}$, where $\{c_1,c_2,c\}=\{1,2,3\}$.

\medskip

\begin{example}\label{A example}  Consider the tensor product
\begin{align}\label{example fock}
\F_1(u_1)\otimes\F_1(u_2)\otimes\F_1(u_3)\otimes\F_2(u_4)\otimes\F_2(u_5)\otimes \F_3(u_6)\,, \qquad C=s_1^3s_2^2s_3.
\end{align}
The 
matrix $B$ is given by
\begin{align}\label{Cartan example}
&B=
\begin{pmatrix}
 -t_2t_3(s_1+s_1^{-1}) & t_2t_3  &     0     &        0      &    0  \\
             t_2t_3 &-t_2t_3(s_1+s_1^{-1})& t_2t_3   & 0  &    0           \\
                 0     &t_2t_3                &t_3^2  & t_1t_3       &      0         \\
                   0   &            0          &t_3t_t  &-t_1t_3 (s_2+s_2^{-1})& t_1t_3    \\
                   0   &      0&0& t_1t_3 &t_1^2
\end{pmatrix}.
\end{align}
We picture this matrix as a Dynkin diagram where 
a circle represents a bosonic node and a crossed circle
a fermionic node. We also attach a marking by $1,2,3$ to each vertex, indicating its type.\medskip

\begin{align*}
\begin{tikzpicture}
\dynkin[root radius=.1cm, edge length=1.3cm, labels={1,1,3,2,1}]{A}{ootot}
\end{tikzpicture}
\end{align*}
\qed
\end{example}

The diagram in the example looks like a diagram of ${\sln}_{4|2}$ 
but it is different because of the markings. 
In the diagram of $\sln_{4|2}$ all bosonic nodes have the same marking 
and all fermionic nodes have the same marking too:
\medskip

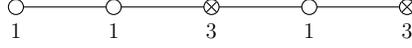
\begin{align*}
\begin{tikzpicture}
\dynkin[root radius=.1cm, edge length=1.3cm, labels={1,1,3,1,3}]{A}{ootot}
\end{tikzpicture}
\end{align*}
And if all colors are the same, e.g. $c_i=2$, $i=1,\dots,\ell+1$, then
\begin{align}\label{A Cartan}
B_{i,j}
= -t_1t_3\bigl((s_2+s_2^{-1})\delta_{ij}-
    \delta_{i,j+1}-\delta_{i,j-1}\bigr)\,.
\end{align}
Up to an overall multiple $-t_1t_3$, 
this coincides with the matrix $B(q,t)$ of type $A_{\ell}$ 
in \cite{FR1}, eq.(2.4) with the identification $s_2=qt^{-1}$.
\medskip

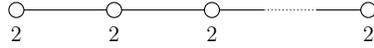
\begin{align*}
\begin{tikzpicture}
\dynkin[root radius=.1cm, edge length=1.3cm, labels={2,2,2,2}]{A}{ooo..o}
\end{tikzpicture}
\end{align*}
However, in contrast to \cite{FR1}, we have one extra boson. Indeed, by construction our root currents
commute with the diagonal Heisenberg algebra 
\begin{align*}
[A_i(z),\Delta^{(\ell)}h_r]=0\qquad (i=1,\ldots,\ell,\ r\neq 0),
\end{align*}
where $\Delta^{(\ell)}$ signifies the iterated coproduct with $\Delta^{(1)}=\Delta$. 
This extra boson allows us to have rational commutation relations between $\La_i(z)$'s,
and will play a role in the affinization to be discussed below. 
\medskip
 
\subsection{Root current $A_0(z)$}\label{affine root A sec}
Fix an arbitrary $\mu\in\C^\times$ with $|\mu|<1$. 
We define the dressed version of the current $e(z)$ of the algebra $\E$ by
\begin{align}\label{dress}
\mathbf{e}(z)=e(z){\psi_\mu^+(C^{-1} z)}^{-1}\,,\qquad
\psi_\mu^+(z)=\prod_{s=0}^\infty \psi^+(\mu^{-s}z)
=\exp\bigl(\sum_{r>0}\frac{\kappa_r }{1-\mu^r}h_{r}z^{-r}\bigr)\,.
\end{align}
The coefficients of the dressed current $\mathbf{e}(z)$
are elements of $\widetilde \E$. 
We warn the reader that  $\mathbf{e}(z)$ at $\mu=0$ does not reduce to $e(z)$ but rather
to $e(z){\psi^+(C^{-1}z)}^{-1}$. This somewhat unusual convention is due to historical reasons.

While the current $e(z)$ has rational commutation relations, the dressed current $\bs e(z)$ satisfies the elliptic commutation relations:
\begin{align*}
\mathbf{e}(z)\mathbf{e}(w)=\mathbf{e}(w)\mathbf{e}(z)\prod_{i=1}^3\frac{\Theta_\mu(q_iw/z)}{\Theta_\mu(q_i^{-1}w/z)}\,.
\end{align*}
Here we use symbols for infinite products and theta functions
\begin{align*}
&(z_1,\ldots,z_r;p)_\infty=\prod_{i=1}^r\prod_{k=0}^\infty(1-z_i p^k)\,,
\quad 
\Theta_p(z)=(z,pz^{-1},p;p)_\infty\,.
\end{align*}

These relations have to be understood as equality of matrix elements. 
Namely, if  $M$ is an admissible $\E$ module then the matrix coefficients of both sides
converge to meromorphic functions which are analytic continuation of each other.

There are two separate motivations for the definition of the dressed current. One is the form of the integrals of motion, see Section \ref{int A sec}. In this section we discuss a different reason: the appearance of the current $A_0(z)$ corresponding to the affine node and ultimately of the screening operator corresponding to the affine node, see Section \ref{sec A screenings}. In type $\ssA$ these two points lead to the same definition of the dressed current.

The dressed  current $\bs e(z)$ has the form
$$
\bs e(z)=b_{c_1}\bs\La_1(z)+\dots + b_{c_{\ell+1}}\bs\La_{\ell+1}(z),
$$
where $\bs\La_i(z)=\La_i(z)\Delta^{(\ell)}{\psi_\mu^+(C^{-1} z)}^{-1}$.

For $i=1,\dots,\ell$, the ratios of the terms in dressed and undressed currents are the same:
\begin{align*}
&
:\frac{\bs\Lambda_i(z)}{\bs\Lambda_{i+1}(z)}:
\,\,=\,\,:\frac{\Lambda_i(z)}{\Lambda_{i+1}(z)}:\,\,
=A_i(a_i^Az)\,,\quad i=1,\ldots,\ell\,.
\end{align*}
However the contractions of individual terms change to  elliptic ones:
\begin{align}
\mathcal{C}\bigl(\bs \La_i(z), \bs \La_j(w)\bigr)= \mathcal{C}\bigl(\La_i(z),\La_j(w)\bigr)+\frac{\kappa_1}{1-\mu}.
\label{cont bs La}
\end{align}

This allows us to define a new current $A_0(z)$ which has rational contractions 
with all root currents $A_i(z)$, $i=1,\dots,\ell$, and is independent of them. Namely, define the zeroth root current by
\begin{align}
A_0(z)=\,\,
:\frac{\bs \Lambda_{\ell+1}(z)}{\bs\Lambda_{1}(\mu z)}:\,\,
=\,\,
:\frac{\Lambda_{\ell+1}(z)}{\Lambda_{1}(\mu z)}\Delta^{(\ell)}{\psi^+(C^{-1}\mu z)}:\,.
\label{A0}
\end{align}

The contractions involving $A_0(z)$ are as follows. Retaining  notation \eqref{Bij}
we have $B_{i,0}=B_{0,i}=0$ if $i\neq 0,1,\ell$. 
The non-trivial ones are:
\begin{align*}
&B_{0,0}=-\frac{t_1t_2t_3}{t_{c_{\ell+1}}t_{c_1}}
(s_{c_{\ell+1}}s_{c_1}-s^{-1}_{c_{\ell+1}}s^{-1}_{c_1})\,,
\\
&B_{0,\ell}=B_{\ell,0}=\frac{t_1t_2t_3}{t_{c_{\ell+1}}}\,,
\qquad
B_{0,1} =\frac{t_1t_2t_3}{t_{c_1}}C\mu^{-1}, \quad 
B_{1,0} =\frac{t_1t_2t_3}{t_{c_1}}C^{-1}\mu\,,
\end{align*}
where $C=\prod_{j=1}^{\ell+1}s_{c_j}$ is the total level and $\ell>1$. 

For $\ell=1$, $B_{0,1}$ and $B_{1,0}$ take the form:
\begin{align*}
    B_{0,1}=-\kappa_1\Big(\frac{1}{t_{c_2}}+\frac{C\mu^{-1}}{t_{c_1}}\Big), 
    \qquad B_{1,0}= -\kappa_1\Big(\frac{1}{t_{c_2}}+\frac{C^{-1}\mu}{t_{c_1}}\Big).
\end{align*}
Thus we have a  $U_q\widehat{\sln}_2$ type deformed Cartan matrix corresponding to $c_1=c_2=2$ and a $U_q\widehat{\gl}_{1|1}$ type deformed Cartan matrix corresponding to $c_1=1$, $c_2=2$:
\begin{align*}
-t_1t_3\begin{pmatrix}
s_2+s_2^{-1} & -1-C\mu^{-1}\\
-1-C^{-1}\mu & s_2+s_2^{-1}\\
\end{pmatrix},    
\qquad
t_3\begin{pmatrix}
t_3 &  t_1+t_2C\mu^{-1} \\
t_1+t_2C^{-1}\mu & t_3 \\
\end{pmatrix}.
\end{align*}

We call  the extended matrix of contractions $\hat B=\bigl(B_{i,j}\bigr)_{i,j=0,1,\ldots,\ell}$ 
the symmetrized deformed affine Cartan matrix. 
\bigskip

\begin{example} We continue to consider the tensor product \eqref{example fock}.
The matrix $\hat B$ is given by
\begin{align}\label{Cartan example2}
&\hat B=
\begin{pmatrix}
t_2^2       &t_2t_3 C\mu^{-1}         &     0     &     0     &       0       & t_1t_2      \\
t_2t_3C^{-1}\mu & -t_2t_3(s_1+s_1^{-1}) & t_2t_3  &    0      &         0     & 0     \\
           0  &t_2t_3               &-t_2t_3(s_1+s_1^{-1})& t_2t_3   &  0 &    0           \\
           0  &      0   &t_2t_3                &t_3^2  & t_1t_3       &      0         \\
            0 &     0    &     0                 &t_1t_3  &-t_1t_3(s_2+s_2^{-1})&t_1t_3     \\
t_1t_2      &    0    &    0        &     0     &t_1t_3                    &t_1^2    \\
\end{pmatrix}.
\end{align}

With the root $A_0(z)$ the Dynkin diagram becomes the following.

\begin{align}\label{dynkin A ex} 
\raisebox{-20pt}{\begin{tikzpicture}[baseline=(origin.base)] 
\dynkin[root radius=.1cm, edge length=1.3cm, labels={{2},1,1,3,2,1},extended, affine mark=t]{A}{ootot}
\end{tikzpicture}}
\end{align}
The affine node is fermionic and we label it by 2, since it corresponds to the ratio of the last and first terms obtained from $\F_3(u_6)$ and $\F_1(u_1)$.
However, note that there is an arbitrary parameter $\mu$ which does not appear in the diagram. \qed
\end{example}

\medskip
One can easily describe the restrictions on the markings of the Dynkin diagram which can appear. 
There is a single  global condition: if there are $a_c$ fermionic nodes of type $c$, 
$c=1,2,3$, then $a_1\equiv a_2\equiv a_3$ modulo 2. 
In addition, there are several local conditions. For example, neighboring bosonic nodes have to have the same marking, neighboring bosonic and fermionic nodes cannot have the same marking.

In particular, \eqref{dynkin A ex} is not a diagram of affine $\sln_{4|2}$, as in any affine Dynkin diagram of type $\sln_{m|n}$ the number of simple odd roots is even (which is true in our setting if all bosonic nodes have the same type).  

\medskip

If all colors are the same, e.g. $c_i=2$, $i=1,\dots,\ell+1$, then 
the matrix $\hat B$ is, up to a scalar multiple,  
the deformed affine Cartan matrix of type $\ssA$ in \cite{KP}.

\bigskip

The following lemma shows that in all cases
the currents $A_i(z)$, $i=0,1,\dots,\ell$, are independent generating currents of the Heisenberg algebra
provided  $\mu\neq1,C^2$.

\begin{lem}\label{det lemma}
We have
\begin{align*}
\det \hat B=\kappa_1^{\ell+1}\prod_{j=1}^{\ell+1}t_{c_j}^{-1}\times (1-\mu)(C^{-1}-C\mu^{-1})\,.
\end{align*}
\end{lem}
\begin{proof}
It is easy to see that $\det \hat B$ as a function of $x=C^{-1}\mu$ 
is a linear function of $x+x^{-1}$. 
One can check that if $x=C^{-1}$ then the vector $(y_0,y_1,\ldots,y_\ell)^T$ with $y_0=C$ and $y_i=s_{c_1}\cdots s_{c_i}$ 
($1\le i\le \ell$) is in the kernel of $\hat B$. 
It follows that $\det\hat  B=a(C+C^{-1}-x-x^{-1})$ with some $a$. 
The coefficient $a$ can be determined from the behavior as $x\to\infty$
\begin{align*}
\det\hat B=(-1)^{\ell} \hat B_{1,0}\hat B_{0,\ell}\prod_{j=2}^{\ell}\hat B_{j,j-1}+O(1)=
-x\prod_{j=1}^{\ell+1}\frac{\kappa_1}{t_{c_j}}+O(1)\,.
\end{align*}
\end{proof}
\medskip

\subsection{Currents $Y_i(z)$ and $qq$-characters.}\label{q char sec}
In this section we describe current $\bs e(z)$ as a $qq$-character in the spirit of \cite{N}, \cite{KP}.

In what follows we write $\hat B=\hat D\hat C$, $\hat C=(C_{i,j})_{i=0}^\ell$,
choosing a diagonal matrix $\hat D={\rm {diag}}(d_0,\dots,d_{\ell})$ in such a way that 
\begin{align*}
&d_i=-\frac{t_1t_2t_3}{t_c},\quad&&  C_{i,i}=s_c + s_c^{-1},&& \text{if  $A_i(z)$ is bosonic of type $c$},\\
&d_i=t_c,\quad && C_{i,i}=s_c - s_c^{-1},&&\text{if  $A_i(z)$ is fermionic of type $c$}.
\end{align*}
The currents $A_i(z)$ correspond to roots.
In order to understand the combinatorics of deformed $\cW$-currents, it is convenient to introduce 
currents $Y_i(z)$, $i=0,\dots,\ell$, which correspond to the fundamental weights.

Define the modes $\ssa_{i,r}$ by setting 
\begin{align*}
A_i(z)=e^{\ssa_{i,0}}:e^{\sum_{r\neq0}\ssa_{i,r}z^{-r}}:\,.
\end{align*}
The zero mode $e^{\ssa_{i,0}}$ is a variable, which takes the value  $u_i/u_{i+1}$
in \eqref{A} and \eqref{A0}, where $u_0=u_{\ell+1}$.

Write elements of non-symmmetrized deformed affine Cartan matrix as 
$C_{i,j}=\sum_{k} m_{i,j}^{(k)}-\sum_s n_{i,j}^{(s)}$, 
where $m_{i,j}^{(k)}$, $n_{i,j}^{(s)}$ are  monomials of the form $s_1^as_2^b\mu^c$, $a,b,c\in\Z$,
and define $Y_i(z)$ by the set of equations
\begin{align}
A_j(z)=\,\,:\prod_i \prod_k Y_i(m_{i,j}^{(k)}z)\prod_s \Big(Y_i(n_{i,j}^{(s)}z)\Big)^{-1}:,  \qquad j=0,\dots,\ell,
\label{AinY}
\end{align}
where $Y_i(z)$ are of the form
\begin{align*}
Y_i(z)=\,\,:e^{\sum_{r\neq0}\ssy_{i,r}z^{-r}}:\times 
\begin{cases}
e^{\ssy_{i,0}}&\text{if $A_i(z)$ is bosonic},\\
e^{Q_{\ssy_i}}z^{\ssy_{i,0}}
&\text{if $A_i(z)$ is fermionic}.\\
\end{cases}
\end{align*}
In the right hand side of \eqref{AinY}, $Y_i(z)$ for fermionic $A_i(z)$ appears only as a ratio $:Y_i(az)Y_i(bz)^{-1}:$, 
so that  $e^{Q_{\ssy_i}}$ cancels out. We shall specify the latter when we discuss the screening currents, 
see Section \ref{sec A screenings} below. 

Due to Lemma \ref{det lemma}, such $Y_i(z)$'s exist and are unique up to an overall shift of zero modes $y_{i,0}$. 
Moreover, we have 
\begin{align}
\mc C(A_j(w),Y_i(z))=\pm \mc C(Y_i(z),A_j(w))
=d_i\delta_{i,j}\,,
\label{defY}
\end{align}
where the sign is $+$ for bosonic nodes and $-$ for fermionic nodes.

We explain this definition in an Example \ref{A example 2} below.

\medskip

Let us set up some language convenient for discussing combinatorics.
In what follows we set $\mu=C s_1^{-\gamma}$, $\gamma\in\C$, and rewrite monomials 
$s_1^i s_2^j \mu^k$, $i,j,k\in\Z$, in the form $s_1^{a}s_2^b$, $a,b\in\C$.
We adopt the shorthand notation ${\bs l}_{a,b}=Y_l(s_1^as_2^bz)$, $l=0,1,\dots,\ell$.

Let $\mc A$ be the commutative ring in formal free variables $\bs l_{a,b}^{\pm1}$, $l=0,1,\dots,\ell$, $a,b\in \C$. 

Let $\pi_0: \mc A\, \to \mc A$ be the ring homomorphism sending 
$\bs 0_{a,b} \mapsto 1$,  $\bs l_{a,b} \mapsto \bs l_{a,b}$, $l\neq 0$.

For a formal monomial $\alpha=s_1^as_2^b$ define a ring homomorphism 
$\tau_\alpha:\,  \mc A \to \mc A$ shifting indices by $(a,b)$. 
For example, we have $\tau_{s_1^as_2^b}\bs 1_{c,d}=\bs 1_{a+c,b+d}$.

The combinatorial study of the $qq$-characters and $q$-characters is 
based on the concept of dominant monomial, see \cite{FR2}. The characters are obtained by 
a combinatorial algorithm ``expanding the dominant monomials", see \cite{FM}, \cite{KP}.  

If current $A_l(z)$ is bosonic, then a monomial $m\in \mc A$ is called $\bs l$-dominant 
if $m$ does not contain negative powers of $\bs l_{a,b}$. If current $A_l(z)$ is fermionic, 
then we suggest to call a monomial $m\in \mc A$ $\bs l$-dominant if $m$ is a product of 
$\tilde {\bs l}_{a,b}$ with $\tilde l\neq l$ and monomials of the form $\bs l_{a,b}\tau_{q_c}(\bs l_{a,b}^{-1})$, 
where in the Cartan matrix we have either $\hat C_{l,l-1}=t_c$ or $\hat C_{l,l+1}=t_c$.

\medskip

\begin{example}\label{A example 2}. We continue to work with \eqref{Cartan example}. In this case we have
$$
d_0=t_2\,, \qquad d_1=d_2=-t_2t_3\,, \qquad d_3=t_3\,, \qquad d_4=-t_1t_3\,,\qquad d_5=t_1\,,
$$
and
\begin{align*}
&\hat C=
\begin{pmatrix}
t_2       &t_3 C\mu^{-1}         &     0     &     0     &       0       & t_1      \\
-C^{-1}\mu & s_1+s_1^{-1} & -1  &    0      &         0     & 0     \\
           0  &-1               &s_1+s_1^{-1}& -1 &  0 &    0           \\
           0  &      0   &t_2                &t_3  & t_1       &      0         \\
            0 &     0    &     0                 &-1 &s_2+s_2^{-1}& -1     \\
t_2      &    0    &    0        &     0     & t_3                    &t_1    \\
\end{pmatrix}\,.
\end{align*}

Reading off the columns of $\hat C$, we obtain $A_i(z)$ currents in terms of the $Y_j(z)$'s: 
\begin{align*}
&A_0(z)=\bs 0_{0,1}\bs 0_{0,-1}^{-1}\bs 1_{-\gamma,0}^{-1}\bs 5_{0,1}\bs 5_{0,-1}^{-1},
&A_1(z)=\bs 0_{\gamma-1,-1}\bs 0_{\gamma+1,+1}^{-1}\bs 1_{1,0}\bs 1_{-1,0}\bs 2_{0,0}^{-1},\ \  
&A_2(z)=\bs 1_{0,0}^{-1}\bs 2_{1,0}\bs 2_{-1,0}\bs 3_{0,1}\bs 3_{0,-1}^{-1}, \\
&A_3(z)=\bs 2_{0,0}^{-1}\bs 3_{-1,-1}\bs 3_{1,1}^{-1}\bs 4_{0,0}^{-1},
&A_4(z)=\bs 3_{1,0}\bs 3_{-1,0}^{-1}\bs4_{0,1}\bs 4_{0,-1}\bs 5_{-1,-1}\bs 5_{1,1}^{-1},\ \ 
&A_5(z)=\bs 0_{1,0}\bs 0_{-1,0}^{-1}\bs 4_{0,0}^{-1}\bs 5_{1,0}\bs 5_{-1,0}^{-1},
\end{align*}
where $s_1^\gamma=C \mu^{-1}$. 

Set $\delta_i=\ln q_i$, $\delta_1+\delta_2+\delta_3=0$. Then zero modes are given by
\begin{align*}
&\ssa_{0,0}=\delta_2 \ssy_{0,0}-\ssy_{1,0}+\delta_2\ssy_{5,0},\ 
\ssa_{1,0}=\delta_3\ssy_{0,0}+2\ssy_{1,0}-\ssy_{2,0},\ 
\ssa_{2,0}=-\ssy_{1,0}+2\ssy_{2,0}+\delta_2\ssy_{3,0},\ \\
&\ssa_{3,0}=-\ssy_{2,0}+\delta_3\ssy_{3,0}-\ssy_{4,0},\ 
\ssa_{4,0}=\delta_1\ssy_{3,0}+2\ssy_{4,0}+\delta_3\ssy_{5,0},\ 
\ssa_{5,0}=\delta_1\ssy_{0,0}-\ssy_{4,0}+\delta_1\ssy_{5,0}\,. 
\end{align*}
These conditions uniquely determine 
$\ssy_{i,0}$'s up to a common additive shift.

The current $\bs e(z)$ can be written:
\begin{align*}
&\bs e(a_0^{-1}z)= \bs \La_1(a_0^{-1}z)
\Bigl(b_{c_1}1+b_{c_2}A^{-1}_1(s_1z)+b_{c_3}A^{-1}_1(s_1z)A^{-1}_2(s_1^2z)+
b_{c_4}A^{-1}_1(s_1z)A^{-1}_2(s_1^2z)A^{-1}_3(s_1^3z)\notag\\ 
&+b_{c_5}A^{-1}_1(s_1z)A^{-1}_2(s_1^2z)A^{-1}_3(s_1^3z)A^{-1}_4(s_1^3s_2z)
+b_{c_6}A^{-1}_1(s_1z)A^{-1}_2(s_1^2z)A^{-1}_3(s_1^3z)A^{-1}_4(s_1^3s_2z)A^{-1}_5(s_1^3s_2^2z)\Bigr),
\notag
\end{align*}
which we symbolically depict as follows.
$$
\begin{tikzcd}
    \bs\La_1\arrow{r}{A_1^{-1}} & \bs\La_2 \arrow{r}{A_2^{-1}} & \bs\La_3 \arrow{r}{A_3^{-1}} & \dots\arrow{r}{A_{\ell-1}^{-1}} &\bs\La_{\ell}\arrow{r}{A_{\ell}^{-1}} &  \bs\La_{\ell+1}
\end{tikzcd}
$$

Ignoring constants $b_{c_i}$ we write it using $Y_i(z)$ 
in the form:
\begin{align}\label{qq-A}
\chi =\bs 0_{\gamma,-1} \bs 0_{\gamma+2,1}^{-1}\bs 1_{0,0} 
+\bs 1_{2,0}^{-1}\bs 2_{1,0}+\bs 2_{3,0}^{-1} \bs 3_{2,-1}\bs 3_{2,1}^{-1}
+ \bs 3_{4,1} \bs 3_{2,1}^{-1}\bs 4_{3,0}+\bs 4_{3,2}^{-1} \bs 5_{4,2} \bs 5_{2,0}^{-1}
+ \bs 0_{2,2}\bs 0_{4,2}^{-1}\bs 5_{2,2}\bs 5_{2,0}^{-1}.
\end{align}
Following \cite{N}, we call the expression $\chi$ the 
$qq$-character of the vector representation corresponding to the Dynkin diagram \eqref{dynkin A ex}.  

According to the general rule, 
for $l=1,2,4$ (that is, for bosonic nodes) a monomial is $\bs l$-dominant if it is a product of 
monomials $\bs l_{a,b}$ and $\tilde {\bs l}_{a,b}^{\pm1}$ with $a,b\in\C$ and $\tilde l\neq l$.
A monomial is $\bs 0$-dominant if it is a product of monomials of the form 
$\bs 0_{a,b} \bs 0^{-1}_{a+2,b}$, $\bs 0_{a,b} \bs 0^{-1}_{a-2,b-2}$, and  $\bs l_{a,b}^{\pm 1}$ 
with $l\neq 0$;  $\bs 3$-dominant if it is a product of monomials of the form 
$\bs 3_{a,b} \bs 3^{-1}_{a+2,b}$, $\bs 3_{a,b} \bs 3^{-1}_{a,b+2}$, and  $\bs l_{a,b}^{\pm 1}$ 
with $l\neq 3$;  $\bs 5$-dominant if it is a product of monomials of the form 
$\bs 5_{a,b} \bs 5^{-1}_{a,b+2}$, $\bs 5_{a,b} \bs 5^{-1}_{a-2,b-2}$, and  $\bs l_{a,b}^{\pm 1}$ with $l\neq 5$.

In \eqref{qq-A} the $l$-th term is $\bs l$-dominant for $1\le l\le 5$, and the last term is 
$\bs 0$-dominant. The $(l+1)$st term is obtained by multiplying the $l$-th term by the 
inverse  of current $A_l(z)$ with an appropriate shift, see \eqref{A}. 
Moreover, the last term and the first term are connected as follows
\begin{align}\label{first and last}
A^{-1}_0(s_1^3s_2^2s_3z)\bs 0_{2,2}\bs 0_{4,2}^{-1}\bs 5_{2,2}\bs 5_{2,0}^{-1} 
=\bs 0_{2,0}\bs 0_{4,2}^{-1}\bs 1_{\gamma+2,1}=\tau_\mu(\bs 0_{\gamma,-1} \bs 0_{\gamma+2,1}^{-1}\bs 1_{0,0}).
\end{align}
\qed
\end{example}
\medskip

While our $qq$-character is closely connected to that of \cite{KP}, they are different in several ways. 
First, our first monomial is $\bs l$-dominant for $l=1,\dots,\ell$, but not $\bs 0$-dominant. 
Second, our $qq$-character is of finite type, namely, applying $\pi_{0}$ we obtain the $qq$-character 
corresponding to the non-affine Cartan matrix. In particular, our $qq$-character contains finitely many terms.
Third, the fermionic nodes are not considered in \cite{KP}\footnote{Fermionic root currents $A_i(z)$ appeared in \cite{BFM}.}.
Fourth, the variables $\bs 0_{a,b}$ do play an important role as they correspond to the dressing of the current 
$e(z)$ and the finite type $qq$-character ``closes up" in the sense of \eqref{first and last}. 

The $qq$-characters of \cite{KP} commute with all screenings operators including the one corresponding to the zeroth node and correspond to modules of quantum toroidal algebra $\E_{\ell+1}$ associated to $\gl_{\ell+1}$. If we start with our $qq$-character and formally require such commutativity, we would have to add infinitely many terms which correspond to the vector representation of $\E_{\ell+1}$, while \cite{KP} starts with a dominant monomial which produces a $qq$-character corresponding to a Fock module of $\E_{\ell+1}$.

\medskip

The usual $q$-character of the evaluation vector representation of $U_q \widehat{\mathfrak{sl}}_\ell$
in the sense of \cite{FR1} is recovered after applying $\pi_0$ in the case of $\otimes_{i=1}^{\ell+1} \mc F_2(u_i)$, 
when all Fock spaces are of the same sort. In this case we have
$$
b_2^{-1}\bs e(z)=\bs 0_{0,\gamma+1}^{-1}\bs 1_{0,0}+\bs 2_{0,1}\bs 1_{0,2}^{-1}+\bs 3_{0,2}\bs 2_{0,3}^{-1}+\dots 
+\bs \ell_{0,\ell+1}^{-1}\bs 0_{0,\ell}\,
$$
where $s_2^\gamma=C\mu^{-1}$. 
\medskip

\subsection{Screenings}\label{sec A screenings}
We discuss the screening operators.

First, we clearly have $[A_i(z),\bs\La_{j}(w)]= 0$ whenever $j\not \equiv i,i+1\ \bmod \ell+1$.

Moreover, for $i=0,1,\dots,\ell$, the non-trivial contractions of $A_i(z)$ have the form \eqref{condition} where 
\begin{align*}
A(z)=\frac{b_{c_i}}{b_{c_{i+1}}}A_i(z)\,,\quad
V(z)=b_{c_i}\bs\La_i((a^A_i)^{-1}z)\,,\quad  V'(z)=b_{c_{i+1}}\bs \La_{i+1}((a^A_i)^{-1}\mu^{\delta_{0,i}}z)\,. 
\end{align*}
Here we set $\bs \La_{\ell+1}(z)=\bs \La_0(z)$.  
If the node is bosonic, $c_i=c_{i+1}$, then
$p_3=p_1=s_c$, $p_2=s_b$, where $\{c_i,b,c\}=\{1,2,3\}$. If the node is fermionic, then 
$p_1=s_{c_{i+1}}$, $p_2=s_d$, $p_3=s_{c_{i}}$, where $\{c_i,c_{i+1},d\}=\{1,2,3\}$.

Accordingly we define two screening currents 
$S_i^\pm(z)=e^{Q_{\sss_i}}z^{\sss^\pm_{i,0}}S_i^{\pm,osc}(z)$ 
for each bosonic node and 
one screening current 
$S_i^f(z)=e^{Q_{\sss_i}}z^{\sss^f_{i,0}}S^{f, osc}_i(z)$ 
for each fermionic node, see \eqref{screening}, \eqref{choose-wt}:
\begin{align}
c_i=c_{i+1}:
&\quad A_i(z)=s_{c_i}s_{c_{i+1}}    
:\frac{S^+_i(s_b^{-1}z)}{S^+_i(s_b z)}:\
=s_{c_i}s_{c_{i+1}} 
:\frac{S^-_i(s_c^{-1}z)}{S^-_i(s_c z)}:\,,
\label{Spm}\\
c_i\neq c_{i+1}:
&\quad A_i(z)=s_{c_i}s_{c_{i+1}} 
:\frac{S^f_i(s_d^{-1}z)}{S^f_i(s_d z)}:\,,
\label{Sf}
\end{align}
where $(c_i,b,c)=cycl(1,2,3)$ and $(c_i,c_{i+1},d)=\{1,2,3\}$. 
We choose $Q_{\sss_i}$ by demanding $[\ssy_{j,0},Q_{\sss_i}]=g_i\delta_{i,j}$, where 
\begin{align}
g_i=\begin{cases}
\log q_c & \text{for $S_i^+$},\\
\log q_b & \text{for $S_i^-$},\\
-1 &  \text{for $S_i^f$}.\\
\end{cases}
\end{align}
In the fermionic case \eqref{Sf}, we set in addition $[\sss_{i,0},Q_{\ssy_j}]=-\delta_{i,j}$ 
so that we have
\begin{align*}
S^f_i(z)Y_i(w)=\frac{1}{z-w}:S^f_i(z)Y_i(w):\,, \quad Y_i(w)S^f_i(z)=\frac{1}{w-z}:Y_i(w)S^f_i(z):\,.
\end{align*}

Introduce the corresponding screening operators $S_i^\pm$, $S^f_i$ by \eqref{Sc-op} when well-defined. 
\bigskip

Let $S_i$ stand for either $S_i^\pm$ when $A_i(z)$ is a bosonic current, 
or $S^f_i$ when $A_i(z)$ is a fermionic current. 
It follow from Lemma \ref{screening lemma} that  
the currents $e(z)$ and $\bs e(z)$ both commute with all $S_i$ with $i\neq0$, and that
$\bs e(z)$ commutes with $S_0$ up to a $\mu$-difference:
\begin{align}
&[S_i, e(z)]=[S_i, \bs e(z)]=0\,, \qquad i=1,\ldots,\ell\,,
\label{ScA}\\
&[S_0,\bs e(z)]=b_{c_1}[S_0,\bs \La_1(z)-\bs \La_1(\mu z)]\,.\label{ScA0}
\end{align}
The relation \eqref{ScA0} implies the commutativity of $S_0$ with integrals of motion, see Theorem \ref{A comm thm} below.
In view of these relations, we call $\bs e(z)$ the deformed $\cW$-current of type $\ssA$.

\bigskip
 
 \begin{example}
We again illustrate the construction of screenings on the example of \eqref{example fock}, see also Example \ref{A example 2}. In this case 
\begin{align*}
&\frac{\bs\Lambda_6(z)}{\bs\Lambda_1(\mu z)}=
s_2^{-1}\frac{S^f_0(s_2^{-1}z)}{S^f_0(s_2z)}\,, \quad
&\frac{\bs\Lambda_1(z)}{\bs\Lambda_2(z)}=
q_1\frac{S_1^\pm(s_{(5\pm1)/2}^{-1}a^A_1z)}{S_1^\pm(s_{(5\pm1)/2}a^A_1z)}\,,
\\
&\frac{\bs\Lambda_2(z)}{\bs\Lambda_3(z)}=
q_1\frac{S_2^\pm(s_{(5\pm1)/2}^{-1}a^A_2z)}{S_2^\pm(s_{(5\pm1)/2}a^A_2z)}\,,
\quad
&\frac{\bs\Lambda_3(z)}{\bs\Lambda_4(z)}=
s_3^{-1}\frac{S^f_3(s_3^{-1}a^A_3z)}{S^f_3(s_3 a^A_3z)}\,,
\\
&\frac{\bs\Lambda_4(z)}{\bs\Lambda_5(z)}=
q_2\frac{S_4^\pm(s_{2\pm 1}^{-1}a^A_4z)}{S_4^\pm(s_{2\pm 1}a^A_4z)}\,,
\quad
&\frac{\bs\Lambda_5(z)}{\bs\Lambda_6(z)}=
s_1^{-1}\frac{S^f_5(s_1^{-1}a^A_5z)}{S^f_5(s_1a^A_5z)}\,.
\end{align*}
In particular the zero modes of screening operators satisfy 
\begin{align*}
&\frac{u_6}{u_1}=s_2^{-1}q_2^{-\sss^f_{0,0}}\,,\ 
\frac{u_1}{u_2}=q_1q_{(5\pm 1)/2}^{-\sss_{1,0}^\pm}\,,\ 
\frac{u_2}{u_3}=q_1q_{(5\pm 1)/2}^{-\sss_{2,0}^\pm},\\
&\frac{u_3}{u_4}=s_3^{-1}q_3^{-\sss^f_{3,0}},\ 
\frac{u_4}{u_5}=q_2q_{2\pm1}^{-\sss_{4,0}^\pm},\ 
\frac{u_5}{u_6}=s_1^{-1}q_1^{-\sss^f_{5,0}}\,.
\end{align*}
Recall that we need $\sss_{i,0}$ to act as an 
integer in order to have well defined screening operators. It dictates some conditions on $u_i$.

It is convenient to rename the zeroth screening current by setting
$S^f_0(z)=z^{\Delta-1}\overline{S}^{f}_0(z)$, $\sss^f_{0,0}=\overline{\sss}^{f}_{0,0}+\Delta-1$, 
where $q_2^{\Delta-1}=C^2$. 
Then the screening operator reads
\begin{align*}
S^f_0=\int z^{\Delta-1}\overline{S}^{f}_0(z)\,dz
\end{align*}
and $\bar{\sss}^f_{0,0}$ is determined from 
\begin{align*}
1=q_2^{-\overline{\sss}^{f}_{0,0}}q_{(5\pm 1)/2}^{-\sss_{1,0}^\pm-\sss_{2,0}^\pm}q_3^{-\sss^f_{3,0}}
q_{2\pm1}^{-\sss_{4,0}^\pm}q_1^{-\sss^f_{5,0}}\,.
\end{align*}

\qed
\end{example}

\subsection{Integrals of motion}\label{int A sec}
One of our primary interests is in constructing commuting families of operators. 
In the case of type $\ssA$ the source of such families is the standard transfer matrix construction, see \cite{FJM}. 
Here we recall the answer, which appeared first in \cite{FKSW}.

Define the Feigin-Odesskii kernel function \cite{FO}:
\begin{align}
&\kfun_2(z)=\frac{\Theta_{\mu}(z)\Theta_{\mu}(q_2^{-1}z)}
{\Theta_{\mu}(q_1 z)\Theta_{\mu}(q_3 z)}\,.
\label{kfun}
\end{align}
The function $\kfun_2(z)$ satisfies a series of identities
parametrized by $m,n\in\Z_{\geq 1}$, see \cite{FKSW}:
\begin{align}
\mathop{\mathrm{Sym}}_{z_1,\cdots,z_{m+n}}\prod_{1\le i\le m\atop m+1\le j\le m+n}\kfun_2(z_j/z_i)^{-1}
=\mathop{\mathrm{Sym}}_{z_1,\cdots,z_{m+n}}\prod_{1\le i\le m\atop m+1\le j\le m+n}\kfun_2(z_i/z_j)^{-1}
\,.
\label{id-h}
\end{align}
The following theorem describes the integrals of motion.

\begin{thm}\label{IMtypeA}
\cite{FKSW}, \cite{FJM}\quad 
The following elements $\{\bI_n\}_{n=1}^\infty$ are mutually commutative: 
\begin{align}\label{integrals}
&\bI_n=
\int\!\!\cdots \!\!\int
\mathbf{e}(z_1)\cdots \mathbf{e}(z_n)
\cdot \prod_{j<k}\kfun_2(z_k/z_j)
\ \prod_{j=1}^n\frac{dz_j}{2\pi i z_j}
\,,
\end{align}
where the integral is taken on the unit circle $|z_j|=1$, $j=1,\ldots,n$
(or a common circle of any radius, due to homogeneity) in the region $|q_1|,|q_3|>1$ and extended by analytic continuation everywhere else.
\end{thm}
\begin{proof}
In \cite{FKSW} the theorem is proved directly (on tensor products of Fock spaces) 
with the use of the identity \eqref{id-h}. 

In \cite{FJM}, it is shown that $\bI_n$, up to a constant, are 
Taylor coefficients of the transfer matrix corresponding to the Fock space $\F_2(u)$, where $\mu$ is the twist parameter. Then the commutativity follows from the standard argument with R-matrix.
\end{proof}

We note that in the construction of integrals of motion one can replace $\kfun_2(z)$  with functions
\begin{align*}
&\kfun_1(z)=\frac{\Theta_{\mu}(z)\Theta_{\mu}(q_1^{-1}z)}
{\Theta_{\mu}(q_2 z)\Theta_{\mu}(q_3 z)}\,,\qquad  
&\kfun_3(z)=\frac{\Theta_{\mu}(z)\Theta_{\mu}(q_3^{-1}z)}
{\Theta_{\mu}(q_1 z)\Theta_{\mu}(q_2 z)}\,.
\end{align*}
It is known that while individual integrals $\bI_n$ with $n\geq 2$ depend on this choice, the algebra generated by all $\bI_n$ does not.

Let us verify the commutativity with screenings.
\begin{thm}\label{A comm thm} 
The integrals of motion commute with all screening operators, 
\begin{align*}
[S_i,\bI_n]=0, \qquad i=0,1,\ldots,\ell,\ n\ge 1.
\end{align*}
\end{thm}
\begin{proof}
For  $i\neq 0$, this follows readily from \eqref{ScA}. 
We check the case $i=0$ using \eqref{ScA0}. 
By replacing $\omega_2(x)$ by $\omega_3(x)$ or $\omega_1(x)$ if necessary, it suffices to consider the two cases, 
 $c_1=c_{\ell+1}=2$, or $c_1=1,c_{\ell+1}=3$. 
We assume $|q_1|,|q_3|>1$.

First consider the case $c_1=c_{\ell+1}=2$, $S_0=S_0^-$. Then
\begin{align*}
[S_0,\bs\Lambda_1(z)]=const.
\cA^{(1)}(z)\,,
\quad 
\cA^{(1)}(z)=z:S_0(s_1^{-1}\mu^{-1}z)\bs\Lambda_1(z):\,.
\end{align*}
Noting the symmetry of the integrand and the commutativity 
$S_0(z)\bs e(w)=\bs e(w)S_0(z)$ as meromorphic functions, 
we obtain 
\begin{align*}
[S_0,\bI_n]&\propto 
\sum_{i=1}^n 
\int\!\!\cdots \!\!\int
\mathbf{e}(z_1)\cdots (\cA^{(1)}(z_i)-\cA^{(1)}(\mu z_i))\cdots \mathbf{e}(z_n)
\cdot \prod_{j<k}\kfun_2(z_k/z_j)
\ \prod_{j=1}^n\frac{dz_j}{2\pi i z_j}
\,
\\
&\propto
\int\!\!\cdots \!\!\int
(\cA^{(1)}(z_1)-\cA^{(1)}(\mu z_1))\mathbf{e}(z_2) \cdots \mathbf{e}(z_n)
\cdot \prod_{j<k}\kfun_2(z_k/z_j)
\ \prod_{j=1}^n\frac{dz_j}{2\pi i z_j}\,.
\end{align*}
Let us examine the  relevant poles of $\cA^{(1)}(z_1)\bs e(z_j)$. For symmetry reasons we consider $j=2$. 
We have
\begin{align*}
&\omega_2(z_2/z_1)\cA^{(1)}(z_1)\bs\Lambda_l(z_2)
=F(z_2/z_1) F(\mu^{-1}z_1/z_2) f^{(1)}_l(z_2/z_1)\times:\cA^{(1)}(z_1)\bs\Lambda_l(z_2):\,,
\end{align*}
where
\begin{align*}
&F(x)=\frac{(\mu x,\mu q_2x;\mu)_\infty}{(\mu q_1^{-1}x,\mu q_3^{-1}x;\mu)_\infty}\,,
\\
&f^{(1)}_l(x)=
\begin{cases}
\frac{1-q_2x}{1-q_1^{-1}x} & (l=1),\\
\frac{1-\mu q_2^{-1}x}{1-\mu q_1x} & (l=\ell+1),\\
 1 & (l\neq 1,\ell+1).\\
\end{cases}
\end{align*}
All integrals are taken over the unit circle. 
If $|\mu q_1|>1$, one easily checks that the $z_1$ integral vanishes. 
Otherwise one has to pick the residues at $z_1=q_1^{-1} z_2, \mu q_1 z_2$ coming from $f^{(1)}_l(z_2/z_1)$.
Using again the relation $q_2:S_0(s_1^{-1}z)\bs \Lambda_1(\mu z):\,\,=\,\,:S_0(s_1z)\bs \Lambda_{\ell+1}(z):$, 
we find that 
\begin{align*}
[S_0,\bI_n]&\propto 
\int\!\!\cdots \!\!\int
(\cA^{(2)}(z_2)-\cA^{(2)}(\mu q_1 z_2))\mathbf{e}(z_3) \cdots \mathbf{e}(z_n)
\cdot \prod_{k=3}^n \kfun^{(2)}_2(z_k/z_2)
\cdot \prod_{3\le j<k}\kfun_2(z_k/z_j)
\ \prod_{j=2}^n\frac{dz_j}{2\pi i z_j}\,,
\end{align*}
where 
\begin{align*}
&\cA^{(2)}(z)=z:S_0(s_1^{-1}q_1^{-1}\mu^{-1}z)\bs \Lambda_1(q_1^{-1}z)\bs\Lambda_1(z):\,,
\quad \omega^{(2)}_2(x)=\omega_2(q_1 x)\omega_2(x)\,.
\end{align*}
This current has properties similar to $\cA^{(1)}(z)$ except that $\mu$ is changed to $\mu q_1$, namely
\begin{align*}
&\omega^{(2)}_2(z_2/z_1)\cA^{(2)}(z_2)\bs\Lambda_l(z_3)
=F(z_3/z_2) F(q_1^{-1}\mu^{-1}z_2/z_3) f^{(2)}_l(z_3/z_2)\times:\cA^{(2)}(z_2)\bs\Lambda_l(z_3):\,,
\end{align*}
where $f^{(2)}_{\ell+1}(x)=f^{(1)}_{\ell+1}(q_1x)$ and $f^{(2)}_l(x)=f^{(1)}_l(x)$ ($l\neq\ell+1$). 
We repeat this process several times, 
until either there are no other integration variables left, or until
 $|\mu q_1^m|>1$ with some $m$ and the shifting $z\to \mu q_1^m z$ is no longer obstructed.

In the case $c_1=1$ and $c_{\ell+1}=3$, the analog of $f^{(1)}_l(x)$ becomes
\begin{align*}
&\tilde f^{(1)}_l(x)=
\begin{cases}
\frac{1-q_1x}{1-q_2^{-1}x} & (l=1),\\
\frac{1-\mu q_3^{-1}x}{1-\mu q_2x} & (l=\ell+1),\\
 1 & (l\neq 1,\ell+1).\\
\end{cases}
\end{align*}
There are no poles which obstruct the shift $z\to \mu z$.  
\end{proof}

\medskip

\section{Algebra $\cK$}\label{section BCD}
We define an algebra $\cK$ which produces the deformed $\cW$ currents of types $\ssB,\ssC,\ssD$ 
in the same way as the algebra $\E$ produces the deformed $\cW$ current of type $\ssA$. 
We think of algebra $\cK$ as a
``twisted"  \cite{O} analog of algebra $\E$ or as a ``coideal" subalgebra \cite{L}. 
\subsection{Algebra $\cK$}\label{sec K}

We define an algebra $\cK$ with generating currents
\begin{align*}
&E(z)=\sum_{n\in \Z}E_n z^{-n} \,,\quad
K^\pm(z)=\exp\bigl(\sum_{\pm r> 0}H_r z^{-r}\bigr)\,,
\end{align*}
and an invertible central element $C$. We set 
\begin{align*}
&K(z)=K^-(z)K^+(C^2 z)\,.
\end{align*}

The defining relations are as follows. 
\begin{align}
&g(z,w)E(z)E(w)+g(w,z)E(w)E(z)=
\frac{1}{\kappa_1}
\Bigl(g(z,w)\delta\bigl(C^2\frac{z}{w}\bigr) K(z)
+g(w,z)\delta\bigl(C^2\frac{w}{z}\bigr) K(w)\Bigr)\,,
\label{EE}\\
&K^{\pm}(z)K^{\pm}(w)=K^{\pm}(w)K^{\pm}(z)\,,\label{KK1}\\
&g(z,w) g(z,C^2w)\, K^+(z)K^-(w)=\bar g(z,w)\bar g(z, C^2w)\, K^-(w)K^+(z)\,,
\label{KK2}\\
&g(z,w)\, K^\pm(z)E(w)=\bar g(z,w) E(w)\,K^\pm(z)\,, 
\label{KE}\\
&\mathop{\mathrm{Sym}}_{z_1,z_2,z_3}\frac{z_2}{z_3}
[E(z_1),[E(z_2),E(z_3)]]
=\mathop{\mathrm{Sym}}_{z_1,z_2,z_3}X(z_1,z_2,z_3)\kappa_1^{-1}\delta(C^2 z_1/z_3)K^-(z_1)E(z_2)K^+(z_3)\,,
\label{EEE}
\end{align}
where
\begin{align*}
&\mathop{\mathrm{Sym}}_{z_1,\ldots,z_N}\ f(z_1,\dots,z_N) =\frac{1}{N!}
\sum_{\pi\in\mathfrak{S}_N} f(z_{\pi(1)},\dots,z_{\pi{(N)}})\,,\\
&X(z_1,z_2,z_3)=
\frac{(z_1+z_2)(z_3^2-z_1z_2)}{z_1z_2z_3}G(z_2/z_3)
+\frac{(z_2+z_3)(z_1^2-z_2z_3)}{z_1z_2z_3}G(z_1/z_2)
+\frac{(z_3+z_1)(z_2^2-z_3z_1)}{z_1z_2z_3}\,,
\end{align*}
and $G(w/z)$ stands for the power series expansion of $\bar g(z,w)/g(z,w)$ in $w/z$.
We note that as a rational function 
\begin{align*}
X(z_1,z_2,z_3)=\frac{\kappa_1 g(z_1,z_3)}{g(z_3,z_2)g(z_2,z_1)}(z_3+z_2)(z_1+z_2)(z_2^2-z_3z_1)\frac{z_2}{z_3 z_1}\,.
\end{align*}

The relations for $K^\pm(z)$ are equivalent to
\begin{align*}
[H_r,H_{s}]=- \delta_{r+s,0} \kappa_r\frac{1+C^{2r}}{r} \,.
\end{align*}
Note the difference to \eqref{hh}.

In the presence of quadratic relations, the Serre relation \eqref{EEE} can be reformulated in terms of 
correlation functions, see Lemma \ref{zero-cond} below.

Note that the Fourier coefficients of $K(z)$ are infinite series. 
Therefore the relation \eqref{EE} requires some justification. We proceed as follows.

We define the homogeneous grading of generators by $\deg E_n=n$, $n\in\Z$, $\deg H_{\pm r}=\pm r$, $r\in\Z_{>0}$. Consider the free algebra $\mc A$ generated by $E(z), H(z)$. Let $\widetilde {\mc A}$ be the completion of $\mc A$ with respect to the grading in the positive direction. The elements of $\widetilde {\mc A}$ are series of the form $\sum_{i>-N}f_ig_i$, where $f_i,g_i\in \mc A$ and $\deg g_i=i$. Then we consider $\cK$ to be a graded quotient algebra of $\widetilde {\mc A}$.

We call a $\cK$ module $V$ admissible if $V$ is a graded module and $V=\oplus_{i<N} V_i$, 
where $V_i=\{v\in V, \deg v =i\}$, and if in addition $C$ is diagonalizable. 
If $V$ is an admissible module, then for any $v\in V$, the series $K^+(z)v$ is a polynomial in $z^{-1}$ with values in $V$ and therefore $K(z)v$ is well defined.

Finally, we note that there are obvious automorphisms of $\cK$:
\begin{align}
&\iota:\cK \to \cK, &\qquad E(z)\mapsto -E(z),\quad  &K^\pm(z)\mapsto K^\pm(z),&\label{iota} \\
&\tau_a:\cK \to \cK, &\qquad E(z)\mapsto E(az),\quad  &K^\pm(z)\mapsto K^\pm(az) &(a\in\C^\times).\label{tau}
\end{align}

\subsection{Left comodule structure}\label{sec comodule}
The algebra $\cK$ does not seem to have a natural coproduct. 
Instead it is a comodule over the quantum toroidal algebra $\E$.

Consider the tensor product algebra of $\E$ and $\cK$. 
We denote by $\E\tilde\otimes \cK$ the completion of the tensor algebra $\E\otimes \cK$ 
with respect to the homogeneous grading in the positive direction.
\begin{prop}
The following map $\Delta:\cK\to \E\tilde {\otimes}\cK$ endows
 $\cK$ with a structure of a left $\E$-comodule:
\begin{align}\label{comodule formula}
&\Delta E(z)=e(C_2^{-1}z)\otimes K^+(z)+1\otimes E(z)+\ft(C_2z)\otimes K^-(z)\,,\notag
\\
&\Delta K^+(z)=\psi^+(C^{-1}_1C^{-1}_2z)\otimes K^+(z)\,,
\\
&\Delta K^-(z)=\psi^-(C_2z)^{-1}\otimes K^-(z)\,,\notag
\\
&\Delta C=C\otimes C\,,\notag
\end{align}
where $C_1=C\otimes 1$, $C_2=1\otimes C$.
\qed 
\end{prop}

\begin{proof}
Checking that $\Delta$ preserves the relations \eqref{EE}--\eqref{EEE}, especially the last one, 
demands a  straightforward  but long calculation. 
We use the identity
\begin{align*}
&
\Bigl(\frac{z_1}{z_2}-\frac{z_2}{z_1}\Bigr)(g_{1,3}g_{2,3}-g_{3,1}g_{3,2})
+\Bigl(\frac{z_2}{z_3}-\frac{z_3}{z_2}\Bigr)
g_{3,1}(g_{2,3}+g_{3,2})
+\Bigl(\frac{z_1}{z_3}-\frac{z_3}{z_1}\Bigr)
(g_{1,3}+g_{3,1})g_{2,3}   
\\
&=\kappa_1\Bigl(1-\frac{z_3^2}{z_1z_2}\Bigr)(z_1+z_3)(z_2+z_3)z_3 g_{1,2}\,,
\end{align*}
where $g_{i,j}=g(z_i,z_j)$. 

Coassociativity $(\Delta\otimes \id)\circ \Delta=(\id\otimes\Delta)\circ \Delta$ 
mapping $\cK\to\E\tilde\otimes\E\tilde\otimes \cK$ is a short direct calculation.  
The counit property $(\epsilon\otimes \id)\circ \Delta=\id$ mapping $\cK\to\cK$ is obvious.
\end{proof}
\medskip

\subsection{Boundary Fock modules}\label{sec boundary}
We describe several representations of algebra $\cK$, given in one free boson.

For a complex number $k\in\C^\times$, 
let $\cH_k$ be the Heisenberg algebra generated by 
$\{H_r\}_{r\neq0}$ with the relations $[H_r,H_s]=-\delta_{r+s,0}\kappa_r (1+k^{2r})/r$.

Keeping in mind the Dynkin types $\ssB$ and $\ssC,\ssD$, for $c\in\{1,2,3\}$ we set
\begin{align*}
\cH_c^\ssB=\cH_{s_c^{1/2}}\,, \qquad \cH_c^\ssCD=\cH_{s_c^{-1}}\,,
\end{align*}
and denote by $\F^\ssB_c$, $\F^\ssCD_c$ the corresponding Fock modules of the Heisenberg algebra.

\medskip

We have three $\cK$ modules of type $\ssB$. Set
\begin{align*}
k^\ssB_c=\frac{(1+s_c)(s_d-s_b)}{\kappa_1}\,, \qquad (c,d,b)=cycl(1,2,3).
\end{align*}
Define a vertex operator $\tilde{K}^\pm_c(z)$ by 
\begin{align*}
 \tilde K^\pm_c(z)=\exp\Big(\sum_{\pm r>0} \frac{1}{1+s_c^{-r}}H_rz^{-r}\Big).
\end{align*}
We have $\tilde{K}^\pm_c(z) \tilde{K}^\pm_c(s_c z)=K^\pm(z)$.
Set $\tilde{K}_c(z)=\tilde{K}^-_c(z)\tilde{K}^+_c(s_c z)$.

\begin{prop}\label{prop B boundary} For $c\in\{1,2,3\}$, the map $\cK\to \cH^\ssB_c$ sending
\begin{align*}
E(z)\mapsto k^\ssB_c \tilde{K}_c(z)\,,\quad K^{\pm}(z)\mapsto K^{\pm}(z)\,,\quad C\mapsto s_c^{1/2}\,,
\end{align*} 
endows $\F^\ssB_c$ with a structure of a $\cK$ module of level $s_c^{1/2}$.
\end{prop}
\begin{proof}
The proposition is proved by a direct computation.
\end{proof}

We also have three $\cK$ modules of type $\ssCD$.
\begin{prop}\label{prop CD boundary}
For $c\in\{1,2,3\}$, the map $\cK\to \cH^\ssCD_c$ sending
\begin{align*}
&E(z)\mapsto 0\,,\quad K^{\pm}(z)\mapsto K^{\pm}(z)\,,
\quad C\mapsto s_c^{-1}\,,
\end{align*} 
endows $\F^\ssCD_c$ with a structure of a $\cK$ module of level $s_c^{-1}$.
\end{prop}
\begin{proof}
The proposition is proved by a direct computation.
\end{proof}

\bigskip

Using the comodule map $\Delta$ we obtain the following corollary.
\begin{cor}\label{include}
There exists an algebra homomorphism 
$\cK\to \E\tilde \otimes \cH^\ssB_c$ such that
\begin{align*}
&E(z)\mapsto  e(s^{-1/2}_cz)\otimes K^+(z)+k_c^\ssB  \otimes \tilde{K}_c(z)
+\ft(s_c^{1/2}z)\otimes K^-(z)\,,
\\
&K^+(z)\mapsto \psi^+(C^{-1}_1s_c^{-1/2}z)\otimes K^+(z)
\,,
\quad K^-(z)\mapsto  {\psi^-(s_c^{1/2}z)}^{-1}\otimes K^-(z)\,,\\
&C\mapsto s_c^{1/2} C_1\,.
\end{align*} 
Similarly, there exists an algebra homomorphism $\cK\to \E\tilde\otimes \cH^\ssCD_c$ such that
\begin{align*}
&E(z)\mapsto e(s_c z) \otimes K^+(z)+\ft(s_c^{-1}z) \otimes K^-(z)\,,
\\
&K^+(z)\mapsto  {\psi^+(C^{-1}_1s_c z)}\otimes K^+(z)
\,,
\quad K^-(z)\mapsto { \psi^-(s_c^{-1}z)}^{-1}\otimes K^-(z)\,,\\
&C\mapsto s_c^{-1} C_1\,.
\end{align*}\qed
\end{cor}
It seems likely that the homomorphisms in Corollary \ref{include} are injective. 
If so, it would give us an inclusion of algebra $\cK$ into algebra $\E$ extended by an extra Heisenberg algebra. 
\medskip

Note that twisting the boundary modules by automorphism \eqref{iota} we obtain a new set of boundary modules. Twisting by automorphisms \eqref{tau} leads to isomorphic modules.

\subsection{Root currents $A_i(z)$ of types $\ssB,\ssC,\ssD$}\label{sec B roots}
Fix a sequence of colors $c_1,\ldots,c_{\ell+1}\in\{1,2,3\}$, and consider 
a $\cK$ module defined as 
a tensor product of $\ell$ Fock modules of $\E$ with a boundary Fock module $\F_c^\ssB$ or $\F_c^\ssCD$:
\begin{align}
&\F_{c_1}(u_1)\otimes\cdots\otimes \F_{c_{\ell}}(u_{\ell})\otimes \F_{c_{\ell+1}}^\ssB\,, 
\label{B-mod} \\ 
&\F_{c_1}(u_1)\otimes\cdots\otimes \F_{c_{\ell}}(u_{\ell})\otimes \F_{c_{\ell+1}}^\ssCD\,.
\label{C-mod}
\end{align}
We say that the tensor product has 
\begin{align*}
&\text{type $\ssB$}\qquad \text{for \eqref{B-mod}}\,,\\
&\text{type $\ssC$}\qquad \text{for \eqref{C-mod} with $c_\ell\neq c_{\ell+1}$}\,,\\
&\text{type $\ssD$}\qquad \text{for \eqref{C-mod} with $c_\ell= c_{\ell+1}$}\,.
\end{align*}

Let $C_{\ell+1}$ denote the level of the boundary Fock module:
\begin{align}
&C_{\ell+1}=\begin{cases}
s_{c_{\ell+1}}^{1/2} & \text{for type $\ssB$},\\
s_{c_{\ell+1}}^{-1} &  \text{for type $\ssC,\ssD$}. \\
\end{cases}
\label{lev-bdry}
\end{align}
The total level is 
\begin{align}
C=C_{\ell+1}  
\prod_{i=1}^\ell s_{c_i}.\label{level-BCD}
\end{align}
By the comodule formula \eqref{comodule formula} and the coproduct formulas 
\eqref{e coproduct}, \eqref{psi+ coproduct}, current $E(z)$ acts as a sum of vertex operators in 
$\ell+1$ bosons of the form
\begin{align*}
&E(z)= \sum_{i=1}^\ell b_{c_i}\Lambda_i(z) + k_{c_{\ell+1}}^\ssB\Lambda_0(z) 
+\sum_{i=1}^\ell b_{c_i}\Lambda_{\bar i}(z)\qquad &&\text{for type $\ssB$}\,,
\\
&E(z)= \sum_{i=1}^\ell b_{c_i}\Lambda_i(z) +\sum_{i=1}^\ell b_{c_i}\Lambda_{\bar i}(z)
\qquad &&\text{for type $\ssC,\ssD$}\,.
\end{align*}
In these formulas,
\begin{align*}
&\Lambda_i(z)=1\otimes\cdots\otimes 1
\otimes \overset{\underset{\smile}{i}}{V_{c_i}(a_{i} z;u_i)}
\otimes \psi^+(s^{-1}_{c_{i+1}}a_{i+1} z)\otimes 
\cdots \otimes \psi^+(s^{-1}_{\ell}a_{\ell} z)\otimes K^+(z)\, \qquad (i=1,\dots,\ell), 
\\
&\La_0(z)=1\otimes\cdots\otimes 1\otimes \tilde K_{c_{\ell+1}}(z)\,, 
\\
&\Lambda_{\bar i}(z)=1\otimes\cdots\otimes 1
\otimes \overset{\underset{\smile}{i}}{V_{c_i}^{-1}(a_{i}^{-1} z;u_i)}
\otimes \psi^-(a_{i+1}^{-1}z)^{-1}\otimes 
\cdots \otimes \psi^-(a_{\ell}^{-1}z)^{-1}\otimes K^-(z)\, \qquad (i=1,\dots,\ell),
\end{align*}
where $a_i$'s are given by 
\begin{align}\label{a}
a_i=C_{\ell+1}^{-1}  
\prod_{j=i+1}^\ell s_{c_j}^{-1}.
\end{align}

\bigskip

We have the following contractions:
\begin{align}
&\mc C(\La_i(z),\La_j(w))=
\begin{cases}
-\kappa_1& (i\prec j\,,i\neq \bar j)\\
0 & (i\succ j\,,i\neq \bar j)\\
\end{cases}\,,
\notag\\
&\mc C(\La_i(z),\La_i(w))=
\begin{cases}
-\frac{\kappa_1}{1-q_{c_i}}& (i\neq 0)\\
-\frac{\kappa_1}{1+s_{c_{\ell+1}}}& (i=0)\\
\end{cases}\,,
\label{LaContractionBCD}\\
&\mc C(\La_i(z),\La_{\bar i}(w))=-\kappa_1+\frac{\kappa_1 q_{c_i}}{1-q_{c_i}}a^{-2}_{i}\quad (1\le i\le \ell),\notag\\
&\mc C(\La_{\bar i}(z),\La_{i}(w))=\frac{\kappa_1}{1-q_{c_i}}a^{2}_{i}\quad (1\le i\le \ell).\notag
\end{align}
Here the indices are ordered as $1\prec\cdots\prec \ell\prec 0\prec\bar{\ell}\prec\cdots\prec\bar{1}$, and we set
$c_{\bar i}=c_i$. 

As before, to each neighboring pair of Fock spaces we associate a current $A_i(z)$. Namely, for $i=1,\dots,\ell-1$, 
we define $A_i(z)$ similarly as in \eqref{A},
\begin{align}
A_i(z)=\,\,:\frac{\Lambda_i(a_i^{-1}z)}{\Lambda_{i+1}(a_i^{-1}z)}:\,\,
=\,\,:\frac{\Lambda_{\overline{i+1}}(a_i z)}{\Lambda_{\bar{i}}(a_i z)}:\,.
\label{A-i}
\end{align}
The second equality is due to the identity \eqref{id-VO}.

In addition we define a current $A_\ell(z)$ for each type as follows.
\begin{align}
A_\ell (z)
&=\,\,:\frac{\Lambda_\ell(a^{-1}_\ell z)}{\Lambda_{0}(a^{-1}_\ell z)}:
\,\,=\,\,:\frac{\Lambda_0(a_\ell z)}{\Lambda_{\bar\ell}(a_\ell z)}:
&&\text{for type $\ssB$},
\label{B-ell}
\\
A_\ell(z)  
&=\,\,:\frac{\Lambda_\ell(z)}{\Lambda_{\bar \ell}(z)}:
&&\text{for type $\ssC$},
\label{C-ell}\\
A_\ell (z)
&=\,\, :\frac{\Lambda_{\ell-1}(z)}{\Lambda_{\bar \ell}(z)}:
\,\,= \,\,:\frac{\Lambda_{\ell}(z)}{\Lambda_{\overline{\ell-1}}(z)}:
&&\text{for type $\ssD$}.
\label{D-ell}
\end{align}

We study the contractions of the root currents $A_i(z)$. Denote as before $B_{i,j}=\mc C(A_i(z),A_j(w))$. 

The contractions of $B_{i,j}$ with $i,j\neq \ell$ are clearly the same as in type $\ssA$ 
and are given in \eqref{A-contr}. The new feature is the contractions $B_{i,\ell}, B_{\ell,i}$. 

First, we have
\begin{align*}
B_{i,\ell}=B_{\ell,i}=0, \qquad (i<\ell-2).
\end{align*}
Then for type $\ssB$ we have
\begin{align}\label{B cont}
&B_{\ell-2,\ell}=B_{\ell,\ell-2}=0\,,\qquad
B_{\ell-1,\ell}=B_{\ell,\ell-1}=\frac{t_1t_2t_3}{t_{c_{\ell}}}\,,\\
&B_{\ell,\ell}=
-\frac{t_1t_2t_3}{t_{c_{\ell}}t_{c_{\ell+1}}}(s^{1/2}_{c_{\ell+1}}-s_{c_{\ell+1}}^{-1/2})
(s_{c_\ell}s^{1/2}_{c_{\ell+1}}+s_{c_\ell}^{-1}s^{-1/2}_{c_{\ell+1}})\,.
\end{align}
For type $\ssC$ we have
\begin{align}\label{C cont}
&B_{\ell-2,\ell}=B_{\ell,\ell-2}=0\,,\qquad 
B_{\ell-1,\ell}=B_{\ell,\ell-1}=\frac{t_1t_2t_3}{t_{c_{\ell}}}(s_{c_{\ell+1}}+s_{c_{\ell+1}}^{-1})\,,\\
&B_{\ell,\ell}=-\frac{t_1t_2t_3}{t_{c_{\ell}}}(s_{c_{\ell+1}}+s_{c_{\ell+1}}^{-1})
(s_{c_\ell}s_{c_{\ell+1}}^{-1}+s_{c_\ell}^{-1}s_{c_{\ell+1}})\,.
\end{align}
Finally, for type $\ssD$ we have
\begin{align}\label{D cont}
&B_{\ell-2,\ell}=B_{\ell,\ell-2}=\frac{t_1t_2t_3}{t_{c_{\ell-1}}}\,,\qquad
B_{\ell-1,\ell}=B_{\ell,\ell-1}=
\frac{t_1t_2t_3}{t_{c_{\ell-1}}t_{c_{\ell}}}(s_{c_{\ell-1}}s^{-1}_{c_\ell}-s^{-1}_{c_{\ell-1}}s_{c_\ell})\,,
\\
&B_{\ell,\ell}=
-\frac{t_1t_2t_3}{t_{c_{\ell-1}}t_{c_{\ell}}}(s_{c_{\ell-1}}s_{c_\ell}-s^{-1}_{c_{\ell-1}}s^{-1}_{c_\ell})\,.
\end{align}
Note that, in particular,  $B_{ij}=B_{ji}$ in all cases.

We illustrate these contractions in the library of Appendix \ref{library}. Here we consider the simplest case when 
all Fock spaces are of the same color to describe the connection to Cartan matrices of types $\ssB,\ssC,\ssD$.

Let $c_i=2$, $i=1,\dots,\ell$. Then the only nonzero $B_{i,j}$ with both $i\neq \ell$ and $j\neq \ell$ are
$B_{i,i}=-t_1t_3(s_2+s_2^{-1})$ and
$B_{i,i-1}=B_{i-1,i}=t_1t_3$.

Consider now the case of type $\ssB$. We have $C^2=q_2^{\ell}s_{c_{\ell+1}}$.

Let $c_{\ell+1}=1$.
Then $B_{\ell,\ell-1}=B_{\ell-1,\ell}=t_1t_3$, and
$B_{\ell,\ell}=-t_3(s_1^{1/2}-s_1^{-1/2})(s_2s_1^{1/2}+s_2^{-1}s_1^{-1/2})$. 
The other entries involving index $\ell$ are zero.

Thus the matrix $-B t_3^{-1}(s_1^{1/2}-s_1^{-1/2})^{-1}$ 
coincides with type $\ssB$ matrix $B(q,t)$ (2.4) of \cite{FR1} under the identification 
$q=s_1^{1/2}$, $t=s_2^{-1}$. In particular, in the limit $s_1=s_2=s_3=1$ 
it recovers the symmetrized Cartan matrix of type $\ssB$.

Moreover, all the terms in $E(z)$ are obtained 
from the first one by multiplications by $A_i(z)^{-1}$ as in the natural representation of $B_\ell$ shown below
 (we do not show the arguments of currents or constants in front of vertex operators here).

\bigskip

\begin{tikzcd}
    \La_1\arrow{r}{A_1^{-1}} & \La_2 \arrow{r}{A_2^{-1}} & \cdots\arrow{r}{A_{\ell-1}^{-1}} 
&\La_{\ell}\arrow{r}{A_{\ell}^{-1}} & \La_0\arrow{r}{A_{\ell}^{-1}} 
& \La_{\bar\ell}\arrow{r}{A_{\ell-1}^{-1}} & \cdots\arrow{r}{A_{2}^{-1}} 
& \La_{\bar 2}\arrow{r}{A_{1}^{-1}} &\La_{\bar 1}
\end{tikzcd}
\bigskip

Therefore we depict this case by the following Dynkin diagram.
\begin{align}\label{dynkin B ex}
\begin{tikzpicture}[baseline=(origin.base)] 
\dynkin[root radius=.1cm, edge length=1.3cm, labels={2,2,2,2,1}]{B}{oo.ooo}
\end{tikzpicture}
\end{align}
The case $c_{\ell+1}=3$ is similar. 

However, the deformed Cartan matrix in the case of $c_{\ell+1}=2$ is essentially different; it corresponds to the case studied in \cite{BL}, see also Section 7 of \cite{FR1}. We have $B_{\ell,\ell-1}=B_{\ell-1,\ell}=t_1t_3$, and
$B_{\ell,\ell}=-t_1t_3(s_2-1+s_2^{-1})$.
In this case $-Bt_3^{-1}t_1^{-1}$ in the limit $s_1=s_2=s_3=1$ is equal to 
 symmetrized Cartan matrix of type ${\mathfrak{osp}}(1,2\ell)$. 
We associate to this case the following diagram.
\begin{align}\label{dynkin B-BL ex}
\begin{tikzpicture}[baseline=(origin.base)] 
\dynkin[root radius=.1cm, edge length=1.3cm, labels={2,2,2,2,2}]{B}{oo.oo*}
\end{tikzpicture}
\end{align}
Note that the symmetrized Cartan matrix of type ${\mathfrak{so}}(2\ell+1)$ is twice the  symmetrized Cartan matrix of type ${\mathfrak{osp}}(1,2\ell)$. Moreover, the weight diagrams of the ${\mathfrak{so}}(2\ell+1)$ modules and ${\mathfrak{osp}}(1,2\ell)$ modules are the same and the nilpotent subalgebras in both cases are the same. Therefore, for the purposes of this paper either diagram fits. However, we prefer not to use the affine node to distinguish between these two finite types which was suggested by \cite{FR1}. Instead the affine node distinguishes between affine versions of the deformed Cartan matrices, see Section \ref{sec B affine}.

\medskip

Let us now switch to the types $\ssC$ and $\ssD$. We have $C^2=q_2^{\ell}q_{c_{\ell+1}}^{-1}$.

When $c_{\ell+1}=1$, we make the choice of type $\ssC$. Then the nontrivial entries with index $\ell$ are
$B_{\ell,\ell-1}=B_{\ell-1,\ell}=t_3(s^2_1-s_1^{-2})$, and
$B_{\ell,\ell}=-t_3(s^2_1-s_1^{-2})(s_1s_2^{-1}+s_1^{-1}s_2)$.

Thus the matrix $-B t_3^{-1}t_1^{-1}$ coincides with type $\ssC$ 
matrix $B(q,t)$ of \cite{FR1} under the identification $q=s_1$, $t=s^{-1}_3$. 
In particular, in the limit $s_1=s_2=s_3=1$ it recovers the symmetrized Cartan matrix of type $\ssC$. 

Again, current $E(z)$ looks as a vector representation of type $C$.
\bigskip
$$
\begin{tikzcd}
\La_1\arrow{r}{A_1^{-1}} & \La_2 \arrow{r}{A_2^{-1}} & \cdots
\arrow{r}{A_{\ell-1}^{-1}} 
&\La_{\ell}\arrow{r}{A_{\ell}^{-1}} &\La_{\bar\ell}\arrow{r}{A_{\ell-1}^{-1}} & \cdots\arrow{r}{A_{2}^{-1}} & \La_{\bar 2}\arrow{r}{A_{1}^{-1}} &\La_{\bar 1}
\end{tikzcd}
$$
\bigskip

We associate to it the following Dynkin diagram.
\begin{align}\label{dynkin C ex}
\begin{tikzpicture}[baseline=(origin.base)] 
\dynkin[root radius=.1cm, edge length=1.3cm, labels={2,2,2,2,1}]{C}{oo.ooo}
\end{tikzpicture}
\end{align}
The case $c_{\ell+1}=3$ is similar. 

\medskip 

When $c_{\ell+1}=2$, we make a choice of type $\ssD$. Then nonzero entries involving index $\ell$ are
$B_{\ell,\ell-2}=B_{\ell-2,\ell}=t_1t_3$, and $B_{\ell,\ell}=-t_1t_3(s_2+s_2^{-1})$.

Thus the matrix $-B t_3^{-1}t_1^{-1}$ 
coincides with type $\ssD$ 
matrix $B(q,t)$ of \cite{FR1} under the identification $q=s_1$, $t=s_3^{-1}$. 
In particular, in the limit $s_1=s_2=s_3=1$ it recovers the symmetrized Cartan matrix of type $\ssD$.

Again, current $E(z)$ looks as a vector representation of type $D$.

\bigskip

$$
\begin{tikzcd}
 &&&&  \La_{\ell}\arrow{rd}{A_{\ell}^{-1}} && \\
    \La_1\arrow{r}{A_1^{-1}} & \La_2 \arrow{r}{A_2^{-1}} & \cdots
\arrow{r}{A_{\ell-2}^{-1}} 
&\La_{\ell-1}\arrow{rd}{A_{\ell}^{-1}}\arrow{ru}{A_{\ell-1}^{-1}} &{} 
& \La_{\overline{\ell-1}}\arrow{r}{A_{\ell-2}^{-1}} & \cdots\arrow{r}{A_{2}^{-1}} 
& \La_{\bar 2}\arrow{r}{A_{1}^{-1}} &\La_{\bar 1} \\
    &&&&  \La_{\bar\ell}\arrow{ru}{A_{\ell-1}^{-1}} && \\
\end{tikzcd}
$$

Thus we have the following Dynkin diagram.
\begin{align}\label{dynkin D ex}
\begin{tikzpicture}[baseline=(origin.base)] 
\dynkin[root radius=.1cm, edge length=1.3cm, labels={2,2,2,2,2}]{D}{oo.ooo}
\end{tikzpicture}
\end{align}

Thus we recover deformed $\cW$-algebras of types $\ssB,\ssC,\ssD$ described in \cite{FR1} 
in the case when all Fock spaces are of the same type. We remark that as in type $\ssA$, 
we have the diagonal Heisenberg $\Delta^{(\ell)}H_r$ commuting with $A_i(z)$, $i=1,\ldots,\ell$.
This extra boson allows us to have rational contractions between all the terms. 
\medskip

\subsection{Root current $A_0(z)$ of types $\ssB,\ssC,\ssD$}\label{sec B affine}
Define the dressed current $\Eb(z)$ depending on $\mu\in\C^\times$, $|\mu|<1$,  
\begin{align*}
\Eb(z)=E(z){K^+_\mu(z)}^{-1}\,,\quad K^+_\mu(z)=\prod_{s=0}^\infty K^+(\mu^{-s}z)\,.
\end{align*}
The Fourier coefficients of $\Eb(z)$ are elements of the algebra $\tilde \cK$, 
the algebra $\cK$ completed with respect to homogeneous grading in the positive direction.

\medskip 

Similarly to type $\ssA$, our motivation for the definition of the dressing current is twofold : 
we would like to have integrals of motion and 
we also would like to have the current $A_0(z)$ which produces a screening operator corresponding 
to the affine node. 
It turns out that the second requirement 
restricts $\mu$ to specific values, in stark contrast to type $\ssA$ where $\mu$ is arbitrary; 
see Example \ref{DynkinD4} and Theorems \ref{CD screening thm}, \ref{B screening thm} below.
There are 6 possible choices of $\mu$:  
\begin{align*}
C^2/\mu=s_{c_0}^{-1}\quad \text{or} \quad C^2/\mu=s_{c_0}^2\,, \qquad c_0\in\{1,2,3\}, 
\end{align*}
which match the values of $C^2$ in 6 boundary modules we have defined in Section \ref{sec boundary}.
We have been able to find 
integrals of motion only for the last 3 of them (types $\ssC$ and $\ssD$ below),
see Section \ref{sec B integrals}.

\medskip

We make the choice for $\mu$ and label our choices by types $\ssB, \ssC,\ssD$ as we did for the $\ell$-th node. Namely, we have the following choices for the affine node:
\begin{align*}
\text{type $\ssB$} &:\qquad \mu=C^2 s_{c_0}\,,\\
\text{type $\ssC$} &:\qquad \mu=C^2 s^{-2}_{c_0}\,,\quad c_0\neq c_1\,,\\
\text{type $\ssD$} &:\qquad \mu=C^2 s^{-2}_{c_0}\,,\quad c_0= c_1\,,
\end{align*}
$C$ being the level given by \eqref{level-BCD}.
We stress that the choice of $c_0$ is independent of other $c_i$, that is, it is independent of the choice of the $\cK$ representation. The type of the zeroth node
is not to be confused with the type of tensor product modules \eqref{B-mod}, \eqref{C-mod}.

To that we associate the affine Dynkin diagram, where the affine node has the prescribed type and the color is given the same way as for the 
$\ell$-th 
node, see \ref{sec library conventions}. 
Namely, we depict the affine node as follows.
\begin{align*}
&\ssB, \ c_0=c_{1}:
&&
\begin{tikzpicture}[baseline=(origin.base)] 
\dynkin[root radius=.1cm, edge length=1.3cm, labels={c_0}]{A}{*}
\end{tikzpicture} &&
\\
&\ssB, \ c_0\neq c_{1}: 
&&
\begin{tikzpicture}[baseline=(origin.base)] 
\dynkin[root radius=.1cm, edge length=1.3cm, labels={b}]{A}{o}
\end{tikzpicture} &\{c_0,c_1,b\}=\{1,2,3\}&
\\
&\ssC, \ c_0\neq c_{1}:
&&
\begin{tikzpicture}[baseline=(origin.base)] 
\dynkin[root radius=.1cm, edge length=1.3cm, labels={b}]{A}{o}
\end{tikzpicture} &\{c_0,c_1,b\}=\{1,2,3\}&
\\
&\ssD, \ c_0=c_2: &&
\begin{tikzpicture}[baseline=(origin.base)] 
\dynkin[root radius=.1cm, edge length=1.3cm, labels={c_0}]{A}{o}
\end{tikzpicture}  &&
\\
&\ssD, \ c_0\neq c_2:
&&
\begin{tikzpicture}[baseline=(origin.base)] 
\dynkin[root radius=.1cm, edge length=1.3cm, labels={b}]{A}{t}
\end{tikzpicture}  &\{c_0,c_2,b\}=\{1,2,3\}&
\\
\end{align*}

Current $\Eb(z)$ has the form 
\begin{align*}
&\Eb(z)= \sum_{i=1}^\ell b_{c_i}\bs \Lambda_i(z) + k_{c_{\ell+1}}^\ssB\bs \Lambda_0(z) 
+\sum_{i=1}^\ell b_{c_i}\bs \Lambda_{\bar i}(z)\qquad &&\text{for type $\ssB$}\,,
\\
&\Eb(z)= \sum_{i=1}^\ell b_{c_i}\bs\Lambda_i(z) +\sum_{i=1}^\ell b_{c_i}\bs\Lambda_{\bar i}(z)
\qquad&& \text{for types $\ssC$, $\ssD$}\,,
\end{align*}
where $\bs \Lambda_i(z)=\Lambda_i(z)\Delta^{(\ell)}{K^+_\mu(z)}^{-1}$.
For $\bs\Lambda_i(z)$ the contraction rule \eqref{cont bs La} applies. 

Define the current $A_0(z)$ of each type as follows. 
\begin{align}
&\quad :A_0 (s_{c_0}^{-1/2}z)A_0 (s_{c_0}^{1/2} z):
\,\,=\,\,:\frac{\bs \Lambda_{\bar 1}(\mu^{-1/2} z)}{\bs \Lambda_1(\mu^{1/2} z)}:\qquad
&&\text{for type $\ssB$},
\label{B-0}\\
&\quad A_0 (z)  
\,\,=\,\,:\frac{\bs \Lambda_{\bar 1}(\mu^{-1/2} z)}{\bs \Lambda_1(\mu^{1/2} z)}:\qquad
&&\text{for type $\ssC$}\,,
\label{C-0}\\
&\quad A_0 (z)= \,\,
:\frac{\bs \Lambda_{\bar 2}(\mu^{-1/2} z)}{\bs \Lambda_1(\mu^{1/2} z)}:
\,\,=\,\,:\frac{\bs \Lambda_{\bar 1}(\mu^{-1/2} z)}{\bs \Lambda_2(\mu^{1/2}  z)}:\qquad
&&\text{for type $\ssD$ }
\,.\label{D-0}
\end{align}
These definitions are to be compared with \eqref{B-ell}--\eqref{D-ell}.
Note that, for type $\ssB$, \eqref{B-ell} implies 
\begin{align*}
:A_\ell(s^{-1/2}_{c_{\ell+1}}z)A_\ell(s^{1/2}_{c_{\ell+1}}z):
\,\,=\,\,:\frac{\Lambda_\ell(z)}{\Lambda_{\bar\ell}(z)}:\,.
\end{align*}

The corresponding contractions $B_{i,j}$ are given by 
the same rule as \eqref{B cont}--\eqref{D cont} 
if we interchange $A_i(z)$ with $A_{\ell-i}(z)$ and $c_j$ with $c_{\ell+1-j}$. 

Namely we have 
\begin{align}
&B_{0,1}=B_{1,0}=\frac{t_1t_2t_3}{t_{c_1}}\,,\quad
B_{0,0}=
-\frac{t_1t_2t_3}{t_{c_0}t_{c_1}}(s^{1/2}_{c_0}-s_{c_0}^{-1/2})
(s^{1/2}_{c_0}s_{c_1}+s^{-1/2}_{c_0}s_{c_1}^{-1})
\quad \text{for type $\ssB$}\,;
\label{B cont0}\\
&B_{0,1}=B_{1,0}=\frac{t_1t_2t_3}{t_{c_1}}(s_{c_0}+s_{c_0}^{-1})\,,\quad
B_{0,0}=-\frac{t_1t_2t_3}{t_{c_1}}(s_{c_0}+s_{c_0}^{-1})(s_{c_0}s_{c_1}^{-1}+s_{c_0}^{-1}s_{c_1})
\quad \text{for type $\ssC$}\,;
\label{C cont0}\\
&B_{0,2}=B_{2,0}=\frac{t_1t_2t_3}{t_{c_2}}\,,\quad
B_{0,1}=B_{1,0}=\frac{t_1t_2t_3}{t_{c_1}t_{c_2}}(s^{-1}_{c_1}s_{c_2}-s_{c_1}s^{-1}_{c_2})\,,
\label{D cont0}\\
&B_{0,0}=-\frac{t_1t_2t_3}{t_{c_1}t_{c_2}}(s_{c_1}s_{c_2}-s^{-1}_{c_1}s^{-1}_{c_2})
\quad \text{for type $\ssD$}\,.
\notag
\end{align}
For $\ell>3$, all other $B_{i,j}$ involving index $0$ are zero. The full list of $B_{i,j}$ for small values of $\ell$ is given in Appendix \ref{library}.

\bigskip

We think of the matrix $\hat B=(B_{i,j})_{i,j=0}^\ell$ as the affinization of matrix $B$. In case  all colors except $0$ and $\ell$ are equal, e.g. $c_1=\cdots=c_{\ell}=2$, 
the corresponding Dynkin diagrams are those of non-exceptional affine type, where the colors $c_0$ and $c_\ell$ determine the ends of the diagram,  see
Table \ref{CartanTable} below. For comparison we include type $\ssA$ in the table.

\bigskip

\begin{center}
\renewcommand{\arraystretch}{1.7}
\begin{table}[H]
\begin{tabular}{|c|c|c|c|c|c|c|c|}
\hline
type & Fock space &$C^2$ & $C^2/\mu$ & $\det \hat C$\\
\hline
$\ssA^{(1)}_\ell$ &$ \F_2^{\otimes\ell}\otimes \F_2$
&$q_2^{\ell+1}$ &\text{arbitrary} & $(1-\mu)(s_2^{-\ell-1}-\mu^{-1}s_2^{\ell+1})$ \\
\hline
$\ssB^{(1)}_\ell$ &$ \F_2^{\otimes\ell}\otimes \F^\ssB_3$
&$q_2^{\ell}q_3^{1/2}$ & $q_2$  & $(s_2^2-s_2^{-2})(s_2^{\ell-1}s_3^{1/2}-s_2^{-\ell+1}s_3^{-1/2})$\\
\hline
$\ssC^{(1)}_\ell$ &$\F_2^{\otimes\ell}\otimes \F^\ssCD_3$
&$q_2^{\ell}q_3^{-1}$ & $q_3$ &  $(s_2-s_2^{-1})(s_2^{\ell}s_3^{-2}-s_2^{-\ell}s_3^{2})$\\
\hline
$\ssD^{(1)}_\ell$ &$ \F_2^{\otimes\ell}\otimes \F^\ssCD_2$
&$q_2^{\ell-1}$ & $q_2$ &  $(s_2+s_2^{-1})(s_2^2-s_2^{-2})(s_2^{\ell-2}-s_2^{-\ell+2})$\\
\hline
$\ssA^{(2)}_{2\ell}$ &$\F_2^{\otimes\ell}\otimes \F^\ssB_3$
&$q_2^{\ell}q_3^{1/2}$ & $q_3$ & $(s_2-s_2^{-1})(s_2^{\ell}s_3^{-1/2}-s_2^{-\ell}s_3^{1/2})$ \\
\hline
$\ssA^{(2)}_{2\ell-1}$&$\F_2^{\otimes\ell}\otimes \F^\ssCD_3$
&$q_2^{\ell}q_3^{-1}$ & $q_2$ & $(s_2^2-s_2^{-2})(s_2^{\ell-1}s_3^{-1}-s_2^{-\ell+1}s_3)$\\
\hline
$\ssD^{(2)}_{\ell+1}$  &$ \F_2^{\otimes\ell}\otimes \F^\ssB_3$
&$q_2^{\ell}q_3^{1/2}$ &$q_3^{-1/2}$ & $(s_2-s_2^{-1})(s_2^\ell s_3-s_2^{-\ell}s_3^{-1})$ \\
\hline
\end{tabular}
\caption{\label{CartanTable}}
\end{table}
\renewcommand{\arraystretch}{1}
\end{center}
\medskip

We draw the corresponding Dynkin diagrams, where labels are as in the finite case and the double circle denotes the affine node which corresponds to the dressing rather than a pair of modules.

\newpage 

\hfill \break

\noindent{ $\ssA^{(1)}_\ell$}
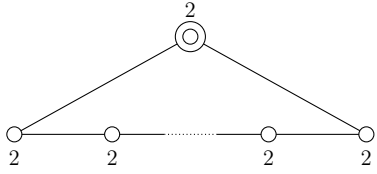
\vspace{-30pt}
\begin{center}
\begin{tikzpicture}
\dynkin[root radius=.1cm, edge length=1.3cm, affine mark=O,labels={{2},2,2,2,2}, extended]{A}{oo.oo}
\end{tikzpicture}
\end{center}

\vspace{50pt}

\noindent{ $\ssB^{(1)}_\ell$}
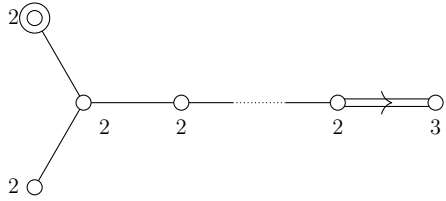
\vspace{-50pt}
\begin{center}
\begin{tikzpicture}
\dynkin[root radius=.1cm, edge length=1.3cm, labels={2,2,2,2,2,3},affine mark=O, extended]{B}{ooo.oo}
\end{tikzpicture}
\end{center}

\vspace{20pt}

\noindent{$\ssC^{(1)}_\ell$}
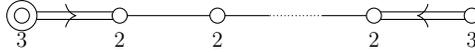
\vspace{-20pt}
\begin{center}
\begin{tikzpicture}
\dynkin[root radius=.1cm, edge length=1.3cm, affine mark=O, labels={3,2,2,2,3},extended]{C}{oo.oo}
\end{tikzpicture}
\end{center}

\vspace{50pt}

\noindent{$\ssD^{(1)}_\ell$}
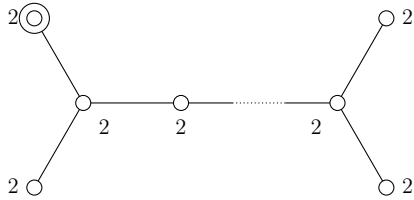
\vspace{-50pt}
\begin{center}
\begin{tikzpicture}
\dynkin[root radius=.1cm, edge length=1.3cm, affine mark=O, labels={2,2,2,2,2,2,2},extended]{D}{ooo.ooo}
\end{tikzpicture}
\end{center}

\vspace{20pt}

\noindent{$\ssA^{(2)}_{2\ell}$}
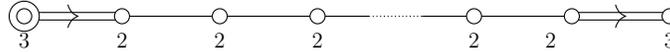
\vspace{-20pt}
\begin{center}
\begin{tikzpicture}
\dynkin[root radius=.1cm, backwards, edge length=1.3cm, affine mark=o, labels={3,2,2,2,2,2,3}]{A}[2]{oo.oooO}
\end{tikzpicture}
\end{center}

\vspace{50pt}

\noindent{$\ssA^{(2)}_{2\ell-1}$}
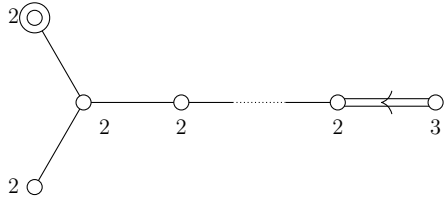
\vspace{-50pt}
\begin{center}
\begin{tikzpicture}
\dynkin[root radius=.1cm,reverse arrows, edge length=1.3cm, labels={2,2,2,2,2,3},affine mark=O, extended]{B}{ooo.oo}
\end{tikzpicture}
\end{center}

\vspace{30pt}

\noindent{$\ssD^{(2)}_{\ell+1}$}

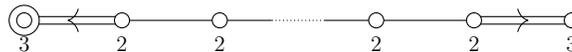
\begin{center}
\vspace{-20pt}
\begin{tikzpicture}
\dynkin[root radius=.1cm, edge length=1.3cm, affine mark=O,  ordering=Kac,labels={3,2,2,2,2,3}]{D}[2]{oo.ooo}
\end{tikzpicture}
\end{center}
\bigskip

\subsection{The $qq$-characters.}\label{sec qq B}
As in type $\ssA$ we write $\hat B=\hat D\hat C$, choosing $\hat D=\mathrm{diag}(d_0,\ldots,d_\ell)$ in such a way that 
${C}_{i,i}$ becomes one of the following, see Appendix \ref{library}:
\begin{align*}
s_c+s_c^{-1}\,,\quad s_bs_c^{-1}+ s_b^{-1}s_c\,,\quad  s_bs_c^{1/2}+s_b^{1/2}s_c\,,\quad 
s_c^{3/2}+s_c^{-3/2}\,,\qquad s_c-s_c^{-1}\,.
\end{align*}
We then introduce the currents $Y_i(z)$ following the rule \eqref{defY}. Then the current $\bs E(z)$ follows the $qq$-character.

We give the following example illustrating various phenomena.

\begin{example}\label{DynkinD4}
Consider the following affine Dynkin diagram.

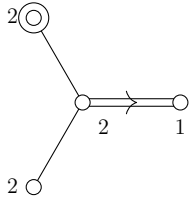
\begin{center}
\begin{tikzpicture}
\dynkin[root radius=.1cm, edge length=1.3cm, labels={2,2,2,1},affine mark=O, extended]{B}{ooo}
\end{tikzpicture}
\end{center}
We label the nodes in the standard way: double circled affine node is the zeroth node, the middle point is node $2$ and the shorter node is node $3$.

This diagram corresponds to the $\F_2\otimes \F_2\otimes \F_2\otimes \F_1^{\ssB}$ of level $C=s_2^3s_1^{1/2}$ with the dressing parameter $\mu=C^2s_2^{-2}=s_1s_2^4$.

We have the following deformed Cartan matrix (see Appendix \ref{library}).
$$
\hat C=\begin{pmatrix}
s_2+s_2^{-1} & 0 & -1 & 0 \\
0 & s_2+s_2^{-1} & -1 & 0 \\
-1 & -1 & s_2+s_2^{-1} & -1 \\
0 & 0& -s_1^{1/2}-s_1^{-1/2} & s_1^{1/2}s_2+s_1^{-1/2}s_2^{-1}\\
\end{pmatrix}.
$$
We use notation $Y_l(s_1^{a/2}s_2^{b/2}z)=\bs l_{a,b}$, $l=0,1,2,3$. Then we have
\begin{align*}
&A_0(z)=\bs 0_{0,-2}\bs 0_{0,2}\bs 2_{0,0}^{-1}, \qquad
&&A_1(z)= \bs 1_{0,-2}\bs 1_{0,2}\bs 2_{0,0}^{-1}, \\
&A_2(z)=\bs 0_{0,0}^{-1}\bs 1_{0,0}^{-1}\bs 2_{0,-2}\bs 2_{0,2} \bs 3_{1,0}^{-1}\bs 3_{-1,0}^{-1},  
&&A_3(z)=\bs 2_{0,0}^{-1}\bs 3_{1,2}\bs 3_{-1,-2}.
\end{align*}

Then the dressed current $\bs E(z)$ takes the form (ignoring the constants and an overall shift) of the $qq$-character with 7 monomials:
$$
\chi=\bs 0_{0,4}^{-1}\bs 1_{0,0}+\bs 0_{0,4}^{-1}\bs 1_{0,4}^{-1}\bs 2_{0,2}+\bs 2_{0,6}^{-1}\bs 3_{-1,4}\bs 3_{1,4} +
\bs 3_{1,4}\bs 3_{1,8}^{-1}+ \bs 2_{2,6}\bs 3_{3,8}^{-1}\bs 3_{1,8}^{-1}+\bs 0_{2,8}\bs 1_{2,8}\bs 2_{2,10}^{-1}+\bs 0_{2,8}\bs 1^{-1}_{2,12}.
$$

Note the following features of this $qq$-character.
We follow the terminology of Section \ref{q char sec}.

Let $\C P$ be the group ring of the weight lattice generated by fundamental weights $\omega_l$, $l=1,\dots,\ell$. Let
$\rho:\mc A \to \C P$ be the ring homomorphism sending $\bs l_{a,b}\mapsto \omega_l$, $l=1,\dots,\ell$, $\bs 0_{a,b}\mapsto 1$. Then $\rho(\chi)$ is the character of vector representation of $\ssB_3$.

The initial monomial has $\bs 1_{0,0}$, the final monomial has $\bs 1_{2,12}^{-1}$. The shift $s_1s_2^6$ is $C^2$.

The ratios between neighboring monomials are $A_1(s_2z)$, $A_2(s_2^2z)$, $A_3(s_2^3z)$, $A_3(s_1s_2^3z)$, $A_2(s_1s_2^4z)$, and $A_1(s_1s_2^{5}z)$ respectively.
The first and the 6th monomial are $\bs 1$-dominant, the second and the 5th are $\bs 2$-dominant, all these ``generate" $q$-characters of 2 dimensional $\U_q\widehat{\mathfrak{sl}}_2$-modules.
The third monimial is $\bs 3$-dominant and ``generates" the $q$-character of a
3 dimensional $\U_q\widehat{\mathfrak{sl}}_2$-module.
This is the reason $\bs E(z)$ commutes with screening operators $S_i$, $i=1,\dots,\ell$, see Section \ref{sec B screenings}.

The 6ht and the 7th monomial are $\bs 0$-dominant.They ``generate" monomials which are $\mu$-shifts of the first and the second monomial:
\begin{align*}
&A_0^{-1}(s_1s_2^5)\bs 0_{2,8}\bs 1_{2,8}\bs 2_{2,10}^{-1}=\bs 0_{2,12}^{-1}\bs 1_{2,8}=\tau_\mu(\bs0_{0,4}^{-1}\bs1_{0,0}), \\
&A_0^{-1}(s_1s_2^5)\bs 0_{2,8}\bs 1_{2,12}^{-1}=\bs0_{2,12}^{-1}\bs 1_{2,12}^{-1}\bs 2_{2,10}=\tau_\mu(\bs 0_{0,4}^{-1}\bs 1_{0,4}^{-1}\bs 2_{0,2}).
\end{align*}
This is the reason for our choice of $\mu$ and why the screening $S_0$ commutes with the integral of $\bs E(z)$.

\medskip

Now we consider the same Dynkin diagram and the same deformed Cartan matrix but treat node 3 as affine:

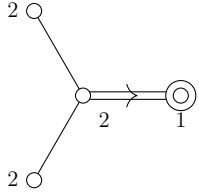
\begin{center}
\begin{tikzpicture}
\dynkin[root radius=.1cm, edge length=1.3cm, labels={2,2,2,1},affine mark=o, extended]{B}{ooO}
\end{tikzpicture}
\end{center}
Then we consider the $\ssD_3$ current corresponding to $\F_2\otimes \F_2\otimes \F_2\otimes \F_2^{CD}$ with the dressing of type $\ssB$. The new level $C^2=s_2^4$ is different but $\mu=C^2s_1=s_2^4s_1$is the same.

Then the $\bs E(z)$ current corresponds to the following
$qq$-character with 6 summands:
\begin{align*}
\bar\chi=\bs2_{0,0} \bs 3^{-1}_{-1,2}\bs 3^{-1}_{1,2}+\bs 2_{0,4}^{-1}\bs 1_{0,2}\bs 0_{0,2}+\bs 1_{0,6}^{-1}\bs 0_{0,2}+\bs 1_{0,2} \bs 0_{0,6}^{-1}+\bs 2_{0,4} \bs 1_{0,6}^{-1} \bs 0_{0,6}^{-1}+\bs 2_{0,8}^{-1} \bs 3_{1,6} \bs 3_{-1,6}.
\end{align*}
Note also that $\ssD_3=\ssA_3$. Here we consider the deformed $\cW$ current corresponding to the second fundamental module.

The first and the fifth monomials are $2$-dominant, the second and the sixth are $2$-antidominant. Moreover the ratio of the first to the second is $A_2(s_2z)$
and the fifth to the sixth is $A_2(s_2^3z)$. Similarly, in direction 1 the ratio of the second to the third is $A_1(s_2)$ and the fourth to the fifth is also $A_1(s_2^2z)$. Finally, in the direction $0$, the ratio of the second to the fourth equals to the ratio of the third to the fifth and equal to $A_0(s_2^2z)$.
It means that $\bar \chi$ is ``closed" in the directions $0,1$, and $2$ (and that $\bar\chi$ commutes with screenings $S_0,S_1,S_2$). 

In the direction of $3$, we do not have such a property. However, one can add a monomial and consider
\begin{align}\label{bar chi}
\tilde\chi=\bar\chi+\bs 3_{1,6} \bs 3_{1,10}^{-1}.
\end{align}
With respect to the $\ssD_3$ algebra it correspond to adding a trivial 1 dimensional representation (in particular, the new current still commutes with screenings $S_0,S_1,S_2$).

On the other hand, the last monomial $\bs 2_{0,8}^{-1} \bs 3_{1,5} \bs 3_{1,7}$,  the new monomial $\bs 3_{1,6} \bs 3_{3,10}^{-1}$ and the shifted monomial $\tau_{\bar\mu}(\bs 2_{0,0} \bs 3^{-1}_{-1,2}\bs 3^{-1}_{1,2})$ together correspond to the 3-dimensional evaluation $U_q\widehat{\mathfrak{sl}}_2$ module in the third direction (in particular, the new current also commutes with the screening $S_3$, see Theorem \ref{B screening thm}).

\medskip

Finally we note that the two currents corresponding to $qq$-characters $\chi$ and $\bar\chi$ are closely related. Namely each monomial of $\chi$ is a shift of a monomial in $\tilde\chi$. Namely, if we apply $\tau_\mu$ to the last three monomials of $\chi$ we obtain $\tau_{s_2}(\tilde\chi)$.
In particular, the integrals (constant terms) of $\tilde\chi$ and $\chi$ coincide. \qed
\end{example}

\begin{remark}\label{rem B}
The phenomena described in the example is rather general. Namely, if one can choose the affine node in several different ways then the corresponding deformed $\cW$ currents are all obtained from each other by shifting some of the monomials. In particular, the integrals (constant terms) of all these currents coincide.

In this paper we study commutative families of operators which include the integral of a deformed $\cW$ current.  
Thus we expect that the families for different choices of the affine node all commute.

While it seems to be difficult to check directly, this fact would follow if the integral of the deformed $\cW$ current had simple spectrum for generic evaluation parameters $u_i$. We expect this is the case and we intend to return to this question when we study the spectrum of IM.

The family of commuting operators correspond to the choices of the affine node of type other than $B$.
Such a choice exists in all cases except when both zeroth and $\ell$-th nodes are of $\ssB$ type (e.g. in type $\ssD^{(2}_{\ell+1}$). \qed
\end{remark}

\subsection{Screenings}\label{sec B screenings}

The screening operators are defined in the same way as for type $\ssA$, see  \eqref{Spm}, \eqref{Sf}.

We call the matrix $\hat B$ stable if $B_{0,\ell}=B_{\ell,0}=0$. 
Most of the deformed affine Cartan matrices of types $\ssB, \ssC,\ssD$ are stable except for a few low rank cases, see Appendix \ref{library}.

In what follows we often assume $\hat B$ is stable to simplify the considerations. We expect that in all such cases this assumption can be dropped.

Let $\hat B$ be stable. 
For $i=1,\cdots,\ell-1$, the screenings are defined by \eqref{Spm}, \eqref{Sf}. 
The screening currents for the $\ell$-th node are given as follows.

For type $\ssB$, 
\begin{align*}
c_{\ell}=c_{\ell+1}:&\quad
A_\ell(z)=s^2_{c_\ell}s_c:\frac{S_{\ell}^+(s_b^{-1}z)}{S_{\ell}^+(s_b z)}:
& \Big((c_\ell,b,c)=cycl(1,2,3)\Big),
\\
c_{\ell}\neq c_{\ell+1}:&\quad
A_\ell(z)=s^2_{c_\ell}s_{c_{\ell+1}}:\frac{S^+_{\ell}(s_b^{-1}z)}{S^+_{\ell}(s_b z)}:\,\, =s^2_{c_\ell}s_{c_{\ell+1}}:\frac{S^-_{\ell}(s_{c_{\ell+1}}^{-1/2}z)}{S^-_{\ell}(s^{1/2}_{c_{\ell+1}} z)}:
& \Big(\{c_{\ell},c_{\ell+1},b\}=\{1,2,3\}\Big).
\end{align*}

For type $\ssC$,
\begin{align*}
c_\ell\neq c_{\ell+1}:&\quad
A_\ell(z)=q_{c_\ell}q^{-1}_{c_{\ell+1}}:\frac{S_{\ell}^+(s_b^{-1}z)}{S_{\ell}^+(s_b z)}:
\,\,=q_{c_\ell}q^{-1}_{c_{\ell+1}}:\frac{S_{\ell}^-(s_{c_\ell+1}^{-2}z)}{S_{\ell}^-(s^2_{c_{\ell+1}} z)}:
\quad \Big(\{c_{\ell},c_{\ell+1},b\}=\{1,2,3\}\Big).
\end{align*}

For type $\ssD$, 
\begin{align*}
c_{\ell-1}=c_{\ell}=c_{\ell+1}:&\quad
A_\ell(z)=q_{c_\ell}:\frac{S_{\ell}^+(s_b^{-1}z)}{S_{\ell}^+(s_b z)}:\,\,=q_{c_\ell}:\frac{S_{\ell}^-(s_c^{-1}z)}{S_{\ell}^-(s_c z)}: &\Big( (c_\ell,b,c)=cycl(1,2,3)  \Big),\\
c_{\ell-1}\neq c_{\ell}=c_{\ell+1}:&\quad
A_\ell(z)=s^{-1}_{b}:\frac{S_{\ell}^f(s_b^{-1}z)}{S_{\ell}^f(s_b z)}: &\Big(\{c_{\ell-1},c_{\ell},b\}=\{1,2,3\}\Big).
\end{align*}

The zeroth screening currents $S_0^\pm(z)$ are defined in the same way (changing indices $\ell$, $\ell-1$, $\ell-2$ to $0,1,2$ respectively). 

Then we have the following theorem.

\begin{thm}\label{CD screening thm} Assume $\hat B$ is stable. The screening operators $S_i$ with $i\neq0$ commute with both $E(z)$ and $\Eb(z)$:
\begin{align*}
[S_i, E(z)]=[S_i, \bs E(z)]=0, \qquad i=1,\ldots,\ell\,.
\end{align*}
For the zeroth screening we have 
\begin{align*}
&[S_0,\bs E(z)]=[S_0,b_{c_1}(\bs \La_1(z)-\bs \La_1(\mu z))]\qquad &&\text{for type $\ssC$}\,,\\
&[S_0,\bs E(z)]=[S_0,b_{c_1}(\bs \La_1(z)-\bs \La_1(\mu z))+b_{c_2}(\bs \La_2(z)-\bs \La_2(\mu z))]
\qquad &&\text{for type $\ssD$}\,.
\end{align*}
\end{thm}
\begin{proof} 
For $i=1,\cdots,\ell-1$ this is a type $\ssA$ computation, so it reduces to Lemma \ref{screening lemma}.

For type $\ssC$, one can check that the triples $(A_\ell(z),\Lambda_{\ell}(z),\Lambda_{\bar\ell}(z))$ and
$(A_0(z),\bs \Lambda_{\bar 1}(z),\bs \Lambda_{1}(\mu z))$
satisfy the conditions of Lemma \ref{screening lemma}.

Likewise, for type $\ssD$, one checks that the triples
$(A_\ell(z), \Lambda_{\ell-1}(z),\Lambda_{\bar\ell}(z))$,
$(A_\ell(z), \Lambda_{\ell}(z),\Lambda_{\overline{\ell-1}}(z))$,
$(A_0(z), \bs\Lambda_{\bar 2}(z),\bs\Lambda_{1}(\mu z))$, 
$(A_0(z), \bs\Lambda_{\bar 1}(z),\bs\Lambda_{2}(\mu z))$
satisfy the conditions of Lemma \ref{screening lemma}.

The remaining case is type $\ssB$ with $i=\ell$. 
The relevant contractions are 
\begin{align*}
&\mathcal{C}(A_\ell(z),\Lambda_\ell(w))=-\frac{t_1t_2t_3}{t_{c_{\ell}}}s^{-1}_{c_\ell}s^{-1/2}_{c_{\ell+1}}\,,
\quad
\mathcal{C}(\Lambda_\ell(z),A_\ell(w))=-\frac{t_1t_2t_3}{t_{c_{\ell}}}s_{c_\ell}s^{1/2}_{c_{\ell+1}}\,,
\\
&\mathcal{C}(A_\ell(z),\Lambda_0(w))=-\mathcal{C}(\Lambda_0(z),A_\ell(w))
=\frac{t_1t_2t_3}{s^{1/2}_{c_{\ell+1}}+s^{-1/2}_{c_{\ell+1}}}\,,\\
&\mathcal{C}(A_\ell(z),\Lambda_{\bar \ell}(w))=\frac{t_1t_2t_3}{t_{c_{\ell}}}s_{c_\ell}s^{1/2}_{c_{\ell+1}}\,,
\quad
\mathcal{C}(\Lambda_{\bar \ell}(z),A_\ell(w))=\frac{t_1t_2t_3}{t_{c_{\ell}}}s^{-1}_{c_\ell}s^{-1/2}_{c_{\ell+1}}\,.
\end{align*}
Suppose $c_\ell=c_{\ell+1}$. We take $c_{\ell}=c_{\ell+1}=2$
for concreteness. 
From the above contractions we compute the singular parts of the operator products, 
\begin{align*}
&S^+_\ell(z)\Lambda_\ell(w)=-\frac{t_1 s_2^{-3/2}w}{z-s_3s_2^{-1/2}w}:S^+_\ell(s_3s_2^{-1/2}w)\Lambda_\ell(w):+O(1),\\
&S^+_\ell(z)\Lambda_\ell(w)=k
\Bigl(\frac{s_3^{-1} s_2^{-1/2}w}{z-s_3^{-1}s_2^{-1/2}w}:S^+_\ell(s_3^{-1} s_2^{-1/2}w)\Lambda_0(w):
-\frac{s_3s_2^{1/2}w}{z-s_3s_2^{1/2}w}
:S^+_\ell(s_3s_2^{1/2}w)\Lambda_0(w):\Bigr)+O(1),\\
&S^+_\ell(z)\Lambda_{\bar \ell}(w)=\frac{t_1 s_2^{3/2}w}{z-s_3^{-1}s_2^{1/2}w}:S^+_\ell(s_3^{-1}s_2^{1/2}w)\Lambda_{\bar \ell}(w):+O(1),
\end{align*}
where $k=t_1 (s^{1/2}_2-s^{1/2}_2)/(s_1s_2^{1/2}-s^{-1}_1s_2^{-1/2})=t_1 b_2/k^{\ssB}_2$.
In view of the relations
\begin{align*}
&:S^+_\ell(s_3z)\Lambda_{\ell}(s_2^{1/2}z):\ =s_3^{-1}s_2\ :S^+_\ell(s_3^{-1}z)\Lambda_0(s_2^{1/2}z):\,,\\
&:S^+_\ell(s_3^{-1}z)\Lambda_{\bar \ell}(s_2^{-1/2}z):\ =s_3s_2^{-1}\ :S^+_\ell(s_3z)\Lambda_0(s_2^{-1/2}z):\,,
\end{align*}
we conclude that
\begin{align*}
[S^+_\ell, b_2\Lambda_{\ell}(w)+k^\ssB_2\Lambda_0(w)+b_2\Lambda_{\bar\ell}(w)]=0\,.
\end{align*}

Note that we can define another screening current $S_\ell^-(z)$ by interchanging $q_1\leftrightarrow q_3$ in $S_\ell^+(z)$.
Note also that $k^\ssB_2=(1+s_2)(s_3-s_1)/\kappa_1$ changes sign under the swap $q_1\leftrightarrow q_3$. So, the screening operator $S_\ell^-$ commutes with a different current (which is also a representation of
$\cK$, see \eqref{iota})
\begin{align*}
[S^-_\ell, b_2\Lambda_{\ell}(w)-k^\ssB_2\Lambda_0(w)+b_2\Lambda_{\bar\ell}(w)]=0\,.
\end{align*}
That the present case admits only one screening has been observed in \cite{FR1}, see the end of 
Section 7 thereof.

The calculation for $c_\ell\neq c_{\ell+1}$ is entirely similar, cf. \cite{FR1}, Theorem 3.
Alternatively one can write the relevant three terms as a ``fusion'' 
and reduce the calculation to  two-dimensional representation of $\mathfrak{sl}_2$, 
see the last line of the proof of Proposition \ref{Kmat} below.
\end{proof}

For type $\ssB$, the current $\Eb(z)$ does not ``close up'' under $S_0$. However, 
definition \eqref{B-0} suggests that we introduce an operator
\begin{align*}
\bs\Lambda_{\bar 0}(z)= \,\,:\bs \Lambda_{\bar 1}(z)A^{-1}_0(\mu^{1/2}s^{-1/2}_{c_0}z):\,\,
= \,\,:\bs\Lambda_{1}(\mu z)A_0(\mu^{1/2}s^{1/2}_{c_0}z):\
\end{align*}
and consider an extended current
\begin{align}\label{extended}
\tilde\Eb(z)=\Eb(z)+k^{\ssB}_{c_0}\bs\Lambda_{\bar 0}(z)\,.
\end{align}
On can think that this current corresponds to a direct sum of the vector representation and the trivial representation, see \eqref{bar chi}. 

\begin{thm}\label{B screening thm}  Assume $\hat B$ is stable.
The screening operators $S_i$ with $i\ne 0$ commute with $\tilde E(z)$:
\begin{align*}
&[S_i,\tilde\Eb(z)]=0,\qquad i=1,\ldots,\ell\,.
\end{align*}
For the zeroth screening we have
\begin{align*}
[S_0,\tilde\Eb(z)]=b_{c_1}[S_0,\bs\Lambda_1(z)-\bs\Lambda_1(\mu z)]\,.
\end{align*}
\end{thm}
\begin{proof}
The proof is parallel to that of Theorem \ref{CD screening thm}.
\end{proof}
\medskip

\subsection{Integrals of motion associated with $\cK$}\label{sec B integrals}
In this subsection we show that algebra $\cK$ possesses a family of integrals of motion
when $C^2=\mu q_{c_0}$, namely if the zeroth node is of type $\ssC$ or $\ssD$.  

First, we prepare a lemma about matrix elements of products of $\Eb(z)$. 
Set
\begin{align*}
&f(x)=\frac{(1-C^2x)(1-C^{-2}x)}{(1-x)^5}
\times \prod_{s=1}^3
\frac{(q_s^{-1}x;\mu)_\infty}{(\mu q_s x;\mu)_\infty}\,.
\end{align*}
Using the defining relations, we find
\begin{align*}
&f(w/z)\Eb(z)\Eb(w)=f(z/w)\Eb(w)\Eb(z)\,.
\end{align*}
With the definition
\begin{align}
\Eb(z_1,\ldots,z_n)=\prod_{i<j}f(z_j/z_i)\times\Eb(z_1)\cdots \Eb(z_n)\,,
\label{E^n}
\end{align} 
the matrix elements of \eqref{E^n} have the form 
$p(z_1,\ldots,z_n)/\prod_{i<j}(z_i-z_j)^4$, where $p$ is a symmetric Laurent polynomial. 
Furthermore, the Serre relations entail the following zero conditions. 

\begin{lem}\label{zero-cond} We have
\begin{align} 
&\Eb(z,q_{c_1} z,q_{c_1}q_{c_2} z)=0\quad (c_1\neq c_2,\ c_1,c_2\in\{1,2,3\}),
\label{zero1}
\\
&\Eb(z, C^{\pm 2}z,C^{\pm 2}q_c^{\pm 1}z)=0\quad (c\in \{1,2,3\}),
\label{zero2}
&\\
&\Eb(C^{-2}z,z,C^{2}z)=0\,.
\label{zero3}
\end{align}
\end{lem}
\begin{proof}
We can write
\begin{align*}
\Eb(z_1,\ldots,z_n)=E(z_1,\ldots,z_n)\prod_{i=1}^n
{K^+_\mu(z_i)}^{-1}\,,  
\end{align*} 
where
\begin{align*}
&E(z_1,\ldots,z_n)=\prod_{i<j}f_0(z_j/z_i)\times E(z_1)\cdots E(z_n)\,,
\\
&f_0(x)=\frac{1-C^2x}{1-x}\frac{1-C^{-2}x}{1-x}\prod_{s=1}^3\frac{1-q_sx}{1-x}\,.
\end{align*}
Since 
${K^+_\mu(z)}^{-1}$ 
on any vector is a Laurent polynomial in $z^{-1}$, 
it is enough to prove the statements for $E(z_1,\ldots,z_n)$. 

Consider the Serre relation \eqref{EEE}.

The left hand side has the form
\begin{align*}
LHS=Sym_{z_1,z_2,z_3}\left[ \Bigl(\frac{z_2}{z_3}-\frac{z_3}{z_2}-\frac{z_1}{z_2}+\frac{z_2}{z_1}\Bigr)E(z_1)E(z_2)E(z_3)\right],
\end{align*}
while the right hand side is
\begin{align*}
RHS=Sym_{z_1,z_2,z_3}\left[
X(z_1,z_2,z_3)\kappa_1^{-1}\delta(C^2 z_1/z_3)K^-(z_1)E(z_2)K^+(z_3) \right].
\end{align*}

Let us consider the left hand side. Each term $E(z_i)E(z_j)E(z_k)$ is defined in the region $|z_i|\gg|z_j|\gg|z_k|$.
It can be rewritten as 
\begin{align*}
E(z_i)E(z_j)E(z_k)=E(z_1,z_2,z_3)\times f_0(z_j/z_i)^{-1}f_0(z_k/z_i)^{-1}f_0(z_k/z_j)^{-1}\,,
\end{align*}
where all matrix elements of $E(z_1,z_2,z_3)$ are symmetric Laurent polynomials multiplied by $\prod_{i<j}(z_i-z_j)^{-4}$.
To find the zero conditions of   $E(z_1,z_2,z_3)$, 
we bring all terms in both sides 
into expansions in the common domain $|z_1|\gg|z_2|\gg|z_3|$ with additional delta functions.
For example, the term $E(z_2)E(z_1)E(z_3)$ gives rise to five delta functions $\delta(q_cz_2/z_1)$ ($c=1,2,3$)
and $\delta(C^{\pm2}z_2/z_1)$. 

\medskip

We observe that the terms without delta functions cancel out, due to the identity of rational functions
\begin{align*}
\mathop{\mathrm{Skew}}_{z_1,z_2,z_3}
\Bigl(\frac{z_2}{z_3}-\frac{z_3}{z_2}-\frac{z_1}{z_2}+\frac{z_2}{z_1}\Bigr)
\frac{1} 
{g_{12}g_{13}g_{23}}
=0\,,
\end{align*}
where $g_{ij}=g(z_i,z_j)$.

\medskip

The terms with one delta function also cancel out. 

For example, the coefficient of $\delta(q_1z_1/z_2)$ comes from 3 terms and cancel out:
\begin{align*}
\Bigl[\bigl(\frac{z_1}{z_3}-\frac{z_3}{z_1}-\frac{z_2}{z_1}+\frac{z_1}{z_2}\bigr)\frac{1}{g_{13}g_{23}}
+\bigl(\frac{z_3}{z_1}-\frac{z_1}{z_3}-\frac{z_2}{z_3}+\frac{z_3}{z_2}\bigr)\frac{1}{g_{23}g_{31}}
+\bigl(\frac{z_2}{z_1}-\frac{z_1}{z_2}-\frac{z_3}{z_2}+\frac{z_2}{z_3}\bigr)\frac{1}{g_{31}g_{32}}
\Bigr]\Bigr|_{z_2=q_1z_1}\hspace{-3pt}=0.
\end{align*}
\medskip

Similarly, the coefficient of $\delta(C^2 z_1/z_3)$ 
comes from 3 terms. 
In the left hand side, we collect all contributions to 
$\kappa_1^{-1}\delta(C^2 z_1/z_3)K^-(z_1)E(z_2)K^+(z_3)$ and find
\begin{align*}
&-\Bigl(\frac{z_3}{z_2}-\frac{z_2}{z_3}-\frac{z_1}{z_3}+\frac{z_3}{z_1}\Bigr)\frac{g_{13}g_{12}}{g_{31}g_{21}}
-\Bigl(\frac{z_1}{z_2}-\frac{z_2}{z_1}-\frac{z_3}{z_1}+\frac{z_1}{z_3}\Bigr)\frac{g_{13}g_{23}}{g_{31}g_{32}}
+\Bigl(\frac{z_2}{z_1}-\frac{z_1}{z_2}-\frac{z_3}{z_2}+\frac{z_2}{z_3}\Bigr)\frac{g_{12}g_{13}g_{23}}{g_{21}g_{31}g_{32}}
\\
&=\kappa_1\bigl(\frac{z_2^2}{z_1z_3}-1\bigr)\frac{z_2(z_2+z_1)(z_2+z_3)}{g_{21}g_{32}}g_{13}\,,
\end{align*}
which coincides with the corresponding term in the right hand side.

\medskip

Finally, we consider terms with two delta functions. 

Consider the term {$\delta(q_{c_1} z_1/z_2)\delta(q_{c_2} z_2/z_3)$ ($c_1\neq c_2$)}. It comes only from 
the last term on the left hand side with non-zero coefficient,
\begin{align*}
\Bigl(\frac{z_2}{z_1}-\frac{z_1}{z_2}-\frac{z_3}{z_2}+\frac{z_2}{z_3}\Bigr)\frac{1}{g_{31}}
\Bigr|_{z_2=q_{c_1}z_1,z_3=q_{c_2}z_2}\neq 0,
\end{align*}
which implies the zero condition \eqref{zero1}.

Consider $\delta(q_i z_1/z_2)\delta(C^2 z_2/z_3)$. 
Again, only the last term from LHS contributes, with non-zero coefficient, which yields
$E(z,q_iz,C^2q_iz)=0\,.$

Likewise, only the last term contributes to the coefficient of $\delta(C^2 z_1/z_2)\delta(q_i z_2/z_3)$, giving the relation
$E(z,q^{-1}_iz,C^{-2}q_i^{-1}z)=0\,.$

Consider $\delta(C^2 z_1/z_2)\delta(C^2 z_2/z_3)$. 
Again only the last term contributes and we find the equality $E(C^{-2}z,z,C^{2}z)=0$. 
\end{proof}

\bigskip
\begin{rem}
Since the zero conditions are important, let us check them  directly  on representations. 
For simplicity we consider the case $c_1=\cdots=c_{\ell+1}=2$. 
The contractions between $\Lambda_i(z)$ are given as follows, 
see \eqref{LaContractionBCD}: 
\begin{align*}
&\cont{\La_i(z)}{\La_i(w)}
=
\begin{cases}
\displaystyle{\frac{1-w/z}{1-q_1w/z}\frac{1-q_2^{-1}w/z}{1-q_3w/z}}&(i=1,\ldots,\ell,\bar{\ell},\ldots,\bar{1}),\\
\displaystyle{\frac{1-w/z}{1-q_2w/z}\frac{1-q_3^{-1}w/z}{1-q_1w/z}\frac{1-q_3^{1/2}q_2w/z}{1-q_3^{1/2}w/z}
\frac{1-q_3^{-1/2}q_2^{-1}w/z}{1-q_3^{-1/2}w/z}}
&(i=0),\\
\end{cases}
\\
&\cont{\La_i(z)}{\La_j(w)}
=
\begin{cases}
\displaystyle{\frac{\bar{g}(z,w)}{g(z,w)}}& (i\prec j,\ i\neq \bar{j}), \\
1 & (i\succ j,\ i\neq \bar{j}), \\
\end{cases}
\\
&\cont{\La_i(z)}{\La_{\bar i}(w)}
=\frac{\bar{g}(z,w)}{g(z,w)}
\frac{1-q_1^{-1}q_2^{-i}C^2w/z}{1-q_2^{-i}C^2w/z}
\frac{1-q_3^{-1}q_2^{-i}C^2w/z}{1-q_2^{-i+1}C^2w/z}
\quad (i=1,\ldots,\ell),
\\
&\cont{\La_{\bar i}(z)}{\La_i(w)}
=
\frac{1-q_1q_2^{i}C^{-2}w/z}{1-q_2^{i-1}C^{-2}w/z}
\frac{1-q_3q_2^{i}C^{-2}w/z}{1-q_2^{i}C^{-2}w/z}
\quad (i=1,\ldots,\ell)\,.
\end{align*}
As an example, consider \eqref{zero2}. 
From the above we observe that 
\begin{align*}
&(1-C^{-2}z_2/z_1)\cont{\Lambda_i(z_1)}{\Lambda_j(z_2)}\Bigl|_{z_2=C^2z_1}= 0 \quad ((i,j)\neq (\bar{1},1)),
\\
&\cont{\Lambda_1(z_2)}{\Lambda_k(z_3)}\Bigl|_{z_3=q_cz_2}=0\quad (k\neq 1,\ c=1,2,3)\,,\quad
\cont{\Lambda_1(z_2)}{\Lambda_1(z_3)}\Bigl|_{z_3=q_2z_2}=0\,,
\\
&\cont{\Lambda_{\bar 1}(z_1)}{\Lambda_1(z_3)}\Bigl|_{z_3=C^2q_c z_1}=0\quad (c=1,3).
\end{align*}
Therefore the operator 
$\prod_{r<s}f_0(z_s/z_r)\cdot \Lambda_i(z_1)\Lambda_j(z_2)\Lambda_k(z_3)$
vanishes at $(z_1,z_2,z_3)=(z,C^2z,C^2q_cz)$ for all $i,j,k$, and \eqref{zero2} follows. 
 The other zero conditions can be verified similarly.  
\end{rem}

\bigskip

First, we consider the case  $C^2=\mu q_2$. 

\medskip

\begin{thm}\label{mu q_2} Assume that  $C^2=\mu q_2$. 
Then the following elements $\{\bI_n\}_{n=1}^\infty$ 
are mutually commutative,
\begin{align}
\bI_n=
\int\!\!\cdots \!\!\int
\Eb(z_1)\cdots\Eb(z_n)\cdot
\prod_{j<k}\kfun_2(z_k/z_j)
\,\prod_{j=1}^n\frac{dz_j}{2\pi i z_j}\,.
\label{bIn}
\end{align}
Here the integral is taken on the unit circle $|z_j|=1$, $j=1,\ldots,n$ in the region $|q_1|,|q_3|>1$ and extended by analytic continuation everywhere else.
\end{thm}
Theorem \ref{mu q_2} is proved in Appendix \ref{proof app}.
\medskip

In the case $C^2=\mu q_3$ or $C^2=\mu q_1$ we use a different kernel function. For example we have the following theorem.

\begin{thm}\label{mu q_3} Assume that  $C^2=\mu q_3$, and define  
\begin{align}\label{bIn3}
&\bI'_n=
\int\!\!\cdots \!\!\int
\Eb(z_1)\cdots\Eb(z_n)\cdot\prod_{j<k}\kfun_3(z_k/z_j)
\,\prod_{j=1}^n\frac{dz_j}{2\pi i z_j}\,.
\end{align}
We choose the contour for $z_j$  for each $j=1,\ldots,n$ in  such a way that 
\begin{align*}
&\text{$\mu^s q_i^{-1}z_k$ ($s\ge0, i=1,2$, $k<j$) are inside},\\
&\text{$\mu^{-s} q_i z_k$ ($s\ge0, i=1,2$, $k<j$) are outside}.
\end{align*}
Then the elements $\{\bI'_n\}_{n=1}^\infty$ are mutually commutative.
\end{thm}
\begin{proof}
The integrand of \eqref{bIn3} are obtained from the one of \eqref{bIn}
by interchanging the roles of $q_2$ and $q_3$. 
The result of the previous subsection tells that, 
if the parameters satisfy $|q_1|,|q_2|>1$, 
then the integrals \eqref{bIn3} over the unit circle give a commutative family.
Being an analytic continuation in the parameter $q_2$ from $|q_2|>1$ to $|q_2|<|q_1|^{-1}$,
the integrals \eqref{bIn3} remain commutative. 
\end{proof}

Thus in a generic admissible $\cK$ module, we have three commutative families of operators. We call them deformed integrals of motion due to the following theorem.

\begin{thm}\label{BCD comm thm} Assume $\hat B$ is stable. 
For $C^2=\mu q_2$, the elements $\bI_n$ commute with all screening operators, 
\begin{align*}
[S_i,\bI_n]=0, \qquad i=0,1,\ldots,\ell,\ n\ge 1.
\end{align*}

\end{thm}
\begin{proof}
The argument for type $\ssC$ is similar to that in type $\ssA$, see Theorem \ref{A comm thm}. 

For type $\ssD$ the formulas become a little more cumbersome. We illustrate it in the case $c_0=c_1=c_2=2$.
We start from 
\begin{align*}
&[S_0,\Eb(z)]=const.(\cA^{(1)}(z)-\cA^{(1)}(\mu z))\,,
\end{align*}
where now $\cA^{(1)}(z)$ is a sum of two terms
\begin{align*}
&\cA^{(1)}(z)=z :S_0(s_3^{-1}\mu^{-1/2}z)\bigl(\bs\Lambda_1(z)+\bs\Lambda_2(z)\bigr):\,.
\end{align*}
Taking residues repeatedly we obtain currents for $m=1,2,\ldots$
\begin{align*}
&\cA^{(m)}(z)=z:S_0(s_3^{-1}q_1^{-m+1}\mu^{-1/2}z)
\sum_{k=0}^m a^{(m)}_k \bs \Lambda_2(q_3^{-m+1}z)\cdots \bs\Lambda_2(q_3^{-m+k}z)
 \bs \Lambda_1(q_3^{-m+k+1}z)\cdots \bs\Lambda_1(z):\,,\\
&a^{(m)}_k=\prod_{j=1}^k\frac{1-q_3^{m-j+1}}{1-q_3^j}\frac{1-q_1q_3^j}{1-q_1q_3^{m-j+1}}\,.
\end{align*}
The rest of the argument is the same. 
\end{proof}

\medskip

It would be interesting to study the spectrum of deformed integrals of motion.

 \medskip
 
Theorem \ref{BCD comm thm} deals with affine nodes of types $\ssC$ and $\ssD$. If the affine node is of type $\ssB$, the integral of current $\tilde{\bs E}(z)$, see \eqref{Cartan example}, commutes with all the screenings. In the case the $\ell$-th node is not of type $\ssB$, this integral coincides with the integral of another current for which the affine node is of type $\ssC$ or $\ssD$ (the same as the $\ell$-th node for the original current, see \eqref{bar chi}). Therefore, we can include the integral of $\tilde{\bs E}(z)$ into a family of integral of motions corresponding to that current. We expect this family commutes with all screening operators. 

\section{Additional remarks}\label{section suppl}
\subsection{Integrals of motion of KZ type}\label{sec KZ}

In this subsection we continue with the Cartan matrices in Table \ref{CartanTable}
except for types $\ssA^{(1)}_\ell,\ssD^{(2)}_{\ell+1}$, and construct another set of commuting operators which commute with integrals $\bI_n$. 
We hope that these integrals may be more convenient for Bethe ansatz study in the future.

The results in this section depend on the following technical statement which we do not discuss.
\begin{conj}
The $\cK$ module $\mathbb{F}=\F_2(u_1)\otimes\cdots\otimes\F_2(u_\ell)\otimes \F^X_c$ is irreducible for generic parameters $u_1,\dots,u_\ell$. \qed
\end{conj}

The vector space $\mathbb{F}$ 
is an irreducible representation of the set of $\ell+1$ bosons $\{\ssa_{i,r}\mid i=0,\ldots,\ell,\ r\neq 0\}$.
We regard  $\{\ssa_{i,0}\mid i=0,\ldots,\ell\}$ as functions of the parameters 
 $u_1,\ldots,u_\ell$ 
through \eqref{A-i}, \eqref{B-ell}--\eqref{D-ell} and \eqref{B-0}--\eqref{D-0}. 
Separating the zero modes we shall write $A_i(z)=e^{\ssa_{i,0}}A^{osc}_i(z)$, $Y_i(z)=e^{\ssy_{i,0}}Y_i^{osc}(z)$. 
To each $i=0,\ldots,\ell$ we associate a reflection operator $R_i\in\End\mathbb{F}$ defined as follows.

\begin{prop}
There exist operators  $R_i\in\End\mathbb{F}$ with the properties
\begin{align*}
&R_i Y_j(z)=Y_j(z)R_i \quad (j\neq i)\,,\\
&R_i\Bigl(e^{\ssy_{i,0}}Y^{osc}_i(z)+e^{\ssy_{i,0}-\ssa_{i,0}}:\frac{Y^{osc}_i(z)}{A_i^{osc}(\hat s_i z)}:\Bigr)
=
\Bigl(e^{\ssy_{i,0}-\ssa_{i,0}}Y^{osc}_i(z)+e^{\ssy_{i,0}}:\frac{Y^{osc}_i(z)}{A_i^{osc}(\hat s_i z)}:\Bigr)R_i\,,
\end{align*}
where $\hat s_i=s_2$ for $i\neq 0,\ell$, $\hat s_i=s_2s_3^{1/2},s_2s_3^{-1},s_2$ 
if $i\in\{0,\ell\}$ is of type $\ssB$, $\ssC$, $\ssD$, respectively.
\end{prop}
In other words, $R_i$ is an operator depending only on one boson $\{\ssa_{i,r}\}_{r\neq0}$ as well as the zero modes 
$\{\ssa_{j,0}\}_{j=1,\ldots,\ell}$, and implements the Weyl reflection on the latter. 
\begin{proof}
It is clear that for $i=1,\ldots,\ell-1$ the $R$ matrix $R_i=\Rc_{i,i+1}(u_i/u_{i+1})$,  obtained from the universal $R$ matrix of $\E$, has the required properties. 
For $i=\ell$ and type $\ssB$, $\ssC$, we use Proposition \ref{Kmat} to get 
$R_\ell=\ssK_\ell(u_\ell)$.
In the case of type $\ssD$, from the remark after Proposition \ref{Kmat} we have
$\ssK_\ell A_{\ell-1}(z) \ssK_\ell=A_\ell(z)$, $\ssK_\ell^2=1$.
Hence we can take
$R_\ell=\ssK_\ell \Rc_{\ell-1,\ell}(u_{\ell-1}u_{\ell}) \ssK_\ell$.

The case $i=0$ is similar.
\end{proof}

For $i\neq 0$, they are intertwiners of $\cK$ modules. Writing the $\Eb(z)$ current as 
$\Eb(z;u_1,\ldots,u_\ell)$ we have
\begin{align}
&\Rc_{i,i+1}(u_i/u_{i+1})\Eb(z;u_1,\ldots,u_i,u_{i+1},\ldots,u_\ell)
\label{RiE}\\
&=
\Eb(z;u_1,\ldots,u_{i+1}, u_i,\ldots,u_\ell)\Rc_{i,i+1}(u_i/u_{i+1})
\quad (i=1,\ldots,\ell-1)\,,
\notag\\
&\ssK_{\ell}(u_\ell)\Eb(z;u_1,\ldots,u_{\ell-1},u_\ell)=
\Eb(z;u_1,\ldots,u_{\ell-1},u^{-1}_\ell)\ssK_{\ell}(u_\ell)\quad (i=\ell).
\label{RlE}
\end{align} 
For $i=0$ the corresponding relation takes the form 
\begin{align}
&\ssK_1(u_1)\Eb^*(z;u_1,u_{2},\ldots,u_\ell)=\Eb^*(z;u^{-1}_1,u_{2},\ldots,u_\ell)\ssK_1(u_1)\,,
\label{R0E}
\end{align} 
where $\Eb^*(z;u_1,\ldots,u_\ell)$ means 
$\Eb(z;u_1,\ldots,u_\ell)-b_2\bigl(\bs \La_{1}(z)-  \bs \La_{1}(\mu z)\bigr)$
or 
$\Eb(z;u_1,\ldots,u_\ell)-b_2\bigl(\bs \La_{1}(z)-  \bs \La_{1}(\mu z)\bigr)
-b_2\bigl(\bs \La_2(z)-  \bs \La_2(\mu z)\bigr)$, 
depending on whether the zeroth node is of type $\ssC$ or of type $\ssD$.

We now introduce operators $\ssT_i$, $i=1\ldots,\ell$, by
\begin{align*}
&\ssT_i=\ssT_i^+\ssT_i^-\,,\\
&\ssT_i^+=\Rc_{i-1,i}(u_i/u_{i-1})\cdots\Rc_{1,2}(u_i/u_1)
\ssK_1(u_i)\Rc_{1,2}(u_1u_i)\cdots\Rc_{i-1,i}(u_{i-^1}u_i)\,,\\
&\ssT_i^-=\Rc_{i,i+1}(u_iu_{i+1})\cdots\Rc_{\ell-1,\ell}(u_iu_{\ell})
\ssK_\ell(u_i)\Rc_{\ell-1,\ell}(u_i/u_\ell)\cdots\Rc_{i,i+1}(u_i/u_{i+1})\,.
\end{align*}
We call $\ssT_i$ integrals of motion of KZ type for the following reason.

\begin{thm}
The operators $\ssT_i$ and $\bI_n$ in Theorems \ref{mu q_2}, \ref{mu q_3} are mutually commutative:
\begin{align*}
&[\ssT_i,\ssT_j]=0\,\qquad (i,j=1,\ldots,\ell),\\
&[\ssT_i, \mathbf{I}_n]=0\,\qquad (i=1,\ldots,\ell,\ n\ge 1).
\end{align*}
\end{thm}
\begin{proof}
The commutativity of $\ssT_i$'s is a simple consequence of the (ordinary and boundary)
Yang-Baxter equations. 

To see the second statement, we use \eqref{RiE}, \eqref{RlE} and \eqref{R0E}.
The only issue is 
to check that one can safely shift $\bs\Lambda_1(z_j)$ to $\bs\Lambda_1(\mu z_j)$ without encountering poles in between. 
This is straightforward. 
\end{proof}

\subsection{Exceptional types}\label{sec G2}
The $\cW$ algebras and integrals of motion of type $\ssA$ are obtained from the quantum toroidal algebra $\E$.
In this paper we have introduced an algebra $\cK$ which allows us to treat 
deformed $\cW$ algebras of non-exceptional types uniformly. A natural question is what happens in exceptional types.

For an exceptional type, one can consider 
a similar algebra by taking the current $T(z)$ in the sense of \cite{FR1}, together with a vertex operator $Z(z)$ in one extra boson such that $Z(z)T(w)=T(w)Z(z)$ and such that
all terms in $E(z)=T(z)Z(z)$ have rational contractions. 
This gives the quantum algebra in this type similar to $\E$
and $\cK$.

For example, in the case of the seven dimensional representation of ${\mathsf G}_2$, 
one obtains the relation with four $\delta$-functions  
\begin{align*}
&g(z,w)E(z)E(w)+g(w,z)E(w)E(z)\\
&=c_1\Bigl(\delta\bigl(q_1^3q_2^6 \frac{w}{z}\bigr)w^3K(w)+
\delta\bigl(q_1^3q_2^6 \frac{z}{w}\bigr)z^3K(z)\Bigr)\\
&+c_2\Bigl(\delta\bigl(q_1^2q_2^4  \frac{w}{z}\bigr)w^3 :E(q_1q_2^2w)\tilde{K}(w):   
+\delta\bigl(q_1^2q_2^4  \frac{z}{w}\bigr)z^3:E(q_1q_2^2z)\tilde{K}(z):\Bigr),
\end{align*}
where $K(z)=\,\,:Z(z)Z(q_1^3q_2^6 z):$, $\tilde{K}(z)=\,\,:Z(z)Z(q_1^2q_2^4 z)Z^{-1}(q_1q_2^2z):$, and  $c_1,c_2$ are constants and
\begin{align*}
&\mathcal{C}(Z(z),Z(w))=(1-q_1)(1-q_3)\frac{1+q_1q_2^2}{1+q_1^3q_2^6}(q_1q_2^2-1+q_2)\,,
\\
&\mathcal{C}(K(z),E(w))=(1-q_1)(1-q_3)(1+q_1q_2^2)(q_1q_2^2-1+q_2)\,,
\\
&\mathcal{C}(\tilde K(z),E(w))=(1-q_1)(1-q_3)(q_1q_2^2-1+q_2)\,.
\end{align*}
The role of this algebra is not clear since $\ssG_2$ is not a part of any family. In particular, we do not expect any comodule or coalgebra structure.

\appendix

\section{Proof of Theorem \ref{mu q_2}}\label{proof app}
In this Section we prove Theorem \ref{proof app}.
We have  $C^2=\mu q_2$.

\subsection{Commutativity $[\bI_1,\bI_2]=0$.} 
As an illustration, let us verify the commutativity of $\bI_1$ and $\bI_2$. 

We note that, by making use of the decomposition 
\begin{align*}
&\kfun_2(x)=q_2C^2f(x)\sigma_2(x)\sigma_2(x^{-1})\,,\quad
\sigma_2(x)=(1-x)^3 \frac{(\mu x,\mu^2 q_2x)_\infty}{(q_1^{-1}x,q_3^{-1}x)_\infty}\,,
\end{align*}
the integral \eqref{bIn} can be rewritten in terms of the currents \eqref{E^n} as  
\begin{align*}
const.\ \bI_n&= 
\int\!\!\cdots \!\!\int
\Eb(z_1,\ldots, z_n)  \cdot \prod_{j\neq k}\sigma_2(z_k/z_j)
\prod_{j=1}^n\frac{dz_j}{z_j}\,.
\end{align*}

Consider the products
\begin{align*}
&\bI_1\bI_2=\int\!\!\int \!\!\int 
\Eb(z_1,z_2,z_3)\times f(z_2/z_1)^{-1}f(z_3/z_1)^{-1}\sigma_2(z_2/z_3)\sigma_2(z_3/z_2)
\prod_{j=1}^3\frac{dz_j}{2\pi i z_j}\,,\\
&\bI_2\bI_1= 
\int\!\!\int \!\!\int\Eb(z_1,z_2,z_3)\times f(z_1/z_2)^{-1}f(z_1/z_3)^{-1}\sigma_2(z_2/z_3)\sigma_2(z_3/z_2)\prod_{j=1}^3
\frac{dz_j}{2\pi i z_j}\,.
\end{align*}
The integral in $\bI_1\bI_2$ is initially defined for $|z_1|\gg|z_2|=|z_3|=1$, 
while in  $\bI_2\bI_1$ it is defined  for $|z_1|\ll |z_2|=|z_3|=1$.
In both cases we move the contour for $z_1$ to the unit circle. Along the way we 
pick up residues at the poles $z_1=q_2^{-1}z_i, \mu^{-1}q_2^{-1}z_i$  
or $z_1=q_2z_i, \mu q_2 z_i$, $i=2,3$, respectively.  

When all variables are on the unit circle, the two integrals coincide thanks to the identity \eqref{id-h}. 

Let us compare the residues at $z_1=q_2^{\pm1}z_3$. 
We obtain respectively
\begin{align*}
J_1=\int\!\!\int_{|z_2|=|z_3|=1}\Eb(z_2,q_2^{-1}z_3,z_3) f(q_2z_2/z_3)^{-1}\sigma_2(z_2/z_3)\sigma_2(z_3/z_2)
\prod_{i=2}^3\frac{d z_i}{2\pi i z_i}\,,\\
J_2=\int\!\!\int_{|z_2|=|z_3|=1}\Eb(z_2,z_3,q_2z_3) f(q_2z_3/z_2)^{-1}\sigma_2(z_2/z_3)\sigma_2(z_3/z_2)
\prod_{i=2}^3\frac{d z_i}{2\pi i z_i}\,.
\end{align*}
If we rename $z_3$ in $J_1$ to $q_2z_3$ (so that $q_2z_3$ is on the unit circle), then 
the two integrands become the same thanks to the identity
\begin{align}
f(x)^{-1}\sigma_2(q_2^{-1}x)=\sigma_2(x)\frac{(1-x)^2(1-q_2^{-1}x)^2}{(1-\mu x)(1-\mu^{-1}q_2^{-1}x)(1-q_1x)(1-q_3x)}\,.
\label{id1}
\end{align}
The integrand of $J_2$ has poles at (see Figure \ref{Fig1} below)
\begin{align*}
z_3=&\mu^mq_s^{-1}z_2\quad (s=1,3,\ m\ge0);\\
z_3=&\mu^{-m}q_sq_2^{-1}z_2\quad (s=1,3,\ m\ge1);\\
z_3=&q_sz_2\quad (s=1,3)\,.
\end{align*}
Among them the points $z_3=q_sz_2$
are inside the contour for $J_1$ (after renaming)
and outside that for $J_2$. 
However these poles are actually absent due to the zero condition \eqref{zero1}.  
Hence we have $J_1=J_2$.

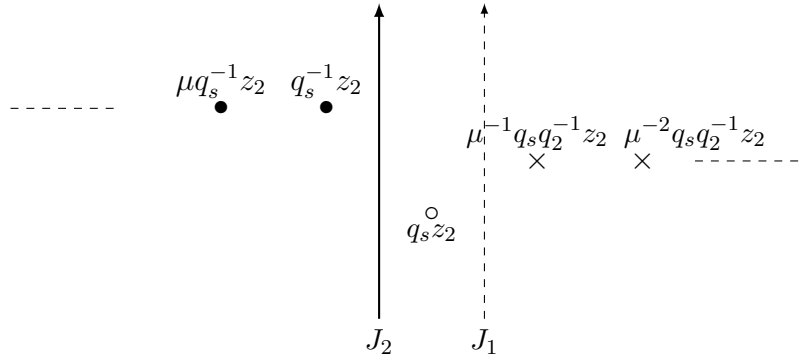
\begin{figure}[H]
\begin{center}
\begin{tikzpicture}[scale=0.7]
\draw[dashed] (-9,5)--(-7,5);
\node at (-5,5){$\bullet$};
\node at (-3,5){$\bullet$};
\node  [above] at (-5,5){$\mu q_s^{-1} z_2$};
\node [above] at (-3,5){$q_s^{-1} z_2$}; 

\node at (1,4){$\times$};
\node at (3,4){$\times$};
\draw[dashed] (4,4)--(6,4);
\node  [above] at  (1,4) {\small $\mu ^{-1} q_sq_2^{-1} z_2$};
\node  [above] at  (4,4) {\small $\mu ^{-2} q_sq_2^{-1} z_2$};

\node at (-1,3){$\circ$};
\node  [below] at  (-1,3) {\small $q_s z_2$};

\node [below] at (-2,1) {\small $J_2$};
\draw [-latex,thick] (-2,1)--(-2,7);

\node [below] at (0,1) {\small $J_1$};
\draw [-latex,dashed] (0,1)--(0,7);

\end{tikzpicture}
\caption{Integration contours on the $z_3$-plane ($s=1,3$)
for $J_1$, $J_2$.
\label{Fig1}}
\end{center}
\end{figure}

Next let us consider the residues at $z_1=(\mu q_2)^{\pm 1}z_3$. Similarly as above we obtain
\begin{align*}
J'_1=\int\!\!\int_{|z_2|=|z_3|=1}\Eb(z_2,\mu^{-1}q_2^{-1}z_3,z_3) f(\mu q_2 z_2/z_3)^{-1}\sigma_2(z_2/z_3)\sigma_2(z_3/z_2)
\prod_{i=2}^3\frac{d z_i}{2\pi i z_i}\,,\\
J'_2=\int\!\!\int_{|z_2|=|z_3|=1}\Eb(z_2,z_3,\mu q_2 z_3) f(\mu q_2 z_3/z_2)^{-1}\sigma_2(z_2/z_3)\sigma_2(z_3/z_2)
\prod_{i=2}^3\frac{d z_i}{2\pi i z_i}\,.
\end{align*}
We have another identity 
\begin{align}
f(x)^{-1}\sigma_2(\mu^{-1}q_2^{-1}x)=\sigma_2(x)\frac{(1-x)^2(1-\mu^{-1}q_2^{-1}x)^2}
{(1-\mu^{-1}q_1 x)(1-\mu^{-1}q_3 x)(1-q_1x)(1-q_3x)}\,.
\label{id2}
\end{align}
After renaming $z_3\to \mu q_2 z_3$ in $J_1'$, the two integrands become the same.
The integrand of $J_2'$ has poles at (see Figure \ref{Fig2})
\begin{align*}
z_3=&\mu^mq_s^{-1}z_2\quad (m\ge0, s=1,3);\\
z_3=&\mu^{-m}q_2^{-1}q_sz_2\quad (m\ge 1, s=1,3);\\
z_3=&q_s z_2,\ \mu^{-1}q_sz_2\quad (s=1,3)\,.
\end{align*}
The last ones
$z_3=q_sz_2,\ \mu^{-1}q_sz_2$ ($s=1,3$)
are inside the contour for $J_1'$, while they are outside for $J_2'$. 
Using the zero conditions \eqref{zero1},\eqref{zero2}, 
we conclude that $J_1'=J_2'$.\qed

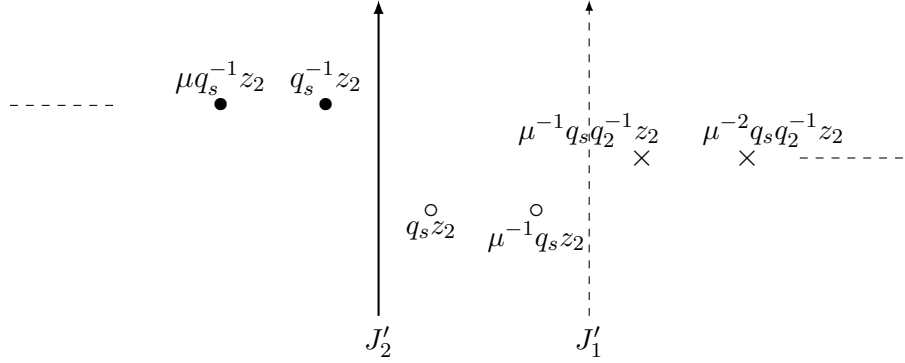
\begin{figure}[H]
\begin{center}
\begin{tikzpicture}[scale=0.7]
\draw[dashed] (-9,5)--(-7,5);
\node at (-5,5){$\bullet$};
\node at (-3,5){$\bullet$};
\node  [above] at (-5,5){$\mu q_s^{-1} z_2$};
\node [above] at (-3,5){$q_s^{-1} z_2$}; 

\node at (3,4){$\times$};
\node at (5,4){$\times$};
\draw[dashed] (6,4)--(8,4);
\node  [above] at  (2,4) {\small $\mu ^{-1} q_sq_2^{-1} z_2$};
\node  [above] at  (5.5,4) {\small $\mu ^{-2} q_sq_2^{-1} z_2$};

\node at (-1,3){$\circ$};
\node at (1,3){$\circ$};
\node  [below] at  (1,3) {\small $\mu^{-1}q_s z_2$};
\node  [below] at  (-1,3) {\small $q_s z_2$};

\node [below] at (-2,1) {\small $J'_2$};
\draw [-latex,thick] (-2,1)--(-2,7);

\node [below] at (2,1) {\small $J'_1$};
\draw [-latex,dashed] (2,1)--(2,7);

\end{tikzpicture}
\caption{Integration contours on the $z_3$-plane ($s=1,3$) for $J_1'$, $J_2'$.
\label{Fig2}}
\end{center}
\end{figure}

\medskip

\subsection{The general case.} 
We consider the general case $[\bI_m,\bI_n]=0$. 
Call the integration variables $z_1,\ldots,z_m$ for $\bI_m$ and $w_1,\ldots,w_n$ for $\bI_n$. 
We proceed in the same way, rewriting $\bI_m\bI_n$ and $\bI_n\bI_m$ as integrals over the unit circle and
picking residues with respect to some groups of variables.

First consider $\bI_m\bI_n$. 
In view of symmetry and zeros on the diagonal, 
it is sufficient to consider residues from 
\begin{align*}
&z_i=q_2^{-1}w_i\quad (1\le i\le k),\\
&z_i=\mu^{-1} q_2^{-1}w_i\quad (k+1\le i\le k+l).
\end{align*}
The result is (we write only the integrand)
\begin{align*}
&J_1=\Eb(\{q_2^{-1}w_i,w_i\}_{i=1}^k,\{\mu^{-1}q_2^{-1}w_i,w_i\}_{i=k+1}^{k+l},
\{z_j\}_{j=k+l+1}^m,\{w_j\}_{j=k+l+1}^n)
\times F_1G_1H_1\,,
\end{align*}
with
\begin{align*}
F_1&=\prod_{1\le i\neq j\le k}f(q_2w_j/w_i)^{-1}\sigma_2(w_j/w_i)^2
\prod_{k+1\le i\neq j\le k+l}f(\mu q_2w_j/w_i)^{-1}\sigma_2(w_j/w_i)^2\\
&\times
\prod_{1\le i\le k\atop k+1\le j\le k+l}f(q_2w_j/w_i)^{-1}\sigma_2(\mu^{-1}w_j/w_i)\sigma_2(w_i/w_j)
\prod_{k+1\le i\le k+l\atop 1\le j\le k}f(\mu q_2w_j/w_i)^{-1}\sigma_2(\mu w_j/w_i)\sigma_2(w_i/w_j)
\,,
\\
G_1
&=\hspace{-7pt}\prod_{j=k+l+1}^{m}\hspace{-3pt}
\Bigl(\prod_{1\le i\le k}f(w_i/z_j)^{-1}\sigma_2(q_2^{-1}w_i/z_j)\sigma_2(q_2z_j/w_i)
\hspace{-8pt}\prod_{k+1\le i\le k+l}\hspace{-8pt}f(w_i/z_j)^{-1}\sigma_2(\mu^{-1}q_2^{-1}w_i/z_j)\sigma_2(\mu q_2z_j/w_i)\Bigr)\\
&\times\prod_{j=k+l+1}^{n}
\Bigl(
\prod_{1\le i\le k}f(q_2w_j/w_i)^{-1}\sigma_2(w_j/w_i)\sigma_2(w_i/w_j)
\prod_{k+1\le i\le k+l}f(\mu q_2w_j/w_i)^{-1}\sigma_2(w_j/w_i)\sigma_2(w_i/w_j)\Bigr)
\\
H_1&=\prod_{k+l+1\le i\le m\atop k+l+1\le j\le n}f(w_j/z_i)^{-1}
\prod_{k+l+1\le i\neq j\le m}\sigma_2(z_j/z_i)
\prod_{k+l+1\le i\neq j\le n}\sigma_2(w_j/w_i)\,.
\end{align*}

For $\bI_n\bI_m$ we work similarly, picking poles at 
\begin{align*}
&z_i=q_2w_i\quad (1\le i\le k),\\
&z_i=\mu q_2 w_i\quad (k+1\le i\le k+l),
\end{align*}
ending up with 

\begin{align*}
&J_2=\Eb(\{q_2 w_i,w_i\}_{i=1}^k,\{\mu q_2 w_i,w_i\}_{i=k+1}^{k+l},
\{z_j\}_{j=k+l+1}^m,\{w_j\}_{j=k+l+1}^n)
\times F_2G_2H_2\,,
\end{align*}
where
\begin{align*}
F_2&=\prod_{1\le i\neq j\le k}f(q_2w_j/w_i)^{-1}\sigma_2(w_j/w_i)^2
\prod_{k+1\le i\neq j\le k+l}f(\mu q_2w_j/w_i)^{-1}\sigma_2(w_j/w_i)^2\\
&\times
\prod_{1\le i\le k\atop k+1\le j\le k+l}f(q_2w_i/w_j)^{-1}\sigma_2(w_i/w_j)\sigma_2(\mu w_j/w_i)
\prod_{k+1\le i\le k+l\atop 1\le j\le k}f(\mu q_2w_j/w_i)^{-1}\sigma_2(w_j/w_i)\sigma_2(\mu^{-1} w_i/w_j)
\,,
\\
G_2
&=\hspace{-8pt}\prod_{i=k+l+1}^{m}\hspace{-3pt}
\Bigl(\prod_{1\le j\le k}f(z_i/w_j)^{-1}\sigma_2(q_2w_j/z_i)\sigma_2(q_2^{-1}z_i/w_j)\hspace{-8pt}
\prod_{k+1\le i\le k+l}\hspace{-8pt}f(z_i/w_j)^{-1}\sigma_2(\mu q_2w_j/z_i)\sigma_2(\mu^{-1}q_2^{-1}z_i/w_j)\Bigr)\\
&\times\prod_{j=k+l+1}^{n}
\Bigl(
\prod_{1\le i\le k}f(q_2w_i/w_j)^{-1}\sigma_2(w_i/w_j)\sigma_2(w_j/w_i)
\prod_{k+1\le i\le k+l}f(\mu q_2w_i/w_j)^{-1}\sigma_2(w_i/w_j)\sigma_2(w_j/w_i)\Bigr)\,,
\\
H_2&=\prod_{k+l+1\le i\le m\atop k+l+1\le j\le n}f(z_i/w_j)^{-1}
\prod_{k+l+1\le i\neq j\le m}\sigma_2(z_j/z_i)
\prod_{k+l+1\le i\neq j\le n}\sigma_2(w_j/w_i)\,.
\end{align*}
Using the identities \eqref{id1},\eqref{id2} one can check that
\begin{align*}
J_1\bigr|_{w_i\to q_2w_i(1\le i\le k),\ w_j\to \mu q_2 w_j\ (k+1\le j\le k+l)}=J_2\,.
\end{align*}
The contours can be chosen to be the unit circle 
by the same mechanism observed above.
It remains to show that under symmetrization with respect to $\{z_i\}_{i=k+l+1}^m$
and  $\{w_j\}_{j=k+l+1}^n$ we have
\begin{align*}
\mathop{\mathrm{Sym}}H_1=\mathop{\mathrm{Sym}}H_2\,.
\end{align*}
This reduces to identity \eqref{id-h}. The proof of Theorem \ref{mu q_2} is now complete.\qed

\section{$K$ matrices}\label{sec:Kmat}
It is well known that the Hopf algebra $\E$ is equipped with the universal R matrix, which gives rise to
an intertwiner of $\E$ modules
\begin{align}
\Rc(u_1/u_2):\F_2(u_1)\otimes\F_2(u_2)\to \F_2(u_2)\otimes\F_2(u_1)\,. \label{Rmat}
\end{align}
Since $\Rc(u_1/u_2)$ commutes with the diagonal Heisenberg $\Delta h_r$, it is written in 
terms of a single boson $\{\ssa^\ssA_r\}$ entering the root current 
$A(z)=\,\,:\Lambda_1(s_2z)\Lambda_2(s_2z)^{-1}:$. Exhibiting the dependence on parameters we shall write
\begin{align*}
\Rc(u_1/u_2)=\Rc(u_1/u_2;q_1,q_2,q_3,\{\ssa^\ssA_r\})\,.
\end{align*}

\begin{prop}\label{Kmat}
Consider a $\cK$ module $\F_2(u)\otimes \F^X_c$ ($X=\ssB,\ssCD$), 
and let $A(z)=\,\,:e^{\sum_{r\in\Z}\ssa_rz^{-r}}:$
be the root current associated with $E(z)$.
Then there exists an intertwiner of $\cK$ modules
\begin{align*}
&\ssK(u):\F_2(u)\otimes\F^{X}_c\to \F_2(u^{-1})\otimes\F^{X}_c\,
\end{align*}
in the following cases.
\begin{align}
\text{type $\ssB$}&\quad(X=\ssB, c=3):&& 
\ssK(u)=\Rc(u;q_1,q_2q_3^{1/2},q_3^{1/2},\{\ssa_r\})\,, \label{KB}
\\
\text{type $\ssC$}&\quad(X=\ssCD, c=3):&& 
\ssK(u)=\Rc(u^2;q_1,q_2q_3^{-1},q_3^2,\{\ssa_r\})\,,\label{KC}
\\
\text{type $\ssD$}&\quad (X=\ssCD, c=2):&& 
\ssK(u)=(-1)^{\ssN}\,.\label{KD}
\end{align}
In the last line, 
$\ssN$ denotes the number operator $\sum_{r>0}\nu_r^{-1}\ssa_{-r}\ssa_r$,  $\nu_r=[\ssa_r,\ssa_{-r}]$.
Note that $\ssK(u)$ is independent of $u$ in this case.
\end{prop}
\begin{proof}
First note that, by extracting the diagonal Heisenberg, the intertwining relation \eqref{Rmat} for type $\ssA$
reduces to 
\begin{align*}
\Rc(u_1/u_2)\bigl(u_1\Lambda^\ssA_+(z)+u_2\Lambda^\ssA_-(z)\bigr)
=\bigl(u_2\Lambda^\ssA_+(z)+u_1\Lambda^\ssA_-(z)\bigr)\Rc(u_1/u_2)\,,
\end{align*}
where $\Lambda^\ssA_\pm(z)$ are vertex operators in $\{\ssa^\ssA_r\}_{r\neq0}$ such that
\begin{align}
&\Lambda^\ssA_-(z)=\,\,:\Lambda^\ssA_+(q_2z)^{-1}:\,,
\label{caseA1}\\
&\mathcal{C}\bigl(\Lambda^\ssA_+(z),\Lambda^\ssA_+(w)\bigr)=-\frac{(1-q_3)(1-q_1)}{1+q_2}q_2\,.
\label{caseA2}
\end{align}

For type $\ssC$ and type $\ssD$ we proceed the same way as in type $\ssA$ to obtain the reduced intertwining relation
\begin{align*}
&\ssK(u)\bigl(u \Lambda_+(z)+u^{-1}\Lambda_-(z)\bigr)=\bigl(u^{-1}\Lambda_+(z)+u\Lambda_-(z)\bigr)\ssK(u)\,,
\end{align*}
where $\Lambda_\pm(z)$ are vertex operators  in $\{\ssa_r\}_{r\neq0}$ 
such that 
\begin{align}
&\Lambda_-(z)=\,\,:\Lambda_+(q_2q_c^{-1}z)^{-1}:\,.
\label{caseBCD1}
\end{align}
For type $\ssC$ ($c=3$), we have further
\begin{align}
\mathcal{C}\bigl(\Lambda_+(z),\Lambda_+(w)\bigr)=-\frac{(1-q^2_3)(1-q_1)}{1+q_2q_3^{-1}}q_2q_3^{-1}\,.
\label{caseBCD2}
\end{align}
Comparing \eqref{caseBCD1}, \eqref{caseBCD2} with \eqref{caseA1}, \eqref{caseA2}, we obtain \eqref{KC}.

For type $\ssD$ ($c=2$),  \eqref{caseBCD1} becomes $\Lambda_-(z)=\,\,:\Lambda_+(z)^{-1}:$, so that the intertwining relation 
reduces further to $\ssK(u)\ssa_r=-\ssa_r\ssK(u)$ for all $r\neq0$.
This leads to the solution \eqref{KD}.

For type $\ssB$, the reduced intertwining relation involves three terms
\begin{align*}
\ssK(u)\bigl(u\Lambda_{++}(z)+k\Lambda_0(z)+u^{-1}\Lambda_{--}(z)\bigr)
=\bigl(u^{-1}\Lambda_{++}(z)+k\Lambda_0(z)+u\Lambda_{--}(z)\bigr)\ssK(u)\,,
\end{align*}
which corresponds to the $qq$ character of the three dimensional representation of $U_q\widehat{\mathfrak{sl}}_2$.
One can reduce it further to intertwining relation for the two dimensional one 
\begin{align*}
&\ssK(u)\bigl(u^{1/2} \Lambda_+(z)+u^{-1/2}\Lambda_-(z)\bigr)=\bigl(u^{-1/2}\Lambda_+(z)+u^{1/2}\Lambda_-(z)\bigr)\ssK(u)\,
\end{align*}
by introducing $\Lambda_\pm(z)$ such that $\Lambda_{\pm\pm}(z)=\,\,:\Lambda_\pm(s_3^{1/2}z)\Lambda_\pm(s_3^{-1/2}z):$, 
$\Lambda_0(z)=\,\,:\Lambda_+(s_3^{1/2}z)\Lambda_-(s_3^{-1/2}z):$. 
The rest is similar to type $\ssC$.
\end{proof}
\medskip

On tensor products \eqref{B-mod}, \eqref{C-mod}, we write the intertwiners as
$\Rc_{i,i+1}(u_i/u_{i+1})$ ($i=1,\ldots,\ell-1$) indicating the tensor components where they act non-trivially. 
For the intertwiners involving boundary Fock modules we write as $\ssK_\ell(u_\ell)$. 
The standard argument (with $\ell=2$) leads to the boundary Yang-Baxter equation 
\begin{align}
&\Rc_{1,2}(u_1/u_2)\ssK_2(u_1)\Rc_{1,2}(u_1u_2)\ssK_2(u_2)
=\ssK_2(u_2)\Rc_{1,2}(u_1u_2)\ssK_2(u_1)\Rc_{1,2}(u_1/u_2)\,.
\label{RKRK}
\end{align}

As noted above, the K matrix $\ssK_\ell$ of type $\ssD$ is independent of $u_\ell$ and 
satisfies $\ssK^2_\ell=1$. 
Comparing the intertwining relation with the definition of the currents $A_{\ell-1}(z)$ \eqref{A-i}
and $A_\ell(z)$ \eqref{D-ell}, we see that 
\begin{align*}
\ssK_\ell A_{\ell-1}(z)=A_\ell(z) \ssK_\ell\,.
\end{align*}

Though the zeroth node of the Dynkin diagram is not associated with boundary Fock modules, one can consider 
K matrices depending only on $A_0(z)$ and satisfying the intertwining relations, for example 
for type $\ssC$ 
\begin{align*}
\ssK_1(u_1)\bigl(\bs\Lambda_{\bar 1}(z)+\bs\Lambda_1(\mu z)\bigr)=
\bigl(u_1^2\bs\Lambda_{\bar 1}(z)+u_1^{-2}\bs\Lambda_1(\mu z)\bigr)\ssK_1(u_1)\,.
\end{align*}
\bigskip

\section{The library of Cartan matrices}\label{library}

\subsection{Conventions}\label{sec library conventions}
The matrix of contractions $\hat B$ is the deformed version of the symmetrized Cartan matrix. We will give a list of explicit deformed Cartan matrices $\hat C$ of low rank which can be used to write all others as explained in Appendix \ref{sec general}.

The deformed symmetrized Cartan matrix $\hat B$ and the deformed Cartan matrix $\hat C$ are related by a diagonal matrix $\hat D$, namely $\hat B=\hat D\hat C$, where the diagonal entries $d_i$ of $\hat B$ and the diagonal entries of $\hat C$ are given as follows.

The nodes of type $\ssA$ (corresponding to $\F_{c_i}\otimes \F_{c_{i+1}}$).
\begin{align*}
\ssA,\ c_i=c_{i+1}:&\quad d_i=-\frac{t_1t_2t_3}{t_{c_i}},\quad &&C_{i,i}=s_{c_i}+s_{c_i}^{-1}.
&&
\begin{tikzpicture}[baseline=(origin.base)] 
\dynkin[root radius=.1cm, edge length=1.3cm, labels={c_i}]{A}{o}
\end{tikzpicture}
\\
\ssA,\ c_i\neq c_{i+1}:&\quad d_i=t_b,\quad &&C_{i,i}=t_b\quad (b\neq c_i,c_{i+1}).
&&
\begin{tikzpicture}[baseline=(origin.base)] 
\dynkin[root radius=.1cm, edge length=1.3cm, labels={b}]{A}{t}
\end{tikzpicture}
\end{align*}

\bigskip

    The nodes of type $\ssB$  (corresponding to $\F_{c_{\ell}}\otimes \F^{\ssB}_{c_{\ell+1}}$).
\begin{align*}
\ssB, \ c_\ell=c_{\ell+1}:&\quad d_\ell=-\frac{t_1t_2t_3}{t_{c_{\ell+1}}(s^{1/2}_{c_{\ell+1}}+s^{-1/2}_{c_{\ell+1}})}, \quad &&C_{\ell,\ell}=s_{c_{\ell+1}}^{3/2}+s_{c_{\ell+1}}^{-3/2}.
&&
\begin{tikzpicture}[baseline=(origin.base)] 
\dynkin[root radius=.1cm, edge length=1.3cm, labels={c_\ell}]{A}{*}
\end{tikzpicture}
\\
\ssB, \ c_\ell\neq c_{\ell+1}:&\quad d_\ell=-t_b(s^{1/2}_{c_{\ell+1}}-s^{-1/2}_{c_{\ell+1}}),\quad 
&&
C_{\ell,\ell}=s_{c_{\ell}}^{1/2}s^{-1/2}_{b}+s^{-1/2}_{c_{\ell}}s^{1/2}_{b}
\quad (b\neq c_\ell,c_{\ell+1}).
&&
\begin{tikzpicture}[baseline=(origin.base)] 
\dynkin[root radius=.1cm, edge length=1.3cm, labels={b}]{A}{o}
\end{tikzpicture}
\end{align*}

The nodes of type $\ssC$  (corresponding to $\F_{c_\ell}\otimes \F^{\ssC\ssD}_{c_{\ell+1}}$, $c_{\ell+1}\neq c_{\ell}$).

\begin{align*}
\ssC, \ c_\ell\neq c_{\ell+1}:&\quad d_\ell=-t_b(s^2_{c_{\ell+1}}-s^{-2}_{c_{\ell+1}}),\qquad 
C_{\ell,\ell}=s_{c_{\ell+1}}s^{-1}_{c_{\ell}}+s^{-1}_{c_{\ell+1}l}s_{c_{\ell}}\quad (b\neq c_\ell,c_{\ell+1}).
&&
\begin{tikzpicture}[baseline=(origin.base)] 
\dynkin[root radius=.1cm, edge length=1.3cm, labels={b}]{A}{o}
\end{tikzpicture}
\end{align*}
\bigskip

The nodes of type $\ssD$ 
(corresponding to $\F_{c_{\ell-1}}\otimes\F_{c_\ell}\otimes \F^{\ssC\ssD}_{c_{\ell+1}}$,  $c_{\ell+1}=c_{\ell}$).   
\begin{align*}
\ssD, \ c_{\ell+1}=c_{\ell-1}:&\quad d_\ell=-\frac{t_1t_2t_3}{t_{c_{\ell+1}}},\quad C_{\ell,\ell}=s_{c_{\ell+1}}+s_{c_{\ell+1}}^{-1} &&
\begin{tikzpicture}[baseline=(origin.base)] 
\dynkin[root radius=.1cm, edge length=1.3cm, labels={c_\ell}]{A}{o}
\end{tikzpicture}  \\
\ssD, \ c_{\ell+1}\neq c_{\ell-1}:&\quad d_{\ell}=t_b,
\quad C_{\ell,\ell}=t_b \quad (b\neq c_{\ell+1},c_{\ell-1})
&&
\begin{tikzpicture}[baseline=(origin.base)] 
\dynkin[root radius=.1cm, edge length=1.3cm, labels={b}]{A}{t}
\end{tikzpicture}  
\\
\end{align*}
\bigskip

Affine nodes (corresponding to dressing in non $\ssA$ types) are the same as $\ssB, \ssC,\ssD$ nodes after the change
$$c_i\to c_{\ell+1-i}, \qquad d_i\to d_{\ell-i},\qquad C_{i,i}\to C_{\ell-i,\ell-i}.$$

\bigskip

The deformed affine Cartan matrices we obtain have ``local" form: $C_{i,j}=0$ if $|i-j|$ is sufficiently large. Moreover, the non-zero terms stabilize with $\ell$ increased. We give here a number of explicit deformed Cartan matrices of low rank which can be used to write all others as explained in Appendix \ref{sec general}.

We use the following notation. We write $X(c_0;c_1,\dots,c_{\ell-1};c_{\ell+1})Y$ for the choice of the module and the dressing and give the $(\ell+1)\times(\ell+1)$ matrix $\hat C$. Here $Y$ can be $\ssB,\ssC$, or $\ssD$ depending on which boundary module we consider $\F_{c_{\ell+1}}^{\ssB}$ or $\F_{c_{\ell+1}}^{\ssC\ssD}$. As always, in the latter case we choose $\ssD$ if $c_{\ell+1}=c_\ell$ and $\ssC$ otherwise. Similarly $X$ can be $\ssB,\ssC$, or $\ssD$ depending on which dressing we choose. Namely $C^2/\mu=q_{c_0}^{-1/2}$ corresponds to $X=\ssB$ while  $C^2/\mu=q_{c_0}$ corresponds to  $X=\ssC$, if $c_0\neq c_1$ and to $X=\ssD$ if $c_0=c_1$.

We skip writing $X$ and $c_0$ in finite types. In addition we also skip $c_{\ell+1}$ and $Y$ in $\ssA$ type.

\subsection{Finite types}\label{sec finite}
\hspace{20pt}
\newline

\noindent{\bf Type $\ssA$.}
\begin{align*}
(2,2,2) &&\left(\begin{matrix}  s_2+s_2^{-1} & -1\\
-1&  s_2+s_2^{-1}
\end{matrix}\right) \\ 
(1,2,3)& & \left(\begin{matrix} t_3 & t_1 \\
t_3 &t_1
\end{matrix}\right)\qquad \\
(1,2,1)& &\left(\begin{matrix} t_3 & t_1 \\
t_1 & t_3
\end{matrix}\right)\qquad \\
(2,2,1)& &\qquad \left(\begin{matrix} s_2+s_2^{-1} & -1 \\
t_1 &t_3
\end{matrix}\right)\ \\ 
(1,2,2)&& \left(\begin{matrix} t_3 & t_1 \\
-1 &s_2+s_2^{-1}
\end{matrix}\right)  \ 
\end{align*}
\medskip

\noindent{\bf Type $\ssB$.}
\begin{align*}
(2,2;2)\ssB &&
\begin{pmatrix}
s_2+s_2^{-1}& -1\\
-s_2^{1/2}-s_2^{-1/2}& s_2^{3/2}+s_2^{-3/2}\\
\end{pmatrix}\quad 
\\
(2,2;3)\ssB &&
\begin{pmatrix}
s_2+s_2^{-1}& -1\\
-s_3^{1/2}-s_3^{-1/2}& 
s_1^{1/2}s_2^{-1/2}+s_1^{-1/2}s_2^{1/2} \\
\end{pmatrix}
\\
(1,2;2)\ssB&&
\begin{pmatrix}
t_3& t_1\\
-s_2^{1/2}-s_2^{-1/2}& s_2^{3/2}+s_2^{-3/2}\\
\end{pmatrix}\quad
\\
(1,2;1)\ssB &&
\begin{pmatrix}
t_3& t_1\\
-s_1^{1/2}-s_1^{-1/2}& 
s_3^{1/2}s_2^{-1/2}+s_3^{-1/2}s_2^{1/2} \\
\end{pmatrix}
\\
(1,2;3)\ssB&&
\begin{pmatrix}
t_3& t_1\\
-s_3^{1/2}-s_3^{-1/2}&
s_1^{1/2}s_2^{-1/2}+s_1^{-1/2}s_2^{1/2} \\
\end{pmatrix}
\end{align*}
\bigskip

\noindent{\bf Type $\ssC$.}\quad 

\begin{align*}
(2,2;3)\ssC &&
\begin{pmatrix}
s_2+s_2^{-1}& -s_3-s_3^{-1}\\
-1& s_2s_3^{-1}+s_2^{-1}s_3
\\
\end{pmatrix}
\\
(1,2;1)\ssC&&
\begin{pmatrix}
t_3& s^2_1-s_1^{-2}\\
-1&s_2s_1^{-1}+s_2^{-1}s_1
\\
\end{pmatrix}\ \ 
\\
(1,2;3)\ssC&&
\begin{pmatrix}
t_3& -t_1(s_3+s_3^{-1})\\
-1&s_2s_3^{-1}+s_2^{-1}s_3
\\
\end{pmatrix}\ \ 
\end{align*}
\bigskip

\noindent{\bf Type $\ssD$.}\quad 

\begin{align*}
(2,2,2;2)\ssD&&
\begin{pmatrix}
s_2+s_2^{-1}& -1 &-1\\
-1& s_2+s_2^{-1}&0\\
-1& 0&s_2+s_2^{-1}\\
\end{pmatrix} \qquad\ \ 
\\
(1,2,2;2)\ssD&&
\begin{pmatrix}
t_3& t_1 &t_1\\
-1& s_2+s_2^{-1}&0\\
-1& 0&s_2+s_2^{-1}\\
\end{pmatrix} \qquad \ \ \ \ 
\\
(2,1,2;2)\ssD&&
\begin{pmatrix}
t_3&t_2&t_2\\
t_2& t_3&s_1s_2^{-1}-s_1^{-1}s_2\\
t_2&s_1s_2^{-1}-s_1^{-1}s_2&t_3\\
\end{pmatrix}\ \ \ \ 
\\
(1,1,2;2)\ssD &&
\begin{pmatrix}
s_1+s_1^{-1}& -1 & -1\\
t_2& t_3&s_1s_2^{-1}-s_1^{-1}s_2\\
t_2&s_1s_2^{-1}-s_1^{-1}s_2& t_3\\
\end{pmatrix}
\\
(3,1,2;2)\ssD&&
\begin{pmatrix}
t_2& t_3& t_3\\
t_2&t_3 & s_1s_2^{-1}-s_1^{-1}s_2\\
t_2& s_1s_2^{-1}-s_1^{-1}s_2&t_3\\
\end{pmatrix}\ \ \ \ 
\end{align*}

\subsection{Affine types}\label{sec affine}
\hspace{20pt}
\newline

\noindent{\it \bf $2\times 2$ cases.}

\medskip 

{\bf $\ssB-\ssB$ types.}

\begin{align*}
\ssB(2;2;2)\ssB &&
&&
\begin{pmatrix}
s_2^{3/2}+s_2^{-3/2}& -s_2^{1/2}-s_2^{-1/2}\\
-s_2^{1/2}-s_2^{-1/2}& s_2^{3/2}+s_2^{-3/2}\\
\end{pmatrix}\qquad \ \ 
\\
\ssB(2;2;1)\ssB&&
&&
\begin{pmatrix}
s_2^{3/2}+s_2^{-3/2}& -s_2^{1/2}-s_2^{-1/2}\\
-s^{1/2}_1-s_1^{-1/2}& s_2^{1/2}s_3^{-1/2}+s_2^{-1/2}s_3^{1/2}\\
\end{pmatrix}\quad \ 
\\
\ssB(1;2;1)\ssB&&
&&
\begin{pmatrix}
s_2^{1/2}s_3^{-1/2}+s_2^{-1/2}s_3^{1/2}& -s^{1/2}_1-s_1^{-1/2}\\
-s^{1/2}_1-s_1^{-1/2}& s_2^{1/2}s_3^{-1/2}+s_2^{-1/2}s_3^{1/2}\\
\end{pmatrix}
\\
\ssB(1;2;3)\ssB&&
&&
\begin{pmatrix}
s_2^{1/2}s_3^{-1/2}+s_2^{-1/2}s_3^{1/2}& -s^{1/2}_1-s_1^{-1/2}\\
-s^{1/2}_3-s_3^{-1/2}& 
s_2^{1/2}s_1^{-1/2}+s_2^{-1/2}s_1^{1/2}\\
\end{pmatrix}
\end{align*}
\bigskip
{\bf $\ssB-\ssC$ types.}
\begin{align*}
\ssB(2;2;1)\ssC &&
&&
\begin{pmatrix}
s_2^{3/2}+s_2^{-3/2}& -(s_1+s_1^{-1})(s_2^{1/2}+s_2^{1/2})\\
-1& s_1s^{-1}_2+s^{-1}_1s_2\\
\end{pmatrix}\qquad
\\
\ssB(1;2;1)\ssC&&
&&
\begin{pmatrix}
s_2^{1/2}s_3^{-1/2}+s_2^{-1/2}s_3^{1/2}& -(s_1+s_1^{-1})(s^{1/2}_1+s_1^{-1/2})\\
-1& s_1s^{-1}_2+s^{-1}_1s_2\\
\end{pmatrix}
\\
\ssB(3;2;1)\ssC&&
&&
\begin{pmatrix}
s_1^{1/2}s_2^{-1/2}+s_1^{-1/2}s_2^{1/2}& -(s_1+s_1^{-1})(s^{1/2}_3+s_3^{-1/2})\\
-1& s_1s^{-1}_2+s^{-1}_1s_2\\
\end{pmatrix}
\end{align*}
{\bf $\ssC-\ssC$ types.}
\begin{align*}
\ssC(1;2;1)\ssC&&
&&
\begin{pmatrix}
s_1s^{-1}_2+s_1^{-1}s_2& -s_1-s_1^{-1}\\
-s_1-s_1^{-1}& s_1s^{-1}_2+s_1s_2^{-1}\\
\end{pmatrix}\qquad
\\
\ssC(1;2;3)\ssC&&
&&
\begin{pmatrix}
s_1s^{-1}_2+s_1^{-1}s_2& -s_3-s_3^{-1}\\
-s_1-s_1^{-1}& s_3s^{-1}_2+s_3s_2^{-1}\\
\end{pmatrix}\qquad
\end{align*}

\noindent{\it \bf $3\times 3$ cases.}

\medskip

{\bf $\ssB-\ssD$ types.}
\begin{align*}
\ssB(2;2,2;2)\ssD&&\quad 
\begin{pmatrix}
s_2^{3/2}+s_2^{-3/2}& -s_2^{1/2}-s_2^{-1/2}  &  -s_2^{1/2}-s_2^{-1/2} \\
-1 & s_2+s_2^{-1}& 0\\
-1 & 0  &  s_2+s_2^{-1}\\
\end{pmatrix}\qquad 
\\
\ssB(1;2,2;2)\ssD &&\quad 
\begin{pmatrix}
s_2^{1/2}s_3^{-1/2}+s_2^{-1/2}s_3^{1/2}& -s_1^{1/2}-s_1^{-1/2}  &  -s_1^{1/2}-s_1^{-1/2} \\
-1 & s_2+s_2^{-1}& 0\\
-1 & 0  &  s_2+s_2^{-1}\\
\end{pmatrix}
\\
\ssB(2;1,2;2)\ssD &&\quad 
\begin{pmatrix}
s_1^{1/2}s_3^{-1/2}+s_1^{-1/2}s_3^{1/2}& -s_2^{1/2}-s_2^{-1/2}  &  -s_2^{1/2}-s_2^{-1/2}\\
t_2 & t_3 & s_1s_2^{-1}-s_1^{-1}s_2\\
t_2 & s_1s_2^{-1}-s_1^{-1}s_2  &  t_3\\
\end{pmatrix}
\\
\ssB(1;1,2;2)\ssD&&\quad 
\begin{pmatrix}
s_1^{3/2}+s_1^{-3/2}& -s_1^{1/2}-s_2^{-1/2}  &  -s_1^{1/2}-s_1^{-1/2} \\
t_2 & t_3 & s_1s_2^{-1}-s_1^{-1}s_2\\
t_2 &  s_1s_2^{-1}-s_1^{-1}s_2 &  t_3 \\
\end{pmatrix}\qquad 
\\
\ssB(1;3,2;2)\ssD&&\quad 
\begin{pmatrix}
s_2^{1/2}s_3^{-1/2}+s_2^{-1/2}s_3^{1/2}& -s_1^{1/2}-s_1^{-1/2}  &  -s_1^{1/2}-s_1^{-1/2} \\
t_2 & t_1 & s_3s_2^{-1}-s_3^{-1}s_2\\
t_2 &  s_3s_2^{-1}-s_3^{-1}s_2 &  t_1\\
\end{pmatrix}
\end{align*}

{\bf $\ssC-\ssD$ types.}

\begin{align*}
\ssC(1;2,2;2)\ssD&&\quad 
\begin{pmatrix}
s_1s_2^{-1}+s^{-1}_1s_2& -1  &  -1 \\
-s_1-s_1^{-1} & s_2+s_2^{-1}& 0\\
-s_1-s_1^{-1} & 0  &  s_2+s_2^{-1}\\
\end{pmatrix}\qquad \qquad
\\
\ssC(1;3,2;2)\ssD&&\quad 
\begin{pmatrix}
s_1s_3^{-1}+s^{-1}_1s_3& -1  &  -1 \\
(s_1+s_1^{-1})t_2 & t_1 &s_2^{-1}s_3-s_2s_3^{-1}\\
(s_1+s_1^{-1})t_2 & s_2^{-1}s_3-s_2s_3^{-1} &  t_1 \\
\end{pmatrix}
\\
\ssC(2;1,2;2)\ssD &&\quad 
\begin{pmatrix}
s_1s_2^{-1}+s^{-1}_1s_2& -1  &  -1 \\
s^2_2-s_2^{-2} & t_3 &s_2^{-1}s_1-s_2s_1^{-1}\\
s^2_2-s_2^{-2} & s_2^{-1}s_1-s_2s_1^{-1} &  t_3\\
\end{pmatrix}\quad
\end{align*}

{\bf $\ssD-\ssD$ types.}

\begin{align*}
\ssD(2;2,2;2)\ssD &&\quad 
\begin{pmatrix}
s_2+s_2^{-1}& 0  &  -s_2-s_2^{-1} \\
0        &s_2+s_2^{-1}    &0\\
-s_2-s_2^{-1} & 0  & s_2+s_2^{-1}\\
\end{pmatrix}\qquad\quad
\\
\ssD(1;1,2;2)\ssD &&\quad 
\begin{pmatrix}
t_3& s_1^{-1}s_2-s_1s_2^{-1}  & -t_3 \\
 s_1^{-1}s_2-s_1s_2^{-1}         & t_3   &-s_1^{-1}s_2+s_1s_2^{-1}\\
-t_3 & -s_1^{-1}s_2+s_1s_2^{-1}  & t_3\\
\end{pmatrix}
\end{align*}

\noindent{\it \bf $4\times 4$ cases}\quad 
\bigskip

{\bf $\ssD-\ssD$ types.}
\begin{align*}
\ssD(2;2,2,2;2)\ssD &&\quad 
\begin{pmatrix}
s_2+s^{-1}_2&  0  & -1  & -1 \\
0& s_2+s^{-1}_2 &-1  &-1     \\
-1 & -1 & s_2+s^{-1}_2  &0  \\
-1 & -1 & 0 & s_2+s^{-1}_2  \\
\end{pmatrix}\qquad\qquad\quad
\\
\ssD(2;2,1,2;2)\ssD&&\quad 
\begin{pmatrix}
t_3 & s_1s_2^{-1}-s_1^{-1}s_2  & t_2  & t_2 \\
s_1s_2^{-1}-s_1^{-1}s_2& t_3 & t_2  & t_2     \\
t_2 &  t_2 & t_3  &s_1s_2^{-1}-s_1^{-1}s_2  \\
t_2 &  t_2 & s_1s_2^{-1}-s_1^{-1}s_2 & t_3  \\
\end{pmatrix}
\\
\ssD(1;1,2,2;2)\ssD &&\quad 
\begin{pmatrix}
t_3&  s_1^{-1}s_2-s_1s_2^{-1}  & t_1  & t_1 \\
  s_1^{-1}s_2-s_1s_2^{-1}& t_3 & t_1  & t_1    \\
-1 & -1 & s_2+s^{-1}_2  &0  \\
-1 & -1 & 0 & s_2+s^{-1}_2  \\
\end{pmatrix}\qquad
\\
\ssD(1;1,3,2;2)\ssD &&\quad 
\begin{pmatrix}
t_2 &  s_1^{-1}s_3-s_1s_3^{-1}  & t_1  & t_1 \\
  s_1^{-1}s_3-s_1s_3^{-1}& t_2 & t_1  &t_1    \\
t_2  & t_2 & t_1  & s_2^{-1}s_3-s_2s_3^{-1}  \\
 t_2 &  t_2 & s_2^{-1}s_3-s_2s_3^{-1} & t_1  \\
\end{pmatrix}
\end{align*}

\bigskip

\subsection{General case}\label{sec general}
In  Appendix \ref{sec finite} we gave several deformed Cartan matrices of finite type, and in Appendix \ref{sec affine}  several deformed Cartan matrices of affine type. In fact Appendix \ref{sec finite} contains all
maximal submatrices of stable deformed affine  Cartan matrices whose upper right and bottom left corner elements are both non zero. Appendix \ref{sec affine}
contains all non-stable  deformed affine Cartan matrices appearing in our construction. 

It is important to keep in mind that we immediately obtain many more examples by permuting $s_1,s_2,s_3$. Another set of examples is obtained by reading the data from right to left. More precisely,
the matrices $\hat C$ corresponding to  $X(c_0;c_1,\dots,c_{\ell};c_{\ell+1})Y$ and to $Y(c_{\ell+1};c_{\ell},\dots,c_1;c_0)X$ are related by conjugation
by the matrix $(\delta_{i+j,\ell})_{i,j=0}^\ell$.

The matrices listed in Appendices \ref{sec finite} and \ref{sec affine} allow us to write a deformed  affine  Cartan matrix corresponding to arbitrary data 
$X(c_0;c_1,\dots,c_{\ell};c_{\ell+1})Y$, since for larger $\ell$ the  deformed Cartan matrices stabilize. One has to follow the following procedure.

First, search the affine examples, keeping in mind the symmetries. If the matrix is found, stop.
If the matrix is not in the list, conclude that $\hat C$ is stable, that is $C_{0,\ell}=C_{\ell,0}=0$. 

Second, find the listed matrix of the finite type and of the largest size corresponding to right most colors: 
$(c_{\ell-i},\dots,c_{\ell};c_{\ell+1})Y$ (it is $2\times 2$ for all cases except $Y=\ssD$ when it is $3\times 3$). That gives the right bottom submatrix of $\hat C$.

Third, repeat for left most colors, that is, find the listed matrix of the finite type and of the largest size corresponding to  $X(c_0;c_1,\dots,c_i)$. This gives the left upper submatrix.

The rest nonzero entries of the matrix are recovered from matrices of type $\ssA$.

The final result is the superposition of matrices which are linked via diagonal entries.

\medskip

For example, the matrix $\hat C$ corresponding to $\ssB(2;3,1;3)\ssC$ is a superposition of the last matrix of type $\ssB$ in our list and of the second matrix of type $\ssC$ with necessary symmetries (the $C_{1,1}$ entry $t_2$ is common):
$$
\begin{pmatrix}
s_2^{1/2}s_3+s_2^{-1/2}s_3^{-1}   & -s_2^{1/2}-s_2^{-1/2}     &     0       \\
t_1  &  t_2 & s_3^2-s_3^{-2} \\
0 & -1 & s_2s_3^{-1}+s_2^{-1}s_3 \\
\end{pmatrix}.
$$
Similarly, matrix $\hat C$ corresponding to $\ssB(2;3,2,1;3)\ssC$ is a superposition of three: the fourth on type $\ssB$ list, he second on type $\ssA$ list and the last on type $\ssC$ list. The first two share a common $C_{1,1}$ entry $t_1$ and the last two share a common $C_{2,2}$ entry $t_3$.
$$
\begin{pmatrix}
s_2^{1/2}s_3+s_2^{-1/2}s_3^{-1}  &  -s_3^{1/2}-s_3^{-1/2}  &     0  & 0     \\
 t_3  & t_1   &  t_3 &0 \\
0 & t_1 & t_3 & -t_2(s_3+s_3^{-1})\\
0 & 0& -1 & s_1s_3^{-1}+s_1^{-1}s_3\\ 
\end{pmatrix}.
$$

\bigskip

{\bf Acknowledgments.\ }
The research of BF is supported by 
the Russian Science Foundation grant project 16-11-10316. 
MJ is partially supported by 
JSPS KAKENHI Grant Number JP19K03549. 
EM is partially supported by a grant from the Simons Foundation  
\#353831. IV is supported in part by Young Russian Mathematics award.
BF, MJ, and EM thank Kyoto University for hospitality during their visit 
in summer 2019 when part of this work was completed.

\end{document}